\documentclass[11pt]{amsart}%
\usepackage{palatino, mathpazo}
\usepackage{amsfonts}
\usepackage{amsmath}
\usepackage{amssymb,latexsym,xcolor}
\usepackage{graphicx}
\usepackage{harpoon}
\usepackage[mathscr]{eucal}
\usepackage{amssymb}%
\usepackage[%backref=page,
linktocpage=true,colorlinks,citecolor=magenta,linkcolor=blue,urlcolor=magenta]{hyperref}
\usepackage[all]{xy}
\providecommand{\U}[1]{\protect \rule{.1in}{.1in}}

\newtheorem{theorem}{Theorem}[section]

\newtheorem{corollary}[theorem]{Corollary}
\newtheorem{definition}[theorem]{Definition}

\newtheorem{lemma}[theorem]{Lemma}
\newtheorem{proposition}[theorem]{Proposition}
\newtheorem{Theorem}{Theorem}

\theoremstyle{remark}
\newtheorem{remark}[theorem]{Remark}

\newtheorem{question}{Question}

\numberwithin{equation}{section}

\begin{document}
\title[Monodromy of generalized Lam\'e equations]{On monodromy and spectral geometry of generalized Lam\'e equations with four singularities, I: half periods
%\\
%Monodromy of generalized Lam\'e equations
%second order Fuchsian dierential equations  
%with apparent parameters, I.
}
%\author{Zhijie Chen}
%\address{Department of Mathematical Sciences, Yau Mathematical Sciences Center, Tsinghua University, Beijing, 100084, China }
%\email{zjchen2016@tsinghua.edu.cn}
\author{Erjuan Fu}
\address{Beijing Institute of Mathematical Sciences and Applications, Beijing, 101408, China}
\email{fej.2010@tsinghua.org.cn, ejfu@bimsa.cn}
\author{Chang-Shou Lin}
\address{Department of Mathematics, National Taiwan University, Taipei 10617, Taiwan }
\email{cslin@math.ntu.edu.tw}

\begin{abstract}
We consider the unitary monodromy problem of the following generalized Lam\'e equations with apparent parameters
\begin{equation*}
    y^{\prime\prime}(z)=\left(\frac{3}{4}\sum_{k=0}^3\wp(z-\frac{\omega_k}{2};\tau)+\sum_{k=0}^3T_k\zeta(z-\frac{\omega_k}{2};\tau)+B\right)y(z),
\end{equation*}
where $T_0,\cdots, T_3, B$ are apparent parameters. 
%$(T_0,\cdots, T_3, B)\in AP$. Here, $AP$ denotes the space of apparent parameters, which is an algebraic set. 
%with $T_0,\cdots, T_3, B\in \mathbb{C}$ and $\sum_{i=0}^3T_i=0.$ 
We first decompose the space of apparent parameters, which turns out to be an algebraic set, into three irreducible components. These three components intersect at $(T_0,\cdots, T_3, B)=(0,\cdots, 0)$, which plays an important role in determining whether the monodromy matrices is unitary or not.
%yields important information on the unitary 
%each of them can be identified with $\mathbb{C}$. Our approach indicates that the geometry of AP yields some important informations.

Following the approach in KdV theory, we define the spectral polynomial which is a degree 4 polynomial of the apparent parameter. We then obtain that the monodromy is not completely reducible if and only if the apparent parameter is a zero of the spectral polynomial. By introducing a branched double cover of the apparent space,  which parametrizes all one-dimensional common eigenspaces, we determine the monodromy data for all apparent parameters. By noticing that the equation under the covering map is exactly the spectral polynomial, we obtain that the generalized Lam\'e curve is isomorphic to the spectral curve. 
%Therefore, we have a spectral curve which parametrizes all one-dimensional common eigenspaces. 

Finally, with the help of the spectral curve defined by the spectral polynomial, 
%the monodromy data can be locally expressed as a holomophic function of the apparent parameter. 
we  characterize the conditional stability sets in two directions by making use of the local analytic coordinates of the monodromy data and then prove that the monodromy matrices are unitary if and only if $(T_0,\cdots, T_3, B)=(0,\cdots, 0)$ when the period $\tau\in i\mathbb{R}_{>0}$. %the periodic parallelogram is rectangular. 
%$\tau\in i\mathbb{R}$.  %   Together with \cite{Lin-2025}, we are able to answer those two questions completely.

\iffalse 
On a torus $E_{\tau}=\mathbb{C}/\Lambda(\tau)$, where $\Lambda(\tau)=\mathbb{Z}\oplus\tau \mathbb{Z}, \text{Im}(\tau)>0$, we consider the space of apparent parameters (denoted by AP) of the following ODE
\begin{equation}\tag{1}\label{eq_07141315}
    \begin{aligned}
        &y^{\prime\prime}(z)=q(z)y, q(z)=\frac{3}{4}\sum\limits_{k=0}\wp (z-w_k/2)\\
        &+\sum T_k\zeta(z-w_k/2)+B,\ \  \sum\limits_{k=0}^3 T_k=0.
    \end{aligned}
\end{equation}
For (\ref{eq_07141315}), the monodromy matrices $M_j$, $j=1,2$, always commute. In section  0, we pose two fundamental questions [Q1] and [Q2] concerning the monodromy matrices. Together with [12], we are able to answer those two questions completely. Our approach indicates that the geometry of AP yields some important informations. For example, (i) for each irreducible component of AP, there associates with a polynomial. This polynomial plays a basic role to distinguish which parameter such that (\ref{eq_07141315}) is completely reducible. (ii) The ODE (\ref{eq_07141315}) is particalarly special at the parameter where AP has a singularity.
\fi 
\end{abstract}

\thanks{Erjuan Fu was supported by NSFC (No. 12401188) and BIMSA Start-up Research Fund.}

\maketitle
%\makecontents
\tableofcontents
\setcounter{tocdepth}{1} 

\section{Introduction}\label{intro}

\iffalse 
 In this paper, we use the notations $\omega_0=0$, $\omega_1=1$, $\omega_2=\tau$ and $\omega_3=1+\tau$, $\wp(z;\tau)$ and $\zeta(z;\tau)$ are the Weierstrass elliptic function and Weierstrass zeta function such that $\zeta(z;\tau)$ (or simply $\zeta(z)$) is odd and satisfies
\begin{align*}
    \zeta'(z)=-\wp(z).
\end{align*}
Another useful entire function is the (Weierstrass) sigma function $\sigma$ satisfying
\begin{align*}
    \zeta(z)=\sigma'(z)/\sigma(z),\ \ \sigma(z+\omega_i)=-e^{\eta_i(z+\omega_i/2)}\sigma(z).
\end{align*}

\fi 
% In the theory of Fuchsian differential equations, the monodromy representations of solutions  play an important role in studying the geometric nature of solutions. 
This is the first in a series of papers on %characterization of  
the unitary monodromy problem for the following generalized Lam\'e equations with $n+1\in\mathbb{Z}_{>0}$  %regular
 singularities denoted by $p_0, p_1, \cdots, p_n\in E_\tau$:
\begin{equation}\label{eq-lame}
    \begin{aligned}
 &\mathcal{L}_{\textbf{m},\Sigma}(\mathbb{T}, B):\quad  
 y''(z)=q_{\textbf{m},\Sigma}(z;\mathbb{T}, B)y(z),\quad z\in \mathbb{C}\\
%&q(z;\mathbb{T}, B):=\sum_{k=1}^n \frac{m_k}{2}(\frac{m_k}{2}+1)\wp(z-p_k;\tau)+\sum_{k=1}^nT_k\zeta(z-p_k;\tau)+B
\end{aligned}
\end{equation}
where  
$$
q_{\textbf{m},\Sigma}(z;\mathbb{T}, B)=\sum_{k=0}^n \frac{m_k}{2}(\frac{m_k}{2}+1)\wp(z-p_k;\tau)+\sum_{k=0}^nT_k\zeta(z-p_k;\tau)+B,$$
with   
\begin{itemize}
\item $\Sigma:=\{p_0, \cdots, p_n\}$,
\item $\textbf{m}:=(m_0,\cdots, m_n)\in (\mathbb{Z}_{>0})^{n+1}$ satisfying $\sum_{k=0}^n m_k\in 2\mathbb{Z}_{>0}$,
\item
$(\mathbb{T}, B):=(T_0,\cdots, T_n, B)\in \mathbb{C}^{n+2}$  satisfying $
   \sum_{k=0}^nT_k=0$.
\end{itemize}
Here, the notations  $\tau\in \mathbb{H}=\{z\mid \text{Im}(z)>0\}$, $E_\tau=\mathbb{C}/\Lambda_\tau$ with $\Lambda_\tau=\mathbb{Z}+\mathbb{Z}\tau$, and $\wp(z;\tau), \zeta(z;\tau)$ denote the Weierstrass elliptic and zeta functions, respectively. 
%Moreover, we assume that  $m_1,\cdots, m_n\in \mathbb{N}$ with $\sum_{k=1}^n m_k \in 2\mathbb{Z}$ %being an even integer, which is to make sure the monodromy matrices commute. 
%The integrality assumption is such that the equation is an algebraically integrable system and methods in algebraic geometry and modular forms can be brought in to study the detailed structures of the moduli spaces of solutions.
% we consider the following Fuchsian dierential equations of second order with $n$  %regular singularities denoted by $p_1, \cdots, p_n\in E_\tau$:
%Linear ordinary differential equations having only regular singular points are called Fuchsian Differential Equations 
%In a series of coming papers, we aim to study the following second order ODE of Fuchsian type:
The condition $
   \sum_{k=0}^nT_k=0
$ is to make sure $q_{\textbf{m},\Sigma}(z; \mathbb{T}, B)$ is elliptic, i.e., it
is a meromorphic function on $E_\tau$. 
For $z\in\mathbb{C}$, we denote by $[z]:=z \, (\mod \Lambda_\tau) \in E_\tau$, and we write $z$ instead of $[z]$ for a point $[z]\in E_\tau$ to simplify notations when no confusion arises. 
%Those constants $T_1,\cdots,T_m$ and $B$ are not arbitrary. 

%Under the integrality assumption dj N, the equation is an algebraically integrable system and methods in algebraic geometry and modular forms can be brought in to study the detailed structures of the moduli spaces of solutions.

Notice that the local exponents of $\mathcal{L}_{\textbf{m},\Sigma}(\mathbb{T}, B)$ at  $p_k$ are $-{m_k}/{2}$ and $1+{m_k}/{2}$,  then their difference is $m_k\in\mathbb{Z}_{>0}$, so $\mathcal{L}_{\textbf{m},\Sigma}(\mathbb{T}, B)$  might have a solution $y(z)$ with the logarithmic singularity at $p_k$.   If all solutions of  $\mathcal{L}_{\textbf{m},\Sigma}(\mathbb{T}, B)$ are free of logarithmic  singularity at $p_k$,  we call $\mathcal{L}_{\textbf{m},\Sigma}(\mathbb{T}, B)$ is  \textit{apparent} at $p_k$.  If $\mathcal{L}_{\textbf{m},\Sigma}(\mathbb{T}, B)$ is apparent at all singularities, we simply call 
$\mathcal{L}_{\textbf{m},\Sigma}(\mathbb{T}, B)$ is apparent and then the 
 parameter $(\mathbb{T},B)$ is called an apparent parameter. % in this case.
\iffalse 
By adding some restrictions on the paramenter $(\mathbb{T}, B)$, we can guarantee that $\mathcal{L}_{\textbf{m},\Sigma}(\mathbb{T}, B)$ is apparent.
\fi 
Denote the space of apparent parameters by 
$$AP_{\textbf{m},\Sigma}:=\left\{(\mathbb{T},B)\in \mathbb{C}^{n+2} \mid  \sum_{j=0}^nT_j=0 \,\, \text{and}\,\, \mathcal{L}_{\textbf{m},\Sigma}(\mathbb{T}, B)\, \text{ is apparent}.\right\}.$$
By applying the classic Frobenius method, we will see that $AP_{\textbf{m},\Sigma}$ is an algebraic set and the defining polynomials can be explicitly written out for $\textbf{m}=(1,\cdots,1)$, see Section \ref{sec-apparent} for details. 
%all singularities are apparent, i.e., all solutions of  $\mathrm{Eqn}(\mathbb{T}, B)$ are free of logarithmic singularities. 
% no solutions with such a singular behavior near $p_k$. If it holds, then (\ref{ceq07120443}) is called apparent at $p_k$. Those constants $T_1,\cdots,T_m$ and $B$ are choosen so that (\ref{ceq07120443}) is apparent at all the singular points. It is elementary 
\iffalse 
Indeed, it is well known that there is a polynomial $P_k(\mathbb{T},B)$ such that $\mathcal{L}_{\textbf{m},\Sigma}(\mathbb{T}, B)$ is apparent at $p_k$ if and only if the parameters  satisfy $P_k(\mathbb{T},B)=0$. By applying the classic Frobenius method, this polynomial $P_k(\mathbb{T},B)$ can be  obtained through an elementary computation, see Section \ref{sec-apparent} for details.
 %for example, see  \cite{CKL1,CKL2,CKL3,CKL4}.  
In particular, we compute this polynomial for $\textbf{m}=(1,\cdots,1)$  in Section \ref{sec-apparent}.  %The set of all apparent parameters is simply called
%Denote the space of apparent parameters by 
%% denoted by AP, that is
%\begin{align*}
%    AP=\left\{(\mathbb{T},B)\in \mathbb{C}^{n+1} \mid \sum_{j=1}^nT_j=0,\, \, P_k(\mathbb{T},B)=0, \,k=1, \cdots, n.\right\}.
%\end{align*}
%Clearly, $AP$ is an algebraic set. 
\fi 
In general, $AP_{\textbf{m},\Sigma}$ could be very complicated. However, we expect the geometry of $AP_{\textbf{m},\Sigma}$ provides useful informations about $\mathcal{L}_{\textbf{m},\Sigma}(\mathbb{T}, B)$. 

Let $(\mathbb{T}, B)\in AP_{\textbf{m},\Sigma}$. Due to  $\textbf{m}\in (\mathbb{Z}_{>0})^{n+1}$, the local monodromy of $\mathcal{L}_{\textbf{m},\Sigma}(\mathbb{T}, B)$ near each singularity $p_k$ is $\pm I_2$,  we only need to consider the monodromies along the two fundamental cycles of $E_\tau$. 
% In the theory of generalized Lam\'e equations 
%second order Fuchsian differential equations 
%with apparent parameters, the monodromy representations %of solutions 
%play an important role in studying the geometric nature of solutions. 
%One important data is the monodromy data. 
%The monodromy matrices $M_j(z_0),j=1,2$, are defined as follows. 
 Fix a point $z_0\in \mathbb{C}$ such that along the lines $z_0+\mathbb{R}$ and $z_0+\tau\mathbb{R}$, there are no singularities of the potential $q_{\textbf{m},\Sigma}(z;\mathbb{T}, B)$. Let $y_1(z)$ and $y_2(z)$ be two linearly independent solutions of $\mathcal{L}_{\textbf{m},\Sigma}(\mathbb{T}, B)$ near $z_0$. Since $q_{\textbf{m},\Sigma}(z;\mathbb{T}, B)$ is elliptic, there are two  $2\times2$ matrices $M_1(z_0), M_1(z_0)$ such that
\begin{align*}
    \left(\begin{array}{cc}y_1(z+1)\\y_2(z+1) \end{array}\right)=M_1(z_0)\left(\begin{array}{cc}
    y_1(z)\\y_2(z) \end{array}\right), \,\,
 \left(\begin{array}{cc}y_1(z+\tau)\\y_2(z+\tau) \end{array}\right)=M_2(z_0)\left(\begin{array}{cc}
    y_1(z)\\y_2(z) \end{array}\right)  %, \quad j=1,2,
\end{align*}
for $z$ in a neighborhood of $z_0+\mathbb{R}$ or $z_0+\tau\mathbb{R}$. %, where $\omega_1=1$ and  $\omega_2=\tau$. %Of course, $M_j(z_0)$ do depend on $z_0$. In fact, 
Notice that, % $M_j(z_0)$  is  independent of the choice of $z_0$ up to conjugation by an element in $\mathrm{GL}(2,\mathbb{C})$.  
%In fact,
 for different $z_0'$, we have $M_j(z_0')=\pm M_j(z_0)$ because the local monodromy  near each singularity is $\pm I_2$. % because $m_k\in \mathbb{Z}_{>0}, k=1, \cdots, n$. 
Hence, without loss of generality, we can fix the base point $z_0$ and then simply denote the monodromy matrices by $M_1,M_2$.  Of course, $M_i=M_i(\mathbb{T}, B), \, i=1,2$ is depend on the apparent parameter. 
 Furthermore, due to  the assumption $\sum_{k=0}^nm_k$ is an even integer,  we always have $
M_1M_2=M_2M_1$.
Therefore, the monodromy representation can be split into two cases: \emph{completely reducible case and not completely reducible case.}
  %\begin{definition}
The equation $\mathcal{L}_{\textbf{m},\Sigma}(\mathbb{T}, B)$ is called completely reducible if $M_1, M_2$ can be simultaneously diagonalized. Otherwise, it is called not completely reducible.
%\end{definition}
%Clearly, the definition is independent of the choice of $z_0$. %Some natural questions concerning $\mathrm{Eqn}(\mathbb{T}, B)$ are considered in this series of paper:
%
%As we all know, if $\mathcal{L}(\mathbb{T}, B)$ is not completely reducible, then the monodromy matrices are definitely not unitarizable.
%
If $\mathcal{L}_{\textbf{m},\Sigma}(\mathbb{T}, B)$ is completely reducible, up to a common conjugation, we can set \begin{equation}\label{CompRed}M_1=\begin{pmatrix} e^{-2\pi is}&0\\ 0& e^{2\pi is}\end{pmatrix},\quad M_2=\begin{pmatrix} e^{2\pi ir}&0\\ 0& e^{-2\pi ir}\end{pmatrix},\quad s,r\in \mathbb{C}.\end{equation}  Here, $M_1, M_2$ can be in $\{\pm I\}$ simultaneously. %The solutions $y_1(z), y_2(z)$ with the monodromy matrices $M_1, M_2$ are called common eigenfunctions of   $\mathcal{L}_{\textbf{m},\Sigma}(\mathbb{T}, B)$.  %$e^{2\pi is}$ and $e^{2\pi ir}$ might be $\pm 1$. 
%Of course, if $y_1$ and $y_2$ exchange, then $(s,r)$ changes to $(-s,-r)$.
%
 %
If $\mathcal{L}_{\textbf{m},\Sigma}(\mathbb{T}, B)$ is not completely reducible, up to a common conjugation, we can set 
\begin{equation}\label{CompNRed}
M_1=\varepsilon_1\begin{pmatrix} 1&0\\ 1& 1\end{pmatrix},\quad M_2=\varepsilon_2\begin{pmatrix} 1&0\\ \mathcal{C}& 1\end{pmatrix},\end{equation} 
with $\varepsilon_1, \varepsilon_2\in \{\pm 1\}$ and $\mathcal{C} \in \mathbb{C}\cup \{\infty\}$.  If $\mathcal{C}=\infty$, then (\ref{CompNRed}) is understood as 
\begin{equation*}
M_1=\varepsilon_1\begin{pmatrix} 1&0\\ 0& 1\end{pmatrix},\quad M_2=\varepsilon_2\begin{pmatrix} 1&0\\ 1& 1\end{pmatrix}.\end{equation*} 
%\begin{equation}\label{CompNRed}
%M_j=\varepsilon_j\left(\begin{array}{cc} 1&0\\ \chi_j& 1\end{array}\right), \quad \varepsilon_j\in \{\pm 1\},\,\chi_j \in \mathbb{C}\, \text{ and }\, |\chi_1|+|\chi_2|\neq 0.\end{equation} 
 Throughout the paper, $(s,r) \mod \mathbb{Z}^2$ and $\mathcal{C}$ are called the \emph{monodromy data} of  $\mathcal{L}_{\textbf{m},\Sigma}(\mathbb{T}, B)$  for each case.  %\footnote{Here, we use $\chi$ instead of $a$ and the because $\chi$ is used later in this paper.} 
Obviously, the following questions are our first concern:

\begin{question}\label{Q1}
Given $\Sigma$ and \textbf{m}.
How to distinguish the apparent parameters $(\mathbb{T},B)$ such that $\mathcal{L}_{\textbf{m},\Sigma}(\mathbb{T}, B)$ is completely reducible?
\end{question}

\begin{question}\label{Q0}
Given $\Sigma$ and \textbf{m}.
What is the monodromy data of $\mathcal{L}_{\textbf{m},\Sigma}(\mathbb{T}, B)$ with $(\mathbb{T}, B)\in AP_{\textbf{m},\Sigma}(\tau)$?
\end{question}

%The solutions $y_1(z), y_2(z)$ with the monodromy matrices $M_1, M_2$ are called common eigenfunctions of   $\mathcal{L}_{\textbf{m},\Sigma}(\mathbb{T}, B)$.

The key step in solving these two questions is to find the common eigenfunctions of $\mathcal{L}_{\textbf{m},\Sigma}(\mathbb{T}, B)$ so that we can express $M_1, M_2$ in the form of (\ref{CompRed}) and (\ref{CompRed}).  By a common eigenfunction $y(z)$ of $\mathcal{L}_{\textbf{m},\Sigma}(\mathbb{T}, B)$, we mean that $y(z)$ solves $\mathcal{L}_{\textbf{m},\Sigma}(\mathbb{T}, B)$ and  there exist $\lambda_1, \lambda_2 \in \mathbb{C}$ such that 
\begin{equation}\label{intro-commeigen}
y(z+1)=\lambda_1 y(z) \quad \text{and}\quad y(z+\tau)=\lambda_2 y(z).
\end{equation}
%
%Let $T\in AP$, a solution $y(z)$ of $\mathcal{L}(T)$ is called a common eigenfunction of $\mathcal{L}(T)$ if 
%$$y(z+\omega_j)=\lambda_j y(z),\qquad j=1,2,$$
%where $\lambda_j$ is an eigenvalue of $M_j, j\in \{1,2\}$. 
%
Here, $\lambda_j$ is an eigenvalue of $M_j, j\in \{1,2\}$.
Since the local exponents of $y(z)$ at  $p_k$ are $-{m_k}/{2}$ and $1+{m_k}/{2}$, 
then the square of $y(z)$  is an elliptic function of second kind. By the classical function theory, $y(z)$ can be written as a so-called Hermite-Halphen ansatz (see (\ref{HHa}) below).  Apply the transformation law of the Weierstrass sigma function, the monodromy data $(s,r)$ can be calculated in terms of the zeros of $y(z)$ (see (\ref{commony}) ) if $\mathcal{L}_{\textbf{m},\Sigma}(\mathbb{T}, B)$ is completely reducible. Although the computation in Section \ref{sec-data} is written for a special case (see (\ref{q})  below), the procedure can work for $\mathcal{L}_{\textbf{m},\Sigma}(\mathbb{T}, B)$ without any extra difficulty. We emphasize that those computations basically has been carried out in \cite{CLW-2015}. In Section \ref{sec-data}, details are given just for the convience of readers. 

Our second concern is the question of unitarizable monodromy matrices of the ODE. Here ``unitarizable" means that $M_1, M_2$ can be unitary by a common conjugation. 
%
\iffalse 
%In particular, there is an important property when
Suppose that $M_1,M_2$ are expressed by  (\ref{CompRed}) under the basis of solutions $\{y_1, y_2\}$. If $\mathcal{L}_{\textbf{m},\Sigma}(\mathbb{T}, B)$  is completely reducible, then the ratio $h(z)=y_1(z)/y_2(z)$ is a single-valued meromorphic function. In classic analysis,  $h(z)$ is a so-called elliptic function of second type by (\ref{CompRed}). Therefore, we could apply the ODE $\mathcal{L}_{\textbf{m},\Sigma}(\mathbb{T}, B)$ to study such classic subject. 
If $\mathcal{L}_{\textbf{m},\Sigma}(\mathbb{T}, B)$ is not completely reducible, then $M_1,M_2$ can not be unitarizable by (\ref{CompNRed}), where ``unitarizable'' means that the matrix is unitary up to a common conjugation. In particular,
\fi 
 To determine whether the monodromy matrices are unitarizable or not is a basic question in representation theory  and  also an very important application in studying the solvablity of the corresponding curvature equation.  The curvature equation related to $\mathcal{L}_{\textbf{m},\Sigma}(\mathbb{T}, B)$
is as follows:
\begin{align}\label{eq07120532}
    \Delta u+e^u=4\pi\sum_{k=1}^nm_k\delta_{p_k}\text{\ on\ }E_\tau,
\end{align}
where $\delta_{p_k}$ is the Dirac measure at $p_k$. % and $E_\tau=\mathbb{C}/\Lambda_\tau$ with $\Lambda_\tau=\mathbb{Z}+\mathbb{Z}\tau$. 
By the well-known Liouville theorem, the (\ref{eq07120532}) is solvable if and only if there is $(\mathbb{T},B)\in AP_{\textbf{m},\Sigma}(\tau)$ such that the monodromy matrices of $\mathcal{L}_{\textbf{m},\Sigma}(\mathbb{T}, B)$ are unitarizable.  
This connection is well-known. For example, see \cite{FL-2025-2} for a proof. Therefore, to be able to solve the nonlinear PDE (\ref{eq07120532}), we propose to study the following question:
\begin{question}\label{Q-1}
Given $\Sigma$ and \textbf{m}. Is there an element $(\mathbb{T}, B)\in AP_{\textbf{m},\Sigma}$ such that 
$\mathcal{L}_{\textbf{m},\Sigma}(\mathbb{T}, B)$ is unitarizable?
\end{question}

Obviously, if $\mathcal{L}_{\textbf{m},\Sigma}(\mathbb{T}, B)$ is not completely reducible, then it is not unitarizable. For the completely reducible case, Question \ref{Q-1} is equivalent to the monodromy data $(s,r)\in \mathbb{R}^2$. 
\iffalse 
Interestingly, the classical theory of Schr\"odinger equations provides us another method to study Question \ref{Q-1}. 
%There is another interesting question related to $\mathcal{L}_{\textbf{m},\Sigma}(\mathbb{T}, B)$.
For the eigenvalue problem of the following Hill equation 
%one-dimensional Schr\"odinger equation 
with complex-valued periodic potential: 
$$y''(x)+q(x)y(x)=Ey(x),\qquad x\in \mathbb{R},$$
the Hill's discriminant $\Delta(E)$ is defined as the trace of the monodromy matrix and the conditional stability set $\sigma$ is introduced as (see \cite{B-CPAM,Rofe-Beketov}) 
$$\sigma(E)=\Delta^{-1}([-1,1]).$$
\fi 
This motivates us to study the conditional stability set of  $\mathcal{L}_{\textbf{m},\Sigma}(\mathbb{T}, B)$. This idea is inspired by \cite{CL-2020}, where  the conditional stability sets of the Lam\'e equation
\begin{equation}\label{lame}
y''(z)=\left(k(k+1)\wp(z)+\lambda \right)y(z), \qquad k\in \mathbb{Z}_{\geq 1},
\end{equation}
is used to prove the non-existence of solutions for the curvature equation 
$$\Delta u+e^u=8k\pi \delta_0,\qquad \text{on}\,\, E_\tau, \qquad  k\in \mathbb{Z}_{\geq 1}.$$
See also \cite{CFL-2026} for the study of conditional stability sets (as parts of the spectral geometry) of other linear ODEs.
 Remark that the conditional stability sets of (\ref{lame}) coincide with the spectra of the associated Lam\'e operators, but the conditional stability sets of $\mathcal{L}_{\textbf{m},\Sigma}(\mathbb{T}, B)$ are not becasue $B$ depends on $\mathbb{T}$ and  can not be seen as an eigenvalue. 
%We note that $B$ is not an eigenvalue of $\mathcal{L}_{\textbf{m},\Sigma}(\mathbb{T}, B)$ since $B$ depends on $\mathbb{T}$.

Specifically,  let $V$ be an irreducible component of $AP_{\textbf{m},\Sigma}$, and $M_j(\mathbb{T}, B), j=1,2$ be the monodromy matrices of $\mathcal{L}_{\textbf{m},\Sigma}(\mathbb{T}, B)$ with $(\mathbb{T}, B)\in V$. We define the Hill's discriminant by 
\begin{equation}\label{intro-Delta}
\Delta_j(\mathbb{T}, B;V)=\frac{1}{2}\mathrm{tr}M_j(\mathbb{T}, B), \qquad (\mathbb{T}, B)\in V,
\end{equation}
and then \emph{the conditional stability set} on $V$ by
\begin{equation}\label{intro-spectrum}
\sigma_j(\mathbb{T}, B;V)=\Delta_j^{-1}([-1,1]) \subseteq V,\qquad j=1,2.
\end{equation}
 In \cite{CLW-2015,CL-2020,CFL-2022,CKL4,EG-2016}, these two arcs already had some important applications. Natually, we raise the following question. 

\begin{question}
Study the geometry of the  conditional stability sets. 
\end{question}

Obviously, the above questions are highly nontrivial questions for 
 $\mathcal{L}_{\textbf{m},\Sigma}(\mathbb{T}, B)$.  %Served as an example,  
In this paper,  we consider 
$n=4$, $\textbf{m}=\textbf{m}_0:=(1,1,1,1)$ and $$\Sigma= E_\tau[2] :=
\{\frac{\omega_k}{2}\mid k=1,2,3,4\},$$ where $\omega_0=0, \omega_1=1, \omega_2=\tau$ and $\omega_3=1+\tau$.  
Thus the potential has the following form 
\begin{equation}\label{q}
\begin{aligned}
 %5 \mathrm{Eqn}(\mathbb{T}, B):   &y''(z)=q(z)y(z),   \\
&q(z; \mathbb{T}, B):=\dfrac{3}{4}\sum_{k=0}^3\wp(z-\frac{\omega_k}{2})+\sum_{k=0}^3T_k\zeta(z-\frac{\omega_{k}}{2})+B,
\end{aligned}
\end{equation}
with $(\mathbb{T}, B)=(T_0, \cdots, T_3, B)$.  This is the simplest example among the class of equations $\mathcal{L}_{\textbf{m},\Sigma}(\mathbb{T}, B)$.  
We will see later that the space $AP:=AP_{\textbf{m}_0,E_\tau[2]}$ of apparent parameters can be described concretely. In Section \ref{sec-apparent}, we will show that $AP=\cup_{j=1}^3 V_j$, where each $V_j$ is an irreducible component of $AP$ and $V_k\cap V_j=\{(0,0,0,0,0)\}$ for $k\neq j$.  %Thus $AP$ has a unique singularity at $(0,0,0,0)$.  }
Notice that any component $V_j$ of $AP(\tau)$  can be smoothly identified with $\mathbb{C}$, then $(\mathbb{T}, B)\in AP(\tau)$ can be denoted by one parameter $T\in V_\ell$ for some $\ell\in \{1,2,3\}$ and   %$AP$ has a unique singularity at $T=0$. 
$V_k\cap V_j=\{0\}$ for $k\ne j$. 
%$0\in V_j$ for all $j\in\{1,2,3\}$. 
%The  corresponding potential and equation are simply denoted by $q_\ell (z;T)$ and $\mathcal{L}_\ell(T)$, respectively.
%Thus $AP$ has a unique singularity at $(0,0,0,0)$. 
%
%Obviously, the parameter $T=0$ is special, which is the only singluar point of $AP$.
%
Therefore, $T=0$ is the only singular point of $AP$. Naturally, we wonder whether the monodromy matrices might have particular property at this singular point. It turns out that the following is a surprising result.

%Surprising, we have the following result concerning the monodromy of the equation at $T=0$, which accidentally is the unique singularity of $AP$. Indeed, our first result  about the monodromy in Section \ref{sec-monodromy} is related to this special point. 
\begin{theorem}\label{thm-intro-1}
Let  $T\in AP$, then $T=0$ if and only if  % i.e.,  $(T_0,\cdots, T_3,B)=(0,\cdots, 0)$, 
the monodromy matrices $M_k$ of $\mathcal{L}_{\textbf{m}_0, E_\tau[2]}(T)$ satisfy $M_k=\varepsilon_k I_{2\times 2}$ with  $\varepsilon_k\in\{\pm 1\}, k=1,2.$
\end{theorem}
%The proof will be presented in section 1. 
This simply-looking result, proved in Section \ref{sec-monodromy}, has never appeared in the integral L\'ame equations and seems not obvious at all.
%Moreover, this result shows that the monodromy is unitarizable if $T=0$.  In Section \ref{sec-unitary}, we will show that $T=0$ is also a necessary condition for the monodromy to be unitarizable if the periodic parallelogram is rectangular. %, that is, $q(z)=\frac{3}{4}\sum\limits_{k=0}^3\wp(z-\omega_k/2)$.
%
 In order to prove Theorem \ref{thm-intro-1} and other main results of this paper,
%This simply-looking result seems not obvious at all. 
%In order to prove this theorem, 
 we introduce the classical second symmetry product equation for  $\mathcal{L}_\ell(T):=\mathcal{L}_{\textbf{m}_0,E_\tau[2]}(T)$ with $T\in V_\ell$ and $\ell\in \{1,2,3\}$, and also the generalized Lam\'e curve for each irreducible component $V_\ell$  in Section \ref{sec-monodromy}. Indeed, those ideas are motivated by the previous works (\cite{BG-JAM, CFL-2022,CKL1,HHV-2017}) when the potential of the studied ODE is an elliptic function and satisfies the nonlinear KdV equation. 

Let $y_1, y_2$ be two solutions of  $\mathcal{L}_\ell(T)$. 
It is well-known  (see \cite{WW}) that their product $F(z)=y_1y_2$ satisfies the third order equation:
\begin{equation}\label{eqn-0}
    F^{\prime\prime\prime}-4qF^\prime-2q^\prime F=0,
\end{equation}
which is called the second symmetry product equation of  $\mathcal{L}_\ell(T)$ . 
We first prove that (\ref{eqn-0}) always has the following nontrivial elliptic solution
\begin{equation}
 F^{(\ell)}(z;T)=-\sum\limits_{k=0}^3\varepsilon_k^{(\ell)}\zeta(z-\frac{\omega_k}{2})+(2T+\eta_{3-\ell}),
\end{equation}
where $\varepsilon_0^{(\ell)}=\varepsilon_\ell^{(\ell)}=-\varepsilon_j^{(\ell)}=-\varepsilon_k^{(\ell)}=1$ with $\{\ell,j,k\}= \{1,2,3\}$,  $\eta_{0}=0$ and $\eta_1, \eta_2$ are the quasi-periods of $\zeta(z)$. 
%$\varepsilon_k^{(0)}=1, \varepsilon_k^{(i)}=\delta_{ik}$ and $\eta_k=\frac{1}{2}\zeta(\frac{\omega_k}{2})$.
We should emphasize that such an elliptic solution of (\ref{eqn-0})  exists for a general ODE $\mathcal{L}_{\textbf{m}, \Sigma}(\mathbb{T}, B)$, this is one of reasons why our approach seems work %for $\mathcal{L}_{\textbf{m}, \Sigma}(\mathbb{T}, B)$ 
in general.
With the help of this special elliptic solution, we obtain that the space of elliptic solutions of (\ref{eqn-0}) is of dimension 3 for $T=0$ and is of dimension 1 for $T\neq 0$, which implies Theorem \ref{thm-intro-1}.
In particular,  %any solution of  (\ref{eqn-0}) is an elliptic function if $T=0$. 
%Note that $F^{(\ell)}(z;0)\neq F^{(j)}(z;0)$ if $\ell \neq j$. 
%Note that 
our choice of ${F^{(\ell)}}(z;T)$ holomorphically depends on the apparent paramenter $T$ along $V_\ell$ even at $T=0$, where all the solutions of (\ref{eqn-0}) are elliptic. 

\iffalse 
Next, we try to figure out the monodromy for $T\in V_\ell\setminus \{0\}$ in Section \ref{sec-data}.  %Notice that the component $V_i$ of $AP$  can be identified with $\mathbb{C}$, an element of $V_i$ can be denoted by one parameter $T$ and then we denote the  corresponding potential and equation by $q(z;T)$ and $\mathrm{Eqn}(T)$, respectively. 
Even if the potential $q_\ell(z;T):=q_{\textbf{m}_0,E_\tau[2]}(z;T)$ is not a KdV potential, the elliptic solution ${F^{(\ell)}}:={F^{(\ell)}}(z;T)$ of the  
second symmetry product equation motivates us to define
\fi 

By applying $F^{(\ell)}(z;T)$, we then define
 the following so-called \textit{spectral polymial} for $\mathcal{L}_\ell(T)$: % following the approach of KdV potentials. 
\begin{equation}\label{Qell}
Q_\ell(T)=({F^{(\ell)}}^{\prime})^2-2{F^{(\ell)}}^{\prime\prime}F^{(\ell)}+4q_\ell(z; T)(F^{(\ell)})^2,
\end{equation}
where  the derivatives are taken with respect to $z$. This concept of the spectral polynomial  %is derived from the approach of KdV potentials. 
is originally defined and has played important roles in the KdV theory (see \cite{GH-2003, GW2, GW-1996, HHV-2017}.).
By noticing that $Q_\ell(T)$ is  independent of $z$, we prove that 
 \begin{equation*}
       Q_\ell(T)=16\prod_{k\neq  \ell}(T^2+e_\ell-e_k),%\ 1\leq i\leq 3,
    \end{equation*}
  where $e_k=\wp(w_k/2), k=1,2,3$.  
This spectral polynomial $Q_\ell(T)$  answers Question \ref{Q-1} completely for $\mathcal{L}_\ell(T)$ in the following theorem.
\iffalse 
By a direct calculation, we obtain that $Q_\ell(T)$ is identical with the square of the Wronskian of $y_1$ and $y_2$.  %, where $F^{(\ell)}=y_1y_2$.  
Thus we see that if $T\neq 0$, then %Theorem \ref{thm-intro-Q} tells that 
$Q(T)\neq 0$ if and only if $\mathcal{L}_\ell(T)$ is completely reducible. However, if $T=0$, then we do not know whether $y_1, y_2$ are linearly independent or not. The particular property at $T=0$ gives a little difficulty to prove the following result.  % which answers Qestion \ref{Q-1} for $q(z;\mathbb{T}, B)$ completely. % and applying Theorem \ref{thm-intro-1}, we obtain the following proposition and then completely solve Question \ref{Q1}. 
\fi 

\begin{theorem}\label{thm-intro-Q}
    Let $T\in V_\ell$ for some $\ell\in \{1,2,3\}$.   We have $\mathcal{L}_\ell(T)$ is completely reducible if and only if $Q_\ell(T)\neq 0$.
\end{theorem}

Notice that the local exponents of $\mathcal{L}_\ell(T)$  at the each singular point ${\omega_k}/{2}$
%$$E_\tau[2]:=\left\{ \frac{\omega_k}{2}\in E_\tau\,|\, k=0, 1, 2, 3.\right\}$$
are $-{1}/{2}$ and ${3}/{2}$, then 
\iffalse 
the square of a common eigenfunction  is elliptic function of the second kind, so any zero of the common eigenfunction (in our case, the set of zeros  is equal to $\{a, a-{\omega_\ell}/{2}\}$ with $a\in E_\tau\setminus E_\tau[2]$ ) can be well-defined as a point of $E_\tau$.  %it can be express in Weistrass sigma functions.e 
%the common eigenfunction is a Hermite–Halphen ansatz function. 
Thus, 
\fi 
any eigenfunction $y(z)$ can be expressed as the following Hermite-Halphen ansatz:
\begin{equation}\label{HHa}
y(z)=e^{cz}\frac{\sigma(z-a_1)\sigma(z-a_2)}{(\prod_{j=0}^3(z-\frac{\omega_j}{2})^{\frac{1}{2}}},
\end{equation}
where $c,a_1,a_2\in \mathbb{C}$ and $\sigma$ is the Weierstrass sigma function. 
By analyzing the local behavior of $\mathcal{L}_\ell(T)$ at each singularity and each zero (see Section \ref{sec-data}), we can recover the corresponding apparent parameter, denoted by $T_\ell$, from zeros of $y(z)$ and obtain the following %branched double covering 
map % denoted by $T_\ell$
\begin{equation}\label{Tell}
\begin{aligned}
T^{(\ell)}: Y_\ell &\to V_\ell\cong \mathbb{C}\\
\underline{a} &\mapsto T^{(\ell)}(\underline{a})= \zeta(a)+\zeta(a-\frac{\omega_\ell}{2})-\zeta(2a)+\frac{1}{2}\eta_\ell
\end{aligned}
\end{equation} 
where $$Y_\ell:=\{\underline{a} \mid \underline{a}:=\{a, a-\frac{\omega_\ell}{2}\}\in \mathrm{Sym}^2(E_\tau \setminus E_\tau[2]) \}.$$ %is a generalized Lam\'e curve.  %  is the parameter space of all one-dimensional common eigenspaces.
The computation for the map $T^{(\ell)}$ is basically the same as done in \cite{BG-JAM,CKL1,CKL2,CKL3}. For the convenience of readers, we will give details of the calculation in Section \ref{sec-data}.
There is another correspondence, i.e., the following result. 
\begin{theorem}\label{thm-intro-Tell}
The map $T^{(\ell)}$ is a branched double covering map and the branch points are exactly the zeros of $Q_\ell (T)$. In particular, 
$$Y_\ell\cong \Sigma_\ell:=\{(W,T)\mid W^2=Q_\ell(T)\}.$$
%The correspondence (\ref{Tell}) is a two to one and onto covering. Futhermore, the branch point is exactly the zeros of $Q^{(\ell)}(T)$. 
\end{theorem}

%To our surprise, the equation under the covering map is exactly the spectral polynomial $Q_\ell(T)$, i.e., 
%$$Y_\ell\cong \Sigma_\ell:=\left\{(W,T) \mid W^2=Q_\ell(T) \right\}.$$
 %Hence, $Y_\ell$ can be regarded as a generalized Lam\'e curve for the monodromy representation. 
By the correspondence, $Y_\ell\cong\Sigma_\ell$, we introduce the Baker-Akhiezer functions in Section \ref{sec-unitary}. 
Applying the Hermite-Halphen ansatz, we completely determine the monodromy data for $\mathcal{L}_\ell(T)$ in terms of zeros of common eigenfunctions. 

\begin{theorem}\label{thm-intro-data}
Let $\ell\in \{1,2,3\}$ and $T\in V_\ell$.
\begin{enumerate}
\item 
 If $Q_\ell(T)\neq 0$,
%$$[a]\not\in E_\tau[2] \cup \{\pm \frac{\omega_j}{4},\, \pm \frac{3\omega_j}{4},\,\pm \frac{\omega_k}{4},\,\pm\frac{3\omega_k}{4}\},$$ 
 then  $\mathcal{L}_\ell(T)$ 
is completely reducible. Up to a sign,  the monodromy data $(s,r)\mod \mathbb{Z}^2$ 
is determined by  
\begin{equation}\label{eqn-intro-sr}
\left\{
\begin{aligned}
r+s\tau&=2a-\frac{1}{2}\omega_\ell-\frac{1}{2}\omega_3\\
r\eta_1+s\eta_2&=\zeta(2a)-\frac{1}{2}\eta_\ell-\frac{1}{2}\eta_3
\end{aligned}\right.,
\end{equation}
where $\underline{a}\in (T^{(\ell)})^{-1}(T)$.
\item  If   $Q_\ell(T)=0$,  then  $\mathcal{L}_\ell(T)$ is not completely reducible and the monodromy data $\mathcal{C}$ is determined by
\begin{equation}\label{intro-C}
\mathcal{C}=\tau-\frac{2\pi i}{\eta_1+e_k}, \quad \text{if}\,\,\, T^2=e_k-e_\ell\,\, \text{and}\,\, k\in \{1,2,3\}\setminus \{\ell\}.
\end{equation}
\end{enumerate}
\end{theorem}

Combing with Theorem \ref{thm-intro-1}, a direct computation gives us the following corollary. 
\begin{corollary}\label{cor-intro-T=0}
%Let $T\in AP$, then $T=0$ if and only if t
The monodromy matrices of $\mathcal{L}(0)$ satisfy $M_1=M_2=-I_2$.
\end{corollary}

Corollary \ref{cor-intro-T=0} shows that the monodromy is unitarizable if $T=0$.  In Section \ref{sec-unitary}, with the help of spectral geometry, we show that $T=0$ is also a necessary condition for the monodromy to be unitarizable if $E_\tau$ is rectangular.  
Specifically, we define the  conditional stability sets $\sigma_j^{(\ell)} \subseteq V_\ell, \, j\in \{1,2\}$ for each $\ell\in \{1,2,3\}$ as (\ref{intro-spectrum}). In particular, we can obtain the explicit expression of $\sigma_j^{(\ell)}(\tau)$ for $\tau\in i\mathbb{R}_{>0}$, which yields the following result. % that is, $E_\tau$ is a rectangular. 

\begin{theorem}\label{thm-intro-T=0}
Let $\tau\in i\mathbb{R}_{>0}$. % and $\sigma_j(\ell)$ be the arc of stable spectrum for $\mathcal{L}_\ell(T)$.
We have $$\sigma_1^{(\ell)} \cap \sigma_2^{(\ell)}=\{0\}\cup \{T\in V_\ell \mid Q_\ell(T)=0\}.$$ Furthermore, $\mathcal{L}(T)$ with  $T\in AP$ is unitarizable if and only if $T=0$. %, that is, $q(z)=\frac{3}{4}\sum\limits_{k=0}^3\wp(z-\omega_k/2)$.
\end{theorem}

\iffalse 
\begin{corollary}
Let $\tau\in i\mathbb{R}_{>0}$. For any $(\mathbb{T}, B)\in AP(\tau)\setminus \{(0,0,0,0,0)\}$,  the ODE $\mathcal{L}(\mathbb{T},B)$  is not unitarizable.
\end{corollary}
\fi 

However, it seems that the spectral geometry can not help us to figure out the unitary monodromy problem for $T\neq 0$, which is completely solved in \cite{FL-2025-2}. 
The result 
%In , a further analysis is performed and then an improvement of Theorem \ref{thm-intro-data} is obtained and stated in the following. This improvement 
depends on the geometry of $E_\tau$. Denote by 
$$\mathcal{E}:=\left\{\tau\in \mathbb{H}\,\mid\, \text{The Green function has exactly five critical points.}\right\}.$$

\begin{Theorem}\cite{FL-2025-2}
Let $T\in AP\setminus \{0\}$.  
\begin{itemize}
\item  If $\tau\in \mathcal{E}$, then for each $\ell\in \{1,2,3\}$, there exists a unique pair $\pm T_\ell\in V_\ell$ such that $\mathcal{L}(\pm T_\ell)$ is unitarizable. 
\item If $\tau \notin \mathcal{E}$, then $\mathcal{L}(T)$ is always not unitarizable. 
\end{itemize}
\end{Theorem}

The paper is organized as below.  
In Section \ref{sec-apparent}, we obtain the explicit parameter of each irreducible component $V_j, j=1,2,3$ and $AP=\cup_{j=1}^3 V_j$. In Section \ref{sec-monodromy}, we introduce the spectral polynomial $Q_\ell(T)$ on an irreducible component $V_\ell$ and prove Theorem \ref{thm-intro-1} and Theorem \ref{thm-intro-Q}.
In Section \ref{sec-data},  we prove Theorem \ref{thm-intro-Tell} and Theorem \ref{thm-intro-data}.
Finally, Theorem \ref{thm-intro-T=0} is proved in Section \ref{sec-unitary}.

%
%\begin{remark}
%If $\tau\in i\mathbb{R}_{>0}$,  $\sigma_1^{(k)}$ and $\sigma_2^{(k)}$ can be related by the modular property of $\wp(z), \zeta(z)$, see Proposition \ref{prop-modular}. This observation gives us an insight to explore the conditional stabilitiy sets for $\tau\not\in i\mathbb{R}_{>0}$.
%\end{remark}

\section{The space of apparent parameters}\label{sec-apparent}
%\section{Apparent spaces in the primitive case}\label{sec-appspace}
We will concretely determine the apparent space $AP_{\textbf{m},\Sigma}$. 
%$$AP_{\textbf{m},\Sigma}:=\left\{(\mathbb{T},B)\in \mathbb{C}^{n+2} \mid  \sum_{j=0}^nT_j=0 \,\, \text{and}\,\, \mathcal{L}_{\textbf{m},\Sigma}(\mathbb{T},B)\, \text{ is apparent}.\right\}.$$
%
%Let $\textbf{m}=(m_1, \cdots, m_n)\in (\mathbb{Z}_{\geq 1})^n$ and $\sum_{j=1}^nm_j$ be an even number. We now consider  the following generalized Lam\'e equation: 
%\begin{align*}
%\mathcal{L}_{\textbf{m},\Sigma}(\mathbb{T},B; \tau): \quad y''(z)=q_\textbf{m}(z;\mathbb{T}, B, \tau)y(z),\quad z\in \mathbb{C},
%\end{align*}
%Recall that 
%\begin{align*}
%q_{\textbf{m},\Sigma}(z; \mathbb{T},B)=\sum_{j=0}^n\frac{m_j}{2}(\frac{m_j}{2}+1)\wp(z-p_j)+\sum_{j=0}^nT_j\zeta(z-p_j)+B
%\end{align*}
%which is an algebaric set. Indeed, 
Let 
$$q_{\textbf{m},\Sigma}(z;\mathbb{T}, B)=\sum_{j=0}^\infty b^{(k)}_j(z-p_k)^{j-2}$$
be  the Laurant series in a neighborhood of $p_k$,
then
\begin{align*}
b_0^{(k)}&=\frac{m_k}{2}(\frac{m_k}{2}+1)\\
b_1^{(k)}&=T_k\\
b_2^{(k)}&=\sum_{j\neq k}\frac{m_j}{2}(\frac{m_j}{2}+1)\wp(p_k-p_j)+\sum_{j\neq k}T_j\zeta(p_k-p_j)+B
\end{align*}
and $b_3^{(k)}, b_4^{(k)} , \cdots$ are linearly polynomials of $T_0, \cdots, T_n$.
Denote by 
\begin{align*}
c_0^{(k)}(\mu)&=\mu+\frac{m_k}{2},\\
c^{(k)}_t(\mu)&=\frac{\sum_{j=1}^{t}b_j^{(k)}c_{t-j}^{(k)}(\mu)}{\phi(\mu+t)},\qquad t\in \mathbb{Z}_{>0}.
\end{align*}
By the Frobenius method, we have  %$(\mathbb{T}, B)\in AP_{\textbf{m},\Sigma}$ if and only if 
%$c_{m_k+1}^{(k)}(-\frac{m_k}{2})=0,\,k=0, \cdots, n,$ i.e., 
$$ AP_{\textbf{m},\Sigma}=\left\{(\mathbb{T},B)\in \mathbb{C}^{n+1} \mid  \sum_{j=0}^nT_j=0 \,\, \text{and}\,\, c_{m_k+1}^{(k)}(-\frac{m_k}{2})=0, k=0, \cdots, n.\right\}.$$
Notice that $c_{m_k+1}^{(k)}(-\frac{m_k}{2})$ is a polynomial of $T_0, \cdots, T_n, B$, %\footnote{\textcolor{blue}{感觉没有必要加上degree，如果要加上degree的话：If we view $T_k$ as a variable of degree $1$ and $B$ as a variable of degree $2$, then $c_{m_k+1}^{(k)}(-\frac{m_k}{2})$ is of degree $\frac{1}{2}(m_0+1)\cdots (m_n+1)$. }}
 thus $AP_{\textbf{m},\Sigma}$ is an algebraic set.

Assume $m_0=\cdots=m_n=1$, %$\textbf{m}=\textbf{m}_n:=(1,1,\cdots, 1)$,  %and  $n+1=\sum_{j=0}^nm_j$ is even. Since $\textbf{m}$  is fixed, 
we then ignore $\textbf{m}$ in the notations and use the notation 
%$\mathcal{L}_{\Sigma}(\mathbb{T},B; \tau):=\mathcal{L}_{\alpha,\Sigma}(\mathbb{T},B; \tau)$ %$q_{\Sigma}(z; \mathbb{T},B,\tau):=q_{\alpha,\Sigma}(z; \mathbb{T},B,\tau)$ 
$AP_{\Sigma}$ to instead of $AP_{\textbf{m},\Sigma}$.
Notice that 
%\begin{align*}
% AP_{\Sigma}(\tau)=\left\{(\mathbb{T},B)\in \mathbb{C}^{n+1} \mid  \sum_{j=1}^nT_j=0 \,\, \text{and}\,\, c_{2}^{(k)}(-\frac{1}{2})=0, k=1, \cdots, n.\right\},
%\end{align*}
%and 
$c_{2}^{(k)}(-\frac{1}{2})=0$ is equivalent to 
% There is a precise description for $AP(\tau)$.  Indeed, 
%by the Frobenius method, 
%any solution of  $\mathcal{L}(\mathbb{T},B;\tau)$ has no logarithmic singularitiy at $p_k$ if and only if 
\begin{equation}\label{eqn-Tk}
T_k^2-\frac{3}{4}\sum_{j\neq k} \wp(p_k-p_j)-\sum_{j\neq k} T_j\zeta(p_k-p_j)-B=0. \tag{$*_k$} 
\end{equation}
%which is as the same as (\ref{B-T-1}). 
Denote by $\wp_{kj}:=\wp(p_k-p_j)$ and $\zeta_{kj}:=\zeta(p_k-p_j)$, we have 
$$AP_{\Sigma}=\left\{(\mathbb{T},B) \in \mathbb{C}^{n+2} \mid \begin{array}{l}\sum_{j=0}^n T_j=0,\, \text{and for }\,  k=0,1, \cdots, n,\\
T_k^2-\frac{3}{4}\sum_{j\neq k} \wp_{kj}-\sum_{j\neq k} T_j\zeta_{kj}-B=0.%\\ k=1, \cdots, n; \, T_1,\cdots, T_n, B\in \mathbb{C}. 
\end{array}\right\}.$$ %which  is an algebraic set.
In the rest of this paper, we will focus on the case of  $\Sigma=E_\tau[2]$.
 %$$\Sigma=\{p_1, p_2, p_3, p_4\}=E_\tau[2]=\{\frac{\omega_k}{2}\mid k=0,1,2,3\}.$$
%where $\omega_0:=0, \omega_1:=1, \omega_2:=\tau, \omega_3:=1+\tau$.  
%For the convience of notations, we use $T_0, T_1,T_2,T_3$ to replace $T_1, T_2, T_3, T_4$ respectively. 
Notice that  $\sum_{j\neq k}\wp_{kj}=\sum_{j=1}^3e_j=0$,   then  (\ref{eqn-Tk}) can be rewritten as 
\begin{equation}\label{Tk}
T_k^2-\sum_{j\neq k} T_j\zeta_{kj}-B=0, \qquad k=0,1,2,3.
\end{equation}
Thus,
$$AP:=AP_{E_\tau[2]}=\left\{(T_0,\cdots, T_3,B) \in \mathbb{C}^{5} \mid \sum_{j=0}^3 T_j=0,\, \text{and}\, (\ref{Tk})\, \text{hold}.\right\}.$$ 
 Define 
\begin{equation*}
    \begin{aligned}
    &V_1=\{(T,T,-T,-T,B)\,|\,B=T^2-\eta_2T, \,T\in\mathbb{C}\},\\
    &V_2=\{(T,-T,T,-T,B)\,|\,B=T^2-\eta_1T,\,T\in\mathbb{C}\},\\
    &V_3=\{(T,-T,-T,T, B)\,|\,B=T^2,\,T\in\mathbb{C}\},
    \end{aligned}
\end{equation*}
where $\eta_i:=\zeta(z+\omega_i)-\zeta(z),\,i=1,2,3$ are the quasi-periods of $\zeta(z)$. Clearly, $V_i\cap V_j=\{(0,0,0,0,0)\}$ for $i\ne j$. In the following theorem, we will see $AP$ can be decomposed into the union of $V_i, i=1, 2,3$.

\begin{lemma}\label{thm1.1}
    Let $V_i, i=1,2,3$ be defined above. Then  $AP=\bigcup\limits_{i=1}^3V_i$. In particular, $AP$ has exactly one singular point $(0,0,0,0,0)$.
\end{lemma}
\begin{proof}
 Let $(T_0,T_1,T_2,T_3,B)\in AP$, then $\sum_{j=0}^3 T_j=0$ and (\ref{Tk}) holds. 
Notice that $\zeta_{ij}=\zeta(\frac{\omega_i}{2}-\frac{\omega_j}{2})=\frac{1}{2}\eta_i-\frac{1}{2}\eta_j$. 
For any $\ell\in \{1,2,3\}$, denote by $\{\ell, j, k\}=\{1,2,3\}$, we compute
\begin{align*}
    T_\ell^2-T_0^2&=\sum_{i\neq \ell}T_i\zeta_{\ell i}-\sum_{i\neq 0}T_i\zeta_{0i}\\
&=T_0\zeta_{\ell 0}-T_\ell\zeta_{0\ell}+T_j(\zeta_{\ell j}-\zeta_{0j})+T_k(\zeta_{\ell k}-\zeta_{0k})\\
&=T_0\zeta_{\ell 0}+T_\ell\zeta_{\ell 0}+T_j\zeta_{\ell 0}+T_k\zeta_{\ell 0}\equiv 0.
\end{align*}
   % We note that since $\omega_i/2,\,0\leq i\leq 3$, are half periods, we have  $\sum\limits_{j\ne i}\wp_{ij}=0$ for all $i$. 
 %Note that $\sum_{\ell\neq k}\varepsilon_\ell^{(i)}\zeta_{k\ell}=c_i$, we have  $B=T_k^2-\sum_{j\neq k}\zeta_{kj}T_k,$
\iffalse 
We first show that $T_i^2=T_j^2$ for any $i,j$,  if $\sum_{j=0}^3 T_j=0$ and (\ref{Tk}) holds. By (\ref{Tk}), we have
\begin{align*}
    &T_0^2+\dfrac{1}{2}\eta_1T_1+\dfrac{1}{2}\eta_2T_2+\dfrac{1}{2}\eta_3T_3-B=0,\\
    &T_1^2-\dfrac{1}{2}\eta_1T_0+\dfrac{1}{2}(\eta_2-\eta_1)T_2+\dfrac{1}{2}\eta_2T_3-B=0,\\
    &T_2^2-\dfrac{1}{2}\eta_2T_0-\dfrac{1}{2}(\eta_2-\eta_1)T_2+\dfrac{1}{2}\eta_1T_3-B=0,\\
    &T_3^2-\dfrac{1}{2}\eta_3T_0-\dfrac{1}{2}\eta_2T_1-\dfrac{1}{2}\eta_1T_2-B=0,\\
\end{align*}
which implies that
\begin{align*}
    0=T_0^2-T_k^2+\dfrac{1}{2}\eta_k(T_0+T_1+T_2+T_3)=T_0^2-T_k^2,\quad  k=1,2,3.
\end{align*}
\fi 
Hence, $T_\ell=\pm T_0$ for all $\ell\in \{1,2,3\}$. %We use the same argument to show that $T_0^2=T_2^2$ and $T_0^2=T_3^2$. 
Combine with  %(\ref{cc3})
$\sum_{j=0}^3T_j=0$, 
we obtain that 
 %$(T_0,T_1,T_2,T_3,B)\in AP(\tau)$ implies 
$(T_0,T_1,T_2,T_3,B)\in V_\ell$ for some $\ell\in \{1,2,3\}$, because $B=T_0^2-\sum_{j=1}^3T_j\zeta_{0j}$. 

Finally, $\bigcup\limits_{j=1}^3V_j\subseteq AP$ is obvious.  
This proves the lemma.
\end{proof}

In what follows, we introduce the following notations: \begin{align*}
&\eta_0=0,\quad \varepsilon_0^{(i)}=\varepsilon_i^{(i)}=-\varepsilon_j^{(i)}=-\varepsilon_k^{(i)}=1, \, \text{if}\, \{i,j,k\}=\{1,2,3\},
 %&c_1=\eta_2,\ c_2=\eta_1,\,\text{and}\, c_3=0,
 \end{align*}
then $$V_i=\{(\varepsilon_0^{(i)}T,\varepsilon_1^{(i)}T,\varepsilon_2^{(i)}T,\varepsilon_3^{(i)}T,B)\,|\,B=T^2-\eta_{3-i}T, \,T\in\mathbb{C}\},\quad i=1,2,3,$$ which are parameterized by $T(=T_0)\in \mathbb{C}$ only, so we can identify $V_i$ with $\mathbb{C}$. From now on, we simply use the variable $T$ to denote an element of $AP$,  %Note that  $(\mathbb{T}, B)\in AP$  means $T_k= \varepsilon_k^{(j)}T, k=0,1,2,3, B=T^2-c_jT$ for some $j\in \{1,2,3\}$ with $T\in \mathbb{C}$, so we can denote by $T\in AP$ and 
then use the notations $\mathcal{L}(T)$ and $q(z;T)$ to replace $\mathcal{L}(\mathbb{T}, B)$ and $q(z;\mathbb{T}, B)$, respectively. Sometimes, we write $\mathcal{L}_\ell(T)$ and $q_\ell(z;T)$  to emphasize $T\in V_\ell$.

\section{Spectral polynomials}\label{sec-monodromy}

In this section, we will first introduce the classical second symmetry product equation and then derive the definition of  spectral polynomials. With the help of these two tools, we prove Theorem \ref{thm-intro-1}  and Theorem  \ref{thm-intro-Q}. 

\subsection{Second symmetry product equations}
Let $T\in AP$ and $y_1, y_2$ be two solutions of $\mathcal{L}(T)$,   %It is well-known  (see \cite{WW}) that
then their product $F(z)=y_1y_2$ satisfies the following third order equation:
\begin{equation}\label{cc6}
    F^{\prime\prime\prime}-4q F^\prime-2q^\prime F=0,
\end{equation}
which is called the second symmetry product equation of $\mathcal{L}(T)$. 
In particular, if $y_1, y_2$ are chosen as linearly independent common eigenfunctions in the case of (\ref{CompRed}) and $y_1=y_2$ in the case of (\ref{CompNRed}), then $F(z)=y_1y_2$  is an elliptic function.
In fact, any elliptic solution of (\ref{cc6}) can be expressed as a multiple of two common eigenfunctions of $\mathcal{L}(T)$. 
%constructed in this way for any nonsingular point $T\in AP\setminus \{0\}$. 
%We consider the elliptic solutions of (\ref{cc6}). 
Denote by  $\mathrm{ES}(T)$  the space of elliptic solutions  of  (\ref{cc6}) and $\mathrm{d}_T$ the dimension  of  $\mathrm{ES}(T)$.  By the classification of monodromy matrices, we have the following conclusion about  $\mathrm{d}_T$. 

%From the above classification of monodromy matrices,

\begin{proposition}\label{prop1.1}
   Let $T\in AP$
%$$\mathrm{ES}(T)=\left\{y_1y_2\,\mid \, y_1, y_2 \,\, \text{are common eigenfunctions of}\,\, \mathcal{L}(T)\right\}$$
and $F(z)\in \mathrm{ES}(T)$, then $F(z)=y_1y_2$ for some two   common eigenfunctions $y_1, y_2$ of $\mathcal{L}(T)$. Furthermore, we have 
 %The solution space  of (\ref{cc6}) is of dimension 3 and generated by $y_1^2,y_1y_2,y_2^2$.
 %The equation (\ref{cc6}) always possesses an elliptic solution. Moreover, 
%the dimension $\mathrm{d}_T$ of the space of elliptic solutions  of  (\ref{cc6}) satisfies 
\begin{equation*}
\mathrm{d}_T=\left\{\begin{array}{ll} 1 & \begin{array}{l}\text{if either}\,\, \mathcal{L}(T) \,\,\text{is not completely reducible}\\
\text{or}\,\, \mathcal{L}(T) \,\,\text{is completely reducible and}\,\, (s,r)\in  \mathbb{C}^2\setminus(\frac{1}{2}\mathbb{Z})^2\end{array} 
%2 & \text{if}\,\, \mathcal{L}(T) \,\,\text{is completely reducible and}\,\, (r,s)\in \mathbb{C}^2\setminus \left( (\frac{1}{2}\mathbb{Z})^2\cup (\mathbb{C}\setminus\frac{1}{2}\mathbb{Z})^2\right)
%\frac{1}{2}\mathbb{Z}\times (\mathbb{C}\setminus \frac{1}{2}\mathbb{Z})\cup(\mathbb{C}\setminus \frac{1}{2}\mathbb{Z})\times\frac{1}{2}\mathbb{Z}
 \\
\\
3 & \text{if}\,\, \mathcal{L}(T) \,\,\text{is completely reducible and}\,\, (s,r)\in (\frac{1}{2}\mathbb{Z})^2
\end{array}\right.
\end{equation*}
%\begin{enumerate}
%\item If $\mathcal{L}(T)$ is completely reducible and the monodromy data $(r,s)\in (\frac{1}{2}\mathbb{Z})^2$, then the space of elliptic solutions of  (\ref{cc6}) is of dimension 3. 
%\item If $\mathcal{L}(T)$ is completely reducible and the monodromy data $(r,s)\in \mathbb{C}^2\setminus (\frac{1}{2}\mathbb{Z})^2$, then the space of elliptic solutions of  (\ref{cc6}) is of dimension 3. 
%\item If $\mathcal{L}(T)$ is not completely reducible, then the space of elliptic solutions of  (\ref{cc6}) is of dimension 1. 
%\end{enumerate}
\end{proposition}

\begin{proof}
 Let $T\in AP$ and $F(z)\in \mathrm{ES}(T)$.  Define $W(z)$ by 
$$W(z)^2=(F')^2-2FF''+4q F^2.$$
By taking the derivative with respect to $z$, by (\ref{cc6}), it is easy to see that $W(z)\equiv W_0$, a constant independent of $z$. Set 
$$y_\pm(z)=\sqrt{F(z)}\exp\left(\mp\frac{1}{2}\int^z\frac{W_0}{F(z)} dz\right).$$
Then we have $F(z)=y_+(z)y_-(z)$ and 
$$\frac{y'_\pm(z)}{y_\pm(z)}=\frac{F'(z)\mp W_0}{2F(z)}.$$
Thus, 
\begin{align*}
\frac{y''_\pm}{y_\pm}=\left(\frac{y'_\pm}{y_\pm}\right)'+\left(\frac{y'_\pm}{y_\pm}\right)^2&=\frac{FF''}{2F^2}-\frac{F'(F'\mp W_0)}{2F^2}+\frac{(F'\mp W_0)^2}{4F^2}\\
&=\frac{-(F')^2+W_0+2FF''}{4F^2}=\frac{4q F^2}{4F^2}=q.
\end{align*}
This proves $y_\pm(z)$ is a solution of $\mathcal{L}(T)$. Since $F(z)$ is elliptic, it is easy to see that both $y_\pm$ are common eigenfunctions.

Furthermore, if $\mathcal{L}(T)$ is not completely reducible, then $y_1=y_2$ is the unique common eigenfunction, up to a constant multiple. Therefore, $\mathrm{ES}(T)$ is generated by $y_1^2$ and then $\mathrm{d}_T=1$.

If $\mathcal{L}(T)$ is completely reducible and $(s,r)\notin (\frac{1}{2}\mathbb{Z})^2$, then the two linearly independent common eigenfunctions $y_1, y_2$ are unique, up to constant  multiples.  Thus, any elliptic solution $F(z)$ is equal to $y_1y_2$, up to a constant multiple.  Therefore, $\mathrm{d}_T=1$.  

If $\mathcal{L}(T)$ is completely reducible with $(s,r)\in (\frac{1}{2}\mathbb{Z})^2$, this means  all solutions of $\mathcal{L}(T)$ are either periodic (if $s\in \mathbb{Z}$) or antiperiodic (if $s\in \frac{1}{2}+\mathbb{Z}$) in $\omega_1$ direction simultaneously. 
Similarly, the same conclusion holds for the $\omega_2$ direction. This proves $\mathrm{d}_T=3$. 
%
\iffalse 
 This proves the first part. %The seond part is an immediate corollary of the first part. 

Furthermore, let $y_1$ be a common eigenfunction of $\mathcal{L}(T)$ and $y_2$ is picked such that $M_1, M_2$ are expressed in the form of (\ref{CompRed}) or (\ref{CompNRed}). 
%let $y_1(z)$ be a common eigenfunction of $\mathcal{L}(T)$. 

If $\mathcal{L}(T)$ is not completely reducible, then the space of common eigenfunctions is of dimension $1$ and spanned by $y_1(z)$, thus $\mathrm{ES}(T)$ is spanned by $y^2_1$, so $\mathrm{d}_T=1$. 

If $\mathcal{L}(T)$ is completely reducible, then $y_1, y_2$ are two linearly independent common eigenfunctions, thus $y_1^2, y_1y_2, y_2^2$ are linearly independent. 
%
If $(s,r)\in  \mathbb{C}^2\setminus(\frac{1}{2}\mathbb{Z})^2$, then
  at least one common eigenvalue $\lambda_1$ (either $e^{- 2\pi is}$ or $e^{2\pi ir}$) of  $y_1$ is not equal to $\pm 1$ and then the corresponding  common eigenvalue $\lambda_1^{-1}$ of  $y_2$  is also not equal to $\pm 1$. Hence, $y_1^2, y_2^2$ are not elliptic. Since $y_1y_2$ is elliptic, then 
$\mathrm{ES}(T)$ is spanned by $y_1y_2$ and $\mathrm{d}_T=1$. 
%
If  $(s,r)\in  (\frac{1}{2}\mathbb{Z})^2$, then  $y_1^2, y_1y_2, y_2^2$ are elliptic,   %all solutions of $\mathcal{L}(T)$ are common eigenfunctions.  %Pick any two linearly independent solutions $y_1(z), y_2(z)$ of $\mathcal{L}(T)$. 
%It is clear to see that 
 thus $\mathrm{ES}(T)$ is spanned by $y_1^2, y_1y_2, y_2^2$ and $\mathrm{d}_T=3$. 
\fi 
\end{proof}

Next, we try to construct a specific basis for $\mathrm{ES}(T)$. 
Notice that  the local exponents of (\ref{cc6}) at ${w_k}/{2}$ are $-1,\, 1$ and $3$, then any elliptic solution $F(z)$ of  (\ref{cc6})  has simple poles only, and thus it can be written as a linear combination of $\zeta(z-\frac{w_k}{2})$ and $1$. Consequently, we get the following conclusion. 

% \,i.e.,  $F(z)={ \sum_{k=0}^3}d_k\zeta(z-{w_k}/{2})+r$, where $d_k,\ r$ are constants with $\sum_{k=0}^3 d_k=0$.  

\begin{theorem}\label{prop1.2}
 Let $T\in V_\ell$ with $\ell\in\{1,2,3\}$. Then
\begin{equation}\label{eq_07141414}
 F^{(\ell)}(z;T)=-\sum\limits_{k=0}^3\varepsilon_k^{(\ell)}\zeta(z-\frac{\omega_k}{2})+(2T+\eta_{3-\ell}) ,
\end{equation}
%
%$$F^{(\ell)}(z;T):=\sum_{k=0}^{3}d_k^{(\ell)}(T)\zeta(z-\frac{\omega_k}{2})+r^{(\ell)}$$
%with     \begin{align}\label{cc9}
%        d_k^{(\ell)}(T)&=-\varepsilon_k^{(\ell)},\quad  0\leq k\leq 3,\\
%r^{(\ell)}&=2T+\eta_{3-\ell}
%    \end{align}
is an elliptic solution of (\ref{cc6}).   Moreover,
\begin{enumerate}
% is the unique elliptic solution of (\ref{cc6}) up to a constant multiple.
%If $T=0$, $F^{(\ell)}(z;T)$ may not be the unique solution elliptic solution of (\ref{cc6}) up to a constant multiple.
%    where $\varepsilon_0^{(i)}=\varepsilon_i^{(i)}=-\varepsilon_j^{(i)}=-\varepsilon_k^{(i)}=1,\ \{i,j,k\}=\{1,2,3\},\ c_1=\eta_2,\ c_2=\eta_1$ and $c_3=0$. 
\item  if $T\neq 0$, %$T_1^{(i)},T_2^{(i)},T_3^{(i)}\text{ and } T_4^{(i)},   
 then $\mathrm{d}_T=1$ and $\mathrm{ES}(T)$ is spanned by $F^{(\ell)}(z;T)$.    %the space of elliptic solutions of (\ref{cc6}) is of dimension $1$ and generated by $F^{(\ell)}(z;T)$. 
\item if $T=0$, then $\mathrm{d}_T=3$ and $\mathrm{ES}(T)$ is spanned by $F^{(1)}(z;0)$, $F^{(2)}(z;0)$, $F^{(3)}(z;0)$.
%the space of elliptic solutions of (\ref{cc6}) is of  dimension $3$ and generated by $F^{(1)}(z;0), F^{(2)}(z;0), F^{(3)}(z;0)$.
\end{enumerate}
\end{theorem}
%\footnote{\textcolor{blue}{Theorem 3.6里面的计算除了之前那个augmented matrix的rank是用mathematic计算的之外，其他我都有手算。等我周四回北京查一下之前的手稿确认一下，或者再验算一下。\\关于(\ref{sys-2})(\ref{sys-1})，由于要考虑两个未知量不完全一样的system，感觉分开对system标号，后面会清楚一点。}}
\begin{proof}
Let $G^{(\ell)}(z;T)=\sum_{k=0}^3d_k^{(\ell)}\zeta(z-\frac{\omega_k}{2})+r^{(\ell)}$ with $d_k^{(\ell)}, r^{(\ell)}\in \mathbb{C}, k=0,1,2,3$. %First of all,  $G^{(\ell)}(z;T)$ is elliptic if and only if $\sum_{k=0}^3d_k^{(\ell)}=0$. 
\iffalse 
Consider the Laurent expransion of the following function at $z=\frac{\omega_k}{2}\, k=0,1,2,3$:
$$(G^{(\ell)}(z;T))'''-4q_\ell(z;T)(G^{(\ell)}(z;T))'-2   (q_\ell(z;T))'G^{(\ell)}(z;T)$$

If $G^{(\ell)}(z;T)$ is  an elliptic solution of (\ref{cc6}), then the coefficient of $(z-\frac{\omega_k}{2})^{-3}$ term vanishes, which implies $d_k^{(\ell)}, \, k=0,1,2,3,$ and $r^{(\ell)}$ satisfies  the following $4\times 4$ linear system with unknows $d_k^{(\ell)}, \, k=0,1,2,3$:
\begin{equation}\label{sys-2}
        2\varepsilon_k^{(\ell)}Td_k^{(\ell)}+\sum_{k^\prime\neq k}d_{k^\prime}^{(\ell)}\zeta_{kk^\prime }=-r^{(\ell)},\ 0\leq k\leq3.
\end{equation}
\fi 
If $G^{(\ell)}(z;T)$ is  an elliptic solution of (\ref{cc6}), by computing  the coefficient of $(z-\frac{\omega_k}{2})^{-3}$ term of the left hand side of (\ref{cc6}), we have that $d_k^{(\ell)}, \, k=0,1,2,3,$ and $r^{(\ell)}$ satisfies  the following $5\times 5$ linear system with unknowns $d_k^{(\ell)}, \, k=0,1,2,3,\, r^{(\ell)}$:
%  \begin{align}
%        2\varepsilon_k^{(\ell)}Td_k^{(\ell)}+\sum_{k^\prime\neq k}d_{k^\prime}^{(\ell)}\zeta_{kk^\prime }+r^{(\ell)}&=0,\ 0\leq k\leq3.\label{eq07101044}\\
%\sum\limits_{k=0}^3d_k^{(\ell)}&=0 \label{eqn-dk0}
%    \end{align}
\begin{align}
        2\varepsilon_k^{(\ell)}Td_k^{(\ell)}+\sum_{k^\prime\neq k}d_{k^\prime}^{(\ell)}\zeta_{kk^\prime }+r^{(\ell)}&=0,\ 0\leq k\leq3,\label{sys-2}\\
\sum\limits_{k=0}^3d_k^{(\ell)}&=0  \label{sys-1},
    \end{align}
where  (\ref{sys-1}) is due to $G^{(\ell)}(z;T)$ is elliptic.

Denote by 
\begin{equation*}
\Omega^{(\ell)}:=\frac{1}{2}\begin{pmatrix}4T&-\eta_1&-\eta_2&-\eta_3\\\eta_1&4\varepsilon_1^{(\ell)} T&\eta_1-\eta_2&-\eta_2\\ \eta_2&\eta_2-\eta_1&4\varepsilon_2^{(\ell)} T&-\eta_1\\
\eta_3&\eta_2&\eta_1&4\varepsilon_3^{(\ell)} T\end{pmatrix},
\end{equation*}
which is the coefficient matrix of  (\ref{sys-2}) with  unknowns $d_k^{(\ell)}, \, k=0,1,2,3$.  By a direct computation, we obtain the determinant 
$$\det \Omega^{(\ell)}=4T^2(2T-\eta_{3-\ell})(2T+\eta_{3-\ell}).$$%\footnote{\textcolor{red}{$\det \Omega^{(\ell)}$ 和 $\mathrm{rank} \Omega^{(\ell)}$ 的具体计算过程我明天写一下发给您哈。}}

%By a direct computation, we obtain the determinant of the coefficient matrix of  (\ref{sys-2}) is 
%\begin{equation*}
%\frac{1}{16}\left|\begin{array}{cccc}4T&-\eta_1&-\eta_2&-\eta_3\\\eta_1&4\varepsilon_1^{(\ell)} T&\eta_1-\eta_2&-\eta_2\\ \eta_2&\eta_2-\eta_1&4\varepsilon_2^{(\ell)} T&-\eta_1\\
%\eta_3&\eta_2&\eta_1&4\varepsilon_3^{(\ell)} T\end{array}\right|%=4T^2(4T^2-c_i^2)
%=4T^2(2T-\eta_{3-\ell})(2T+\eta_{3-\ell}).
%\end{equation*}

\iffalse 
If $T=\eta_{3-\ell}/2$, by a direct computation, the rank of the coefficient matrix of  (\ref{sys-1})  is $4$, thus the solution space of (\ref{sys-1}) is of dimension 1.
%$d_k^{(\ell)}, \, k=0,1,2,3 $ and $r^{(\ell)}$ uniquely exist up to a constant multiple.\

If $T\neq \eta_{3-\ell}/2$, 
by noticing that 
\begin{equation*}
\sum\limits_{k=0}^3\varepsilon_k^{(\ell)}\left( 2\varepsilon_k^{(\ell)}Td_k^{(\ell)}+\sum_{k^\prime\neq k}d_{k^\prime}^{(\ell)}\zeta_{kk^\prime }+r^{(\ell)}\right)=(2T-\eta_{3-\ell})\sum\limits_{k=0}^3d_k^{(\ell)}.
\end{equation*}
then the last equation $\sum_{k=0}^3d_k^{(\ell)}=0 $ in  (\ref{sys-1})  is redundant. 
%Let $T\neq \eta_{3-\ell}/2$. 
\fi 
%is a polynomial of degree 4 in the $T$ variable. 

If $T\not\in \{0, \pm \eta_{3-\ell}/2\}$, then $d_k^{(\ell)},\, 0\leq k\leq 3,$  are uniquely determined by $r^{(\ell)}$.  Therefore, the solution space of  (\ref{sys-2})  is of dimension 1. Since  $\mathrm{d}_T\geq 1$  by Proposition \ref{prop1.1}, we obtain that $\mathrm{d}_T=1$ and thus any solution of (\ref{sys-2}) provides a solution of (\ref{cc6}). Consequently, (\ref{sys-1}) holds automatically. Therefore, we conclude that if $T\not\in \{0, \pm \eta_{3-\ell}/2\}$, then $\mathrm{d}_T=1$ and 
%\begin{itemize}
%\item[($*$)]$G^{(\ell)}(z;T)$ is an elliptic solution of (\ref{cc6}) if and only if its coefficient $(d_k^{(\ell)}, r_k)$ is a solution of  (\ref{sys-2}).
%\end{itemize}
\begin{equation}\label{iff}
\begin{aligned}
&G^{(\ell)}(z;T) \,\textit{is an elliptic solution of}\, (\ref{cc6})\, \textit{if and only if}\\ 
&\textit{its coefficient }\, 
 (d_k^{(\ell)}, r_k)\, %\text{of}\, G^{(\ell)}(z;T)\,  
\textit{is a solution of}\, (\ref{sys-2}).
\end{aligned}\tag{$*$} 
\end{equation}
Notice that the coefficient of $F^{(\ell)}(z;T)$ satisfies (\ref{sys-2}), then $F^{(\ell)}(z;T)$  is  an elliptic solution of (\ref{cc6})  for  $T\not\in \{0, \pm \eta_{3-\ell}/2\}$.  By the continuity with respect to $T$, we have that 
 $F^{(\ell)}(z;T)$ is an elliptic solution of  (\ref{cc6}) for all $T\in V_\ell$.

If $T=\pm \eta_{3-\ell}/2\neq 0$, by a direct computation, the rank of $\Omega^{(\ell)}$ is $3$, then the rank of the coefficient matrix of  (\ref{sys-2})  is $\geq 3$, thus the dimension of
 the solution space of  (\ref{sys-2})  is $\leq 2$. Hence, $\mathrm{d}_T\leq 2$.  By Proposition \ref{prop1.1},  we have $\mathrm{d}_T\geq 1$, then $\mathrm{d}_T=1$. 

Therefore, $\mathrm{d}_T=1$ for $T\neq 0$. Notice that  $F^{(\ell)}(z;T)$ is an elliptic solution of  (\ref{cc6}) for all $T\in V_\ell$,  then $\mathrm{ES}(T)$ is spanned by $F^{(\ell)}(z;T)$ for all $T\neq 0$.

%Consequently, the space of elliptic solutions of (\ref{cc6}) is of dimension 1 and then generated by the nontrivial elliptic solution $F^{(\ell)}(z;T)$. 
%This proves Theorem \ref{prop1.2}.

 %First of all,  $F^{(\ell)}(z;T)$ with $\ell\in\{1,2,3\}$ is  an elliptic solution of (\ref{cc6}) by a straightforward computation. 

%\textcolor{red}{Notice that the left hand side of (\ref{cc6}) $$(F^{(\ell)}(z;T))'''-4q_\ell(z;T)(F^{(\ell)}(z;T))'-2   (q_\ell(z;T))'F^{(\ell)}(z;T)$$  is a polynomial of $T$ and it is identically equal to $0$ for $T\neq 0$, then it is also $0$ for $T=0$. } Therefore, 
If $T=0$, then %notice that $0\in V_\ell$ for any $\ell\in \{1,2,3\}$,
%Since $\ell_i(0)=(e_i-e_j)(e_i-e_k)\ne 0$, (\ref{cc1}) is completely reducible. 
 $F^{(\ell)}(z;0), %=-\sum_{k=0}^{3}\varepsilon_k^{(\ell)}\zeta(z-w_k/2)+\eta_{3-\ell},
 \, \ell=1,2,3$ are elliptic solutions of the same equation (\ref{cc6}), and they are linearly independent, which implies $\mathrm{d}_T=3$. 
\iffalse 
Let $$k_1F^{(1)}(z;0)+k_2F^{(2)}(z;0)+k_3F^{(3)}(z;0)=0$$ with $k_1, k_2, k_3\in \mathbb{C}$. 
By considering the coefficents of ${(z-\frac{\omega_k}{2})}^{-1}, k=0,1,2,3$ terms, we obtain that $k_1=k_2=k_3=0$. Thus $F^{(1)}(z;0), F^{(2)}(z;0), F^{(3)}(z;0)$ are linearly independent. Since (\ref{cc6}) is a third order ODE, then the space of elliptic solutions of (\ref{cc6}) is of dimension 3.
\fi 
%Now, we assume that $T\neq 0$. 
\end{proof}

Combining Proposition \ref{prop1.1} and Theorem \ref{prop1.2}, we obtain the  monodromy at the singular point $T=0$ stated in the following theorem. 

\begin{theorem}\label{thm-monodromy-1}
  Let $T\in AP$, then $T=0$ if and only if  % i.e.,  $(T_0,\cdots, T_3,B)=(0,\cdots, 0)$, 
the monodromy matrices $M_k$ of $\mathcal{L}(T)$ satisfy $M_k=\varepsilon_k I_{2\times 2}$ with  $\varepsilon_k\in\{\pm 1\}, k=1,2.$
\end{theorem}

\subsection{Spectral polynomials. } %Let $T\in AP$. We now follow the approach of KdV potentials to define the so-called spectral polynomial for  $\mathcal{L}(T)$.

Let $T\in V_\ell$ for some  $\ell\in \{1,2,3\}$. Define % $Q_\ell(T)$ by
\begin{equation}\label{cc11}
    Q_\ell(T)=({F^{(\ell)}}^{\prime})^2-2{F^{(\ell)}}^{\prime\prime}F^{(\ell)}+4q_\ell(z; T)(F^{(\ell)})^2,
\end{equation}
where ${F^{(\ell)}}:={F^{(\ell)}}(z;T)$ and the derivatives are taken with respect to $z$.
By taking a derivative  with respect to $z$, it is easy to see that the RHS of (\ref{cc11}) is independent of $z$. Therefore, $Q_\ell(T)$ is a polynomial of $T$ of degree 4. Throughout the paper, $Q_\ell(T)$ is called \textbf{\textit{the spectral polynomial}} of $\mathcal{L}_\ell(T)$ with $T\in V_\ell$. 
This definition of the spectral polynomial is the same as the case of KdV potentials. %If $q(z)$ is an elliptic KdV potential, then our definition of spectral polynomial is the same as the case of KdV.
\begin{proposition}\label{prop-ell}
 Let  $\ell\in\{1,2,3\}$  and $\{\ell, j, k\}=\{1,2,3\}$, then  % The polynomial $Q_\ell(T)$ is given by
    \begin{equation*}\label{cc12}
  Q_\ell(T)=16(T^2+e_\ell-e_j)(T^2+e_\ell-e_k),%\ 1\leq i\leq 3,
    \end{equation*}
  where $e_i=\wp(w_i/2), i=1,2,3$.
\end{proposition}
\begin{proof}
    Since $Q_\ell(T)$ is independent of $z$, we might analyse the behavior of the right hand side of (\ref{cc11}) near $z=0$ to obtain the proposition. All the computations are straightforward, so we omit the details.
\end{proof}

%\begin{remark}
 Since $e_i\ne e_j$ for $i\ne j$, then $Q_\ell(0)\neq 0$ and  $Q_\ell(T)$ has no multiple root. 
%\end{remark}

%\textcolor{red}{The reason that $\ell_i(T)$ is called spectral because it can be applied to determine whether (\ref{ode-1}) is completely reducible or not.}

\begin{theorem}\label{prop-poly}
    Let $T\in V_\ell$ with $\ell\in \{1,2,3\}$.   We have  $\mathcal{L}_\ell(T)$ is completely reducible if and only if  $Q_\ell(T)\neq0$. % This is why we call 
\end{theorem}

\begin{proof}
Since $Q_\ell(0)\neq 0$ and $\mathcal{L}(0)$ is completely reducible by Theorem 
\ref{thm-monodromy-1}, we only need to prove the statement for $T\neq 0$. 
%Since $T\in V_i\setminus \{0\}$ for some $i\in \{1,2,3\}$, 
Assume that $T\neq 0$, by Theorem \ref{prop1.2}, we have $\mathrm{d}_T=1$ and $\mathrm{V}(T)$ is spanned by $F^{(\ell)}(z;T)$. 
By Proposition \ref{prop1.1}, there are two common eigenfunctions $y_1, y_2$ of $\mathcal{L}_\ell(T)$ such that $F^{(\ell)}(z;T)=y_1y_2$.  Furthermore, from the proof of  Proposition \ref{prop1.1},  we see that  $\mathcal{L}_\ell(T)$  is completely reducible if and only if $y_1, y_2$ are linearly independent.

Denote by $F(z):=F^{(\ell)}(z;T)=y_1y_2$. By simply computation, we have 
    \begin{align*}
        \dfrac{F^\prime(z)}{F(z)}=\dfrac{y_1^\prime}{y_1}+\dfrac{y_2^\prime}{y_2}\quad{\rm and}\quad\dfrac{y_2^\prime}{y_2}-\dfrac{y_1^\prime}{y_1}=\dfrac{W}{F(z)},
    \end{align*}
    where $W=W(y_1,y_2):= y'_2y_1-y_2y'_1$ is the Wronskian of $y_1$ and $y_2$. Hence 
    \begin{equation*}
        \dfrac{1}{2}\left(\dfrac{F^\prime}{F}+\dfrac{W}{F}\right)^\prime=\left(\dfrac{y_2^\prime}{y_2}\right)^\prime=\dfrac{y_2^{\prime\prime}}{y_2}-\left(\dfrac{y_2^\prime}{y_2}\right)^2=q(z)-\dfrac{1}{4}\left(\dfrac{F^\prime}{F}+\dfrac{W}{F}\right)^2,
    \end{equation*}
    which implies 
    \begin{equation*}
        W^2={F^\prime}^2-2F^{\prime\prime}F+4qF^2.
    \end{equation*}
 Comparing to (\ref{cc12}), we have $W^2=Q_\ell(T)$. Hence  $y_1$ and $y_2$ are linearly independent if and only if $Q_\ell(T)\neq 0$.  This proves the theorem. 
\end{proof}

\section{Monodromy data}
\label{sec-data}
 
There are two methods to determine the existence of unitary monodromy.  In this section, we will introduce the one by applying the Hermite-Halphen anastz. In the next section, we will discuss the so-called Baker-Akhizer functions.  The first method can be used to determine the monodromy data for all apparent parameters through the informations of zeros of common eigenfunctions. %Through this pair of functions, we might obtain an efficient tool to determine the existence of unitarizable monodromy matices. 

\subsection{Common eigenfunctions}\label{subsec-commoneigen}
%In Section \ref{intro}, we have introduced the monodromy matrices $M_1, M_2$ (see (\ref{CompRed})(\ref{CompNRed})). 
Let $(\mathbb{T}, B)\in AP$ and $y(z)$ be a common eigenfunction of $\mathcal{L}(\mathbb{T}, B)$.
%a solution $y(z)$ of $\mathcal{L}(T)$ is called a common eigenfunction of $\mathcal{L}(T)$ if 
%$$y(z+\omega_j)=\lambda_j y(z),\qquad j=1,2,$$
%where $\lambda_j$ is an eigenvalue of $M_j, j\in \{1,2\}$.  
We will express $y(z)$  in terms of its zeros.
Let $a$ be a zero of  $y(z)$,  then $a+\omega$ is also a zero for any $\omega\in \Lambda_\tau$. For simplicity, we say the corresponding point $[a]$ on $E_\tau$ is a zero of $y(z)$. 
By noticing the local exponents at each singularity, the square of $y(z)$  is an elliptic function of second kind, thus   %Notice that any common eigenfunction $y(z)$ of $\mathcal{L}(T)$ has \textcolor{red}{exactly two distinct zeros $[a_1], [a_2]$}\footnote{\textcolor{red}{这句话感觉不太准确，$T=0$时可以只有半周期点这一个零点？}} on $E_\tau$, 
$y(z)$ can be expressed as the following so-called Hermite–Halphen ansatz
\begin{equation}\label{commony}
y(z;c,a_1,a_2):=e^{cz}\frac{\sigma(z-a_1)\sigma(z-a_2)}{\left(\prod_{k=0}^3\sigma(z-\frac{\omega_k}{2})\right)^{\frac{1}{2}}},% \ \ a_i\notin \{w_k/2\mid k =0,1,2,3\}.
\end{equation}
where $c, a_1, a_2\in \mathbb{C}$.  It is clearly that $[a_1], [a_2]$ are exactly zeros of $y(z)$.  % but $c$ is undetermined and depends on the representatives of $[a_1], [a_2]$.  In what follows, we need to determine $c$ such that $y(z;c,a_1,a_2)$ is a common eigenfunction of $\mathcal{L}(T)$. 
 %has the expression (\ref{commony}) in terms of its zeros.
%Next we will express the common eigenfunction $y$ in Hermite–Halphen ansatz.  %Let us first return to the general case, i.e., the parameters are not necessary apparent. 
%discuss how to use the zeros $a_1$ and $a_2$ of $y(z,z_0;P_0)$ to compute the monodromy data. %, and will see that consequence of those computations implies the converse of Corollary \ref{cor07111344} also holds. 
%For the simplicity, we write $y(z,z_0;P_0)$ by $y(z)$. 
%Let $c,a_1,a_2\in \mathbb{C}$, % with $[a_1], [a_2]\not\in E_\tau[2]$, where $[z]$ is considered as the point on $E_\tau$, 
%we define 
%Generally, we consider 
%
%where  $c\in \mathbb{C}$ depending on the selection of  representative elements of $[a_1], [a_2]$. Here, $y(z;c,a_1,a_2)$ is called a Hermite–Halphen ansatz of $y(z)$. 

By the transformation law
\begin{equation}\label{trans-law}
\sigma(z+\omega_\ell)=-e^{\eta_\ell(z+\frac{\omega_\ell}{2})}\sigma(z), \quad \ell=1,2,3,
\end{equation}
we have $y(z;c,a_1,a_2)$ is invariant up to a nonzero constant multiple by changing the representatives of $[a_1], [a_2]$ and 
\begin{equation}\label{comm-eigen}
y(z+\omega_\ell;c,a_1,a_2)=e^{c\omega_\ell-\eta_\ell(a_1+a_2-\frac{1}{2}\omega_3)}y(z;c,a_1,a_2).
\end{equation}
Therefore, $y(z;c,a_1,a_2)$ is a solution of $\mathcal{L}(\mathbb{T}, B)$ if and only if $y(z;c,a_1,a_2)$ is a common eigenfunction of $\mathcal{L}(\mathbb{T}, B)$.
%We summarize the above discussion as the first part of the following result. 

We first get a conclusion about the zeros of a common eigenfunction for $(\mathbb{T}, B)\in AP\setminus \{0\}$. 
\begin{lemma}\label{lem-commoneigen}
Let $(\mathbb{T}, B)\in AP$ and $y(z;c,a_1,a_2)$ be a common eigenfunction of $\mathcal{L}(\mathbb{T}, B)$.  If $(\mathbb{T}, B)\neq 0$, then $[a_1], [a_2]\not\in E_\tau[2]$ and $[a_1]\neq  [a_2]$. 
\end{lemma}
\iffalse 
\begin{lemma}\label{lem-commoneigen}
Let $T\in AP$ and $y(z)$ be a common eigenfunction of $\mathcal{L}(T)$.  Up to a nonzero constant multiple,  there exists $c,a_1,a_2\in \mathbb{C}$  such that 
$$y(z)=  y(z;c,a_1,a_2).$$ 
%and thus  $y(z)$  vanishes at $[a_1], [a_2]$. 
 Furthermore,  if $T\neq 0$, then $[a_1], [a_2]\not\in E_\tau[2]$ and $[a_1]\neq  [a_2]$. 
%\begin{enumerate}
%\item if $[a_1], [a_2]\not\in E_\tau[2]$, then $[a_1]\neq [a_2]$ and thus $y(z)$ has exactly two simple zeros $[a_1], [a_2]$ on $E_\tau$. 
%\item if $T\neq 0$, then $[a_1], [a_2]\not\in E_\tau[2]$.
%\item if $T=0$ and either $[a_1]$ or $[a_2]\in E_\tau[2]$, then $[a_1]=[a_2]\in E_\tau[2]$.
%\end{enumerate}
\end{lemma}
\fi

\begin{proof}

Suppose that $[a_1]=[\frac{\omega_k}{2}]$, then we have 
%\textbf{Claim 2:} Let $(\mathbb{T}, B)\in AP$, $k\in \{0,1,2,3\}$ and $c\in \mathbb{C}$, then   
$$y(z;c,a_1,a_2)=y(z;c,k):=e^{c z}\frac{\left(\sigma(z-\frac{\omega_k}{2})\right)^{\frac{3}{2}}}{\left(\prod_{j\neq k}\sigma(z-\frac{\omega_j}{2})\right)^{\frac{1}{2}}}$$  
because the local exponent at $\frac{\omega_k}{2}$ is $\frac{3}{2}$. 
%is a solution of $\mathcal{L}(\mathbb{T}, B)$ if and only if $(\mathbb{T}, B)=(0,0,0,0,0)$.

%\textcolor{red}{这个Claim应该是对的，见Remark \ref{rek-Y}，我们有 $$\overline{Y}_k\setminus {Y}_k=\left\{\{a_1, a_2\}\in \overline{Y}_k\mid a_1\in E_\tau\right\}\cong E_\tau[2],\qquad k=0,1,2,3,$$
%$y(z;c,k)$对应的是$\overline{Y}_k\setminus {Y}_k$中的点。$T=0$对应的是$Y_0$中的点不在$\overline{Y}_0\setminus {Y}_0$中。这个结果和Proposition \ref{pro09101540} （这个是您原来文章的Proposition 1.5）中的结果也是相吻合的。}
%\textbf{Proof of the Claim 2:}  
Notice that   $y:=y(z;c,k)$ is a solution of $\mathcal{L}(\mathbb{T}, B)$ if and only if 
 \begin{equation}\label{eqn-qyck}
q(z;\mathbb{T}, B)=\frac{y''}{y}=\left(\frac{y^\prime}{y}\right)^\prime+\left(\frac{y^\prime}{y}\right)^2.
\end{equation}
%where \begin{equation}\label{eqn-qyck}
%    q(z;\mathbb{T}, B)=\frac{3}{4}\sum_{j=0}^3\wp(z-\frac{\omega_j}{2})+\sum_{j=0}^3T_j\zeta(z-\frac{\omega_j}{2})+B.
%\end{equation}
Since $\zeta(z)={\sigma'(z)}/{\sigma(z)}$ and $\zeta'(z)=-\wp(z)$,
by the definition of $y:=y(z;c,k)$, we have
\begin{align*}
    \frac{y^\prime}{y}=c+\frac{3}{2}\zeta(z-\frac{\omega_k}{2})-\frac{1}{2}\sum_{j\neq k}\zeta(z-\frac{\omega_j}{2}),
    \end{align*}
which is an elliptic function.
Since $q(z;\mathbb{T}, B)$ is  also an elliptic function, 
 then (\ref{eqn-qyck}) holds if and only if the elliptic function %$$H(z):=q(z;\mathbb{T}, B)-\left(\frac{y^\prime}{y}\right)^\prime-\left(\frac{y^\prime}{y}\right)^2\equiv 0,$$
\begin{equation*}%\label{qyc}
\begin{aligned}
%0=&q(z;\mathbb{T}, B)-\frac{y''}{y}\\
H(z):=&q(z;\mathbb{T}, B)-\left(\frac{y^\prime}{y}\right)^\prime-\left(\frac{y^\prime}{y}\right)^2\\
=&\frac{3}{4}\sum_{j=0}^3\wp(z-\frac{\omega_j}{2})+\sum_{j=0}^3T_j\zeta(z-\frac{\omega_j}{2})+B+\frac{3}{2}\wp(z-\frac{\omega_k}{2})\\
&-\frac{1}{2}\sum_{j\neq k}\wp(z-\frac{\omega_j}{2})-\left(c+\frac{3}{2}\zeta(z-\frac{\omega_k}{2})-\frac{1}{2}\sum_{j\neq k}\zeta(z-\frac{\omega_j}{2})\right)^2
    \end{aligned}
\end{equation*}
is identically equal to $0$, 
 i.e., none of $\frac{\omega_k}{2}$'s are poles of $H(z)$ and the constant term of the Laurent expansion of $H(z)$ at $z=0$ is $0$. 

Consider the Laurent expansion of $H(z)$ at $z=\frac{\omega_k}{2}$, there are no $(z-\frac{\omega_k}{2})^{-m}$ terms for $m\geq 2$,
 the coefficients of $(z-\frac{\omega_k}{2})^{-1}$ term gives us 
\begin{align}
T_k&=3(c-\frac{1}{2}\sum_{j\neq k}\zeta_{k j})\label{cTk}\\
T_i&=-(c+\frac{3}{2}\zeta_{ik}-\frac{1}{2}\sum_{j\neq k,i}\zeta_{ij}),%=-(c+2\zeta_{ik}-\frac{1}{2}\sum_{j\neq i}\zeta_{i j}),
\qquad i\neq k.\label{cTi}
\end{align}
%Since 
%\begin{align*}
%T_\ell+\sum_{k\neq \ell}T_k=-\frac{3}{2}\sum_{k\neq \ell}\zeta_{\ell k}-\frac{3}{2}\sum_{k\neq \ell}\zeta_{k\ell}+\frac{1}{2}\sum_{k\neq \ell}\sum_{j\neq \ell, k}\zeta_{k j}=\frac{1}{2}\sum_{k\neq \ell}\sum_{j\neq \ell, k}\zeta_{kj}=0
%\end{align*}
% Finally, the constant term of the Laurent expansion of $H(z)$ at $z=0$  gives us 
%\begin{align}
%B+\sum_{j=1}^3T_j\zeta_{0 j}&=(c-\frac{1}{2}\sum_{j=1}^3\zeta_{0 j})^2=\frac{1}{9}T_0^2\label{BT0}\\
%B+\sum_{j\neq i}T_j\zeta_{i j}&=(c+\frac{3}{2}\zeta_{i0}-\frac{1}{2}\sum_{j\neq 0, i}\zeta_{ij})^2=(c+2\zeta_{i0}-\frac{1}{2}\sum_{j\neq  i}\zeta_{ij})^2=T_i^2\label{BTi}
%\end{align}
%
%Consider the Laurent expansion at $z=\frac{\omega_k}{2}$, 
and the constant term %of $H(z)$ at $z=\frac{\omega_k}{2}$  
gives us
\begin{align}
B+\sum_{j\neq k}T_j\zeta_{k j}&=(c-\frac{1}{2}\sum_{j\neq k}\zeta_{k j})^2=\frac{1}{9}T_k^2\label{BTk}\\
B+\sum_{j\neq i}T_j\zeta_{i j}&=(c+\frac{3}{2}\zeta_{ik}-\frac{1}{2}\sum_{j\neq k, i}\zeta_{ij})^2%=(c+2\zeta_{ik}-\frac{1}{2}\sum_{j\neq  i}\zeta_{ij})^2
=T_i^2, \qquad i\neq k.\label{BTi}
\end{align}
Therefore, $H(z)\equiv 0$ if and only if (\ref{cTk}),(\ref{cTi}),(\ref{BTk}) and (\ref{BTi}) hold. 
Since $(\mathbb{T}, B)\in AP$, by (\ref{Tk}), we have 
\begin{equation}\label{BT}
T_i^2=\sum_{j\neq i} T_j\zeta_{ij}+B, \qquad i=0,1,2,3.
\end{equation}
By (\ref{BTk}) and (\ref{BT}), we obtain $T_k=0$ and then $(\mathbb{T},B)=0$. 
\iffalse 
On the other hand, if $(\mathbb{T},B)=(0,0,0,0,0)$, then (\ref{BTk}), (\ref{BTi}) automatically hold, and  (\ref{cTk}), (\ref{cTi}) are equivalent to 
\begin{equation}\label{cT}
\sum_{j\neq k}\zeta_{k j}
=\sum_{j\neq k,i}\zeta_{ij}-3\zeta_{ik},
\qquad i\neq k,
\end{equation}
which is an identity.  By (\ref{cTk}),  we have $c=\frac{1}{2}\sum_{j\neq k}\zeta_{kj}=\eta_k-\frac{1}{2}\eta_3$. 
\fi 

Finally, if $[a_1]=[a_2]\notin E_\tau[2]$, then $[a_1]$ is a zero of second order of $y(z)$, thus $y(a_1)=y'(a_1)=0$. By the uniqueness of ODE, we have $y(z)\equiv 0$ on $E_\tau$, which is a contradiction. Thus $[a_1]\neq [a_2]$. 
\end{proof}

Next, we give a necessary and sufficient condition on the parameters  such that $y(z;c,a_1,a_2)$ is a solution of $\mathcal{L}(\mathbb{T}, B)$ for some $(\mathbb{T}, B)\in AP$, in the following theorem.  Its proof is a direct computation, which has been carried out in \cite{BG-JAM}.  We present the details here just for the convenience of readers. 
%To prove $y(z)$ of (\ref{ccccc9}) satisfies (\ref{cc1}), we set $G(z)=q(z)-(\left(y'/y\right)^\prime+\left(y'/y\right)^2)$. We can prove that the pole of order of $G$ at $\omega_k/2\leq 1$. Moreover, we obtain that

\iffalse 
\begin{lemma}\label{lem-1}
Let  $(\mathbb{T},B)\in \mathbb{C}^5$ with $\sum_{k=0}^3T_k=0$ and 
 $c,a_1,a_2\in \mathbb{C}$.
Then $y(z;c,a_1,a_2)$ is a solution of $\mathcal{L}(\mathbb{T}, B)$
%the following second order ODE of Fuchsian type
%\begin{equation*} 
 %   y^{\prime\prime}(z)=q(z)y(z),
%\end{equation*}
if and only if  $(\mathbb{T}, B)\in AP$,  $[a_1], [a_2]\not\in E_\tau[2]$, $[a_1]\neq [a_2]$ and the following equations hold:
\begin{align}
%&\zeta(a_2-a_1)+\frac{1}{4}\sum_{k=0}^3\zeta(a_1-\frac{\omega_k}{2})-\frac{1}{4}\sum_{k=0}^3\zeta(a_2-\frac{\omega_k}{2})=0,\\
 & 2\zeta(a_2-a_1)+\zeta(2a_1)-\zeta(2a_2)=0,\label{eqn-a}\\
&c=-T_k+\sum_{i=1}^2\zeta(a_i-\frac{\omega_k}{2})+\eta_k-\frac{1}{2}\eta_3,\label{eqn-c}\\
&B=-3\wp(a_i-a_j)+
\frac{3}{4}\sum_{\ell=0}^3\wp(a_i-\frac{\omega_\ell}{2})-\sum_{\ell=0}^3T_\ell\zeta(a_i-\frac{\omega_\ell}{2})\label{eqn-B},
\end{align}
with $k=0,1,2,3,\,\{i, j\}=\{1,2\}$. 
\end{lemma}
%\textcolor{blue}{林老师好，我把原来的这个Lemma还放在这儿，便于您对比。证明中有个关键的地方涉及$(\mathbb{T}, B)\in AP$，反向我验证是需要有这个前提，我在证明中用红色标出了，您看一下对不对。}
\fi

\begin{theorem}\label{lem-1}
Let $(\mathbb{T},B)\in AP$ and 
 $c,a_1,a_2\in \mathbb{C}$
%Let $c, a_1, a_2\in \mathbb{C}$ 
with $[a_1], [a_2]\notin E_\tau[2]$ and $[a_1]\neq [a_2]$. 
If $y(z;c,a_1, a_2)$ is a solution of $\mathcal{L}(\mathbb{T}, B)$, then the following holds:
\begin{align}
%&\zeta(a_2-a_1)+\frac{1}{4}\sum_{k=0}^3\zeta(a_1-\frac{\omega_k}{2})-\frac{1}{4}\sum_{k=0}^3\zeta(a_2-\frac{\omega_k}{2})=0,\\
 & 2\zeta(a_2-a_1)+\zeta(2a_1)-\zeta(2a_2)=0,\label{eqn-a}\\
&c=\zeta(a_2-a_1)+\frac{1}{2}\sum_{k=0}^3\zeta(a_1-\frac{\omega_k}{2}),\label{eqn-c-1}\\
&T_k=\zeta(a_1-a_2)-\zeta(2a_1)+\sum_{i=1}^2\zeta(a_i-\frac{\omega_k}{2})+\eta_k,\quad k=0,1,2,3,\label{eqn-c}\\
&B=-3\wp(a_1-a_2)+
\frac{3}{4}\sum_{\ell=0}^3\wp(a_1-\frac{\omega_\ell}{2})-\sum_{\ell=0}^3T_\ell\zeta(a_1-\frac{\omega_\ell}{2}).\label{eqn-B}
\end{align}
%with $k=0,1,2,3$. 
%\footnote{\textcolor{blue}{把(\ref{eqn-c})放到Conversely之后，定理的叙述确实更好，但不利于后面Lemma 4.5的应用。Lemma 4.5中是假设$y(z;c,a_1, a_2)$是解之后求$\mathbb{T}$，若把(\ref{eqn-c})放在第一部分可以直接引用。}}

 Conversely, if $a_1,a_2\in \mathbb{C}$
%Let $c, a_1, a_2\in \mathbb{C}$ 
with $[a_1], [a_2]\notin E_\tau[2]$ and $[a_1]\neq [a_2]$  satisfy (\ref{eqn-a}) and $c$ is given by (\ref{eqn-c-1}),
%\begin{equation}\label{eqn-c-1}
%c=\zeta(a_2-a_1)+\frac{1}{2}\sum_{k=0}^3\zeta(a_1-\frac{\omega_k}{2}),
%\end{equation}
then $y(z;c,a_1,a_2)$ is a solution of $\mathcal{L}(\mathbb{T}, B)$ for some  $(\mathbb{T},B)\in AP$, where $\mathbb{T}$ is given by (\ref{eqn-c}) and $B$ is given by (\ref{eqn-B}). 
\end{theorem}

\begin{remark}\label{rek-a-T}
Note that $a_1, a_2$ play the same role in (\ref{eqn-a}) and adding a period to $a_1$ or $a_2$ does not change  (\ref{eqn-a}).  Combing with  (\ref{eqn-a}), it is clear to see that $T_k$ is invariant under the exchange of $a_1,a_2$ and the period shifts of $a_1, a_2$.
%
%For fixed $a_2$, the left hand side of (\ref{eqn-a}) is an elliptic function of $a_1$, whose poles are $[a_2]$ and $[\omega_j/2], j=0,1,2,3$. Furthermore, if $c$ of (\ref{eqn-c-1}) is substituted in (\ref{eqn-c}), then $T_k$ is not changed if $a_i$ is replaced by $a_i+\omega$ with $\omega\in \Lambda_\tau$. 
\end{remark}

\begin{proof}
%First of all, by Lemma \ref{lem-commoneigen}, the conditions  $[a_1], [a_2]\not\in E_\tau[2]$ and $[a_1]\neq  [a_2]$ are natural. Hence, we always assume $[a_1], [a_2]\not\in E_\tau[2]$ and $[a_1]\neq  [a_2]$ in the following proof. 
%
Notice that $y:=y(z;c,a_1,a_2)$ is a solution of $\mathcal{L}(\mathbb{T}, B)$ if and only if \begin{equation}\label{eqn-qy}
q(z;\mathbb{T}, B)=\frac{y''}{y}=\left(\frac{y^\prime}{y}\right)^\prime+\left(\frac{y^\prime}{y}\right)^2\end{equation}
Since $\zeta(z)={\sigma'(z)}/{\sigma(z)}$ and $\zeta'(z)=-\wp(z)$, by the definition of $y:=y(z;c,a_1,a_2)$, we have
\begin{align*}
    \dfrac{y^\prime}{y}=c+\zeta(z-a_1)+\zeta(z-a_2)-\dfrac{1}{2}\sum_{k=0}^3\zeta(z-\dfrac{\omega_k}{2}),
    \end{align*}
    which is an elliptic function. Recall that \begin{equation*}
    q(z;\mathbb{T}, B)=\frac{3}{4}\sum_{k=0}^3\wp(z-\frac{\omega_k}{2})+\sum_{k=0}^3T_k\zeta(z-\frac{\omega_k}{2})+B.
\end{equation*}
%with $(\mathbb{T}, B)\in AP$. 
Notice that if $(\mathbb{T}, B)\in\mathbb{C}^5$ satisfies $\sum_{k=0}^3 T_k=0$, then  $q(z):=q(z;\mathbb{T}, B)$ is  an elliptic function, 
 thus  (\ref{eqn-qy}) holds if and only if the elliptic function $$G(z):=q(z)-\left(\frac{y^\prime}{y}\right)^\prime-\left(\frac{y^\prime}{y}\right)^2\equiv 0,$$
 i.e., none of $\frac{\omega_k}{2}$'s and $a_i$'s are poles of $G(z)$ and the constant term of the Laurent expansion of $G(z)$ at $z=0$ is $0$, because $[a_1], [a_2]\not\in E_\tau[2]$.
%    \begin{equation}\label{eq07111731}
%     \begin{aligned}
%  \dfrac{y^{\prime\prime}}{y}=\left(\dfrac{y^\prime}{y}\right)^\prime+\left(\dfrac{y^\prime}{y}\right)^2=\left(\sum_{i=1}^2(-\wp(z-a_i))+\dfrac{1}{2}\sum_{k=0}^3\wp(z-\dfrac{\omega_k}{2})\right)\\+\left(c+\sum_{i=1}^2\zeta(z-a_i)-\dfrac{1}{2}\sum_{k=0}^3\zeta(z-\dfrac{\omega_k}{2})\right)^2.
%\end{aligned}
%\end{equation}

Recall the laurent expansion of $\wp(z)$ and $\zeta(z)$ at $z=0$ as follows:
\begin{align*}
\wp(z)&=\frac{1}{z^2}+\sum\limits_{n=1}^\infty b_n z^{2n},\\
\zeta(z)&=\frac{1}{z}-\sum\limits_{n=1}^\infty \frac{b_n}{2n+1} z^{2n+1},
\end{align*}
where $b_1={g_2}/{20}$ and $g_2=2(e_1^2+e_2^2+e_3^2)$. 

Firstly, we consider the laurent expansion at $z={\omega_k}/{2}, k\in\{0,1,2,3\}$. Notice that 
\begin{align*}
q(z)&=\frac{\frac{3}{4}}{(z-\frac{\omega_k}{2})^2}+\frac{T_k}{z-\frac{\omega_k}{2}}+\left(\sum_{\ell\neq k}T_\ell\zeta_{k\ell}+B\right)+O((z-\frac{\omega_k}{2}))\\
\frac{y'}{y}&=\frac{-\frac{1}{2}}{z-\frac{\omega_k}{2}}+\left(c+\sum_{i=1}^2\zeta(\frac{\omega_k}{2}-a_i)-\frac{1}{2}\sum_{\ell\neq k}\zeta_{k\ell}\right)\\&-(z-\frac{\omega_k}{2})\sum_{i=1}^2\wp(\frac{\omega_k}{2}-a_i)+O((z-\frac{\omega_k}{2})^2).
\end{align*} 
\iffalse 
and then 
\begin{align*}
\left(\frac{y'}{y}\right)'&=\frac{\frac{1}{2}}{(z-\frac{\omega_k}{2})^2}-\sum_{i=1}^2\wp(\frac{\omega_k}{2}-a_i)+O((z-\frac{\omega_k}{2})),\\
\left(\frac{y'}{y}\right)^2&=\frac{\frac{1}{4}}{(z-\frac{\omega_k}{2})^2}-\left(c+\sum_{i=1}^2\zeta(\frac{\omega_k}{2}-a_i)-\frac{1}{2}\sum_{\ell\neq k}\zeta_{k\ell}\right)\frac{1}{z-\frac{\omega_k}{2}}\\&+\sum_{i=1}^2\wp(\frac{\omega_k}{2}-a_i)
+\left(c+\sum_{i=1}^2\zeta(\frac{\omega_k}{2}-a_i)-\frac{1}{2}\sum_{\ell\neq k}\zeta_{k\ell}\right)^2+O((z-\frac{\omega_k}{2})),
\end{align*}
\fi
In the expansion of $G(z)$, the coefficient of $(z-\omega_k/2)^{-2}$ automatically vanishes; the coefficient of $(z-\omega_k/2)^{-1}$ is zero if and only if the following (\ref{eqn-1}) holds and the constant term vanishes if and only if the following (\ref{eqn-2}) holds.
%then $G(z)\equiv 0$ gives us 
%\begin{equation}%\label{eqn-1}
\begin{align}
T_k&=-\left(c+\sum_{i=1}^2\zeta(\frac{\omega_k}{2}-a_i)-\frac{1}{2}\sum_{\ell\neq k}\zeta_{k\ell}\right)\label{eqn-1}\\
\sum_{\ell\neq k}T_\ell\zeta_{k\ell}+B&=\left(c+\sum_{i=1}^2\zeta(\frac{\omega_k}{2}-a_i)-\frac{1}{2}\sum_{\ell\neq k}\zeta_{k\ell}\right)^2.\label{eqn-2} 
\end{align}
%\end{equation}
%\begin{equation}\label{eqn-1-1}
%%\begin{aligned}
%%T_k&=-\left(c+\sum_{i=1}^2\zeta(\frac{\omega_k}{2}-a_i)-\frac{1}{2}\sum_{\ell\neq k}\zeta_{k\ell}\right),\\
%\sum_{\ell\neq k}T_\ell\zeta_{k\ell}+B&=\left(c+\sum_{i=1}^2\zeta(\frac{\omega_k}{2}-a_i)-\frac{1}{2}\sum_{\ell\neq k}\zeta_{k\ell}\right)^2. 
%%\end{aligned}
%\end{equation}
Note that  $\frac{1}{2}\sum_{\ell\neq k}\zeta_{k\ell}=\eta_k-\frac{1}{2}\eta_3$ with $\eta_0:=0$, (\ref{eqn-1}) and (\ref{eqn-2})  are equivalent to 
\begin{align}
c&=-T_k+\sum_{i=1}^2\zeta(a_i-\frac{\omega_k}{2})+\eta_k-\frac{1}{2}\eta_3,\label{eqn-ck}\\
B&=T_k^2-\sum_{\ell\neq k}T_\ell\zeta_{k\ell}\label{eqn-Bk},
\end{align}
%Let there are $4$ equations in (\ref{eqn-ck}) and (\ref{eqn-Bk}). 
with $k\in \{0,1,2,3\}$.  Here, (\ref{eqn-Bk})  holds automatically due to $(\mathbb{T}, B)\in AP$.

Next, we consider the laurent expansion at $z=a_i, i\in\{1,2\}$. Denote by $\{i,j\}=\{1,2\}$, We have 
\begin{align*}
q(z)&=\frac{3}{4}\sum_{\ell=0}^3\wp(a_i-\frac{\omega_\ell}{2})+\sum_{\ell=0}^3T_\ell\zeta(a_i-\frac{\omega_\ell}{2})+B+O\left((z-a_i)\right)\\
\frac{y'}{y}&=\frac{1}{z-a_i}+\left(c+\zeta(a_i-a_j)-\frac{1}{2}\sum_{\ell=0}^3\zeta(a_i-\frac{\omega_\ell}{2})\right)\\
&+(z-a_i)\left(-\wp(a_i-a_j)+\frac{1}{2}\sum_{\ell=0}^3\wp(a_i-\frac{\omega_\ell}{2})\right)+O\left((z-a_i)^2\right).
\end{align*} 
\iffalse 
and then 
\begin{align*}
\left(\frac{y'}{y}\right)'&=\frac{-1}{(z-\frac{\omega_k}{2})^2}+\left(-\wp(a_i-a_j)+\frac{1}{2}\sum_{\ell=0}^3\wp(a_i-\frac{\omega_\ell}{2})\right)+O\left((z-a_i)\right),\\
\left(\frac{y'}{y}\right)^2&=\frac{1}{(z-\frac{\omega_k}{2})^2}+2\left(c+\zeta(a_i-a_j)-\frac{1}{2}\sum_{\ell=0}^3\zeta(a_i-\frac{\omega_\ell}{2})\right)\frac{1}{z-a_i}\\
&+2\left(-\wp(a_i-a_j)+\frac{1}{2}\sum_{\ell=0}^3\wp(a_i-\frac{\omega_\ell}{2})\right)\\
&+\left(c+\zeta(a_i-a_j)-\frac{1}{2}\sum_{\ell=0}^3\zeta(a_i-\frac{\omega_\ell}{2})\right)^2
+O\left((z-a_i)\right),
\end{align*}
\fi 
%then $G(z)\equiv 0$ gives us
In the expansion of $G(z)$, there are always no $(z-a_i)^{-2}$ term; the coefficient of $(z-a_i)^{-1}$ is zero if and only if the following (\ref{eqn-3}) holds and the constant term vanishes if and only if the following (\ref{eqn-4}) holds. 
\begin{equation}\label{eqn-3}
c+\zeta(a_i-a_j)-\frac{1}{2}\sum_{\ell=0}^3\zeta(a_i-\frac{\omega_\ell}{2})=0,
\end{equation}
\begin{equation}\label{eqn-4}
\begin{aligned}
&\frac{3}{4}\sum_{\ell=0}^3\wp(a_i-\frac{\omega_\ell}{2})+\sum_{\ell=0}^3T_\ell\zeta(a_i-\frac{\omega_\ell}{2})+B\\
=&3\left(-\wp(a_i-a_j)+\frac{1}{2}\sum_{\ell=0}^3\wp(a_i-\frac{\omega_\ell}{2})\right)+\left(c+\zeta(a_i-a_j)-\frac{1}{2}\sum_{\ell=0}^3\zeta(a_i-\frac{\omega_\ell}{2})\right)^2,
\end{aligned}
\end{equation}
which are equivalent to 
\begin{align}
c&=-\zeta(a_i-a_j)+\frac{1}{2}\sum_{\ell=0}^3\zeta(a_i-\frac{\omega_\ell}{2}),\label{eqn-ca}\\
B&=-3\wp(a_i-a_j)+
\frac{3}{4}\sum_{\ell=0}^3\wp(a_i-\frac{\omega_\ell}{2})-\sum_{\ell=0}^3T_\ell\zeta(a_i-\frac{\omega_\ell}{2})\label{eqn-Ba},
\end{align}
with $\{i, j\}=\{1,2\}$. 

Let $(\mathbb{T},B)\in AP$ and 
 $c,a_1,a_2\in \mathbb{C}$
%Let $c, a_1, a_2\in \mathbb{C}$ 
with $[a_1], [a_2]\notin E_\tau[2]$ and $[a_1]\neq [a_2]$. 
If $y(z;c,a_1, a_2)$ is a solution of $\mathcal{L}(\mathbb{T}, B)$, then
%$y(z;c,a_1,a_2)$ is a solution of $\mathcal{L}(\mathbb{T}, B)$ 
$G(z)\equiv 0$ which implies (\ref{eqn-ck}),  (\ref{eqn-Bk}), (\ref{eqn-ca}) and  (\ref{eqn-Ba}) hold. %Clearly, (\ref{eqn-ck}) is the same as (\ref{eqn-c}),  (\ref{eqn-Bk}) is the same as (\ref{Tk}),  and (\ref{eqn-Ba}) is the same as (\ref{eqn-B}).   %\textcolor{red}{This proves the first part. 到这儿(\ref{eqn-a})还没得到，必要性没有得到}
Clearly, (\ref{eqn-ca}) implies  (\ref{eqn-c-1}) and (\ref{eqn-Ba}) implies  (\ref{eqn-B}). Therefore, we only need to show  (\ref{eqn-a})  and (\ref{eqn-c}) hold. 
\iffalse 
On the one hand, 
note that 
$$\sum_{k=0}^3\sum_{\ell\neq k}\zeta_{k\ell}= \sum_{k=0}^3 (\eta_k-\frac{1}{2}\eta_3)=0,$$
by (\ref{eqn-ck}) and (\ref{eqn-ca}), we have 
$$\sum_{k=0}^3T_k=\sum_{k=0}^3\sum_{i=1}^2\zeta(a_i-\frac{\omega_k}{2})-\sum_{i=1}^2\sum_{k=0}^3\zeta(a_i-\frac{\omega_k}{2})=0,$$
which is the prerequisite of $\mathcal{L}(\mathbb{T}, B)$. 
This together with (\ref{eqn-Bk}) are the apparent conditions, i.e., $(\mathbb{T}, B)=(T_0, \cdots, T_3, B)\in AP.$ 
%%Without loss of generality, we assume that $({\bf T}, B)\in V_i$ for some $i\in \{1,2,3\}$, then 
%$T_k=\varepsilon_k^{(i)}T, B=T^2-c_iT$ with $T\in \mathbb{C}$. 
\fi 
%
%Conversely, if $y(z)$ is given by (\ref{commony}), then we want to find $c$ such that $y(z)$ might be a solution of $\mathcal{L}(\mathbb{T}, B)$ for some $(\mathbb{T}, B)\in AP$. First we observe that, by the two identity $i\in \{1,2\}$ of (\ref{eqn-ca}) together,
For this purpose, we introduce the following identity
    \begin{equation}\label{ccccc16}
        \zeta(2z)-\frac{1}{2}\sum_{k=0}^3\zeta(z-\frac{\omega_k}{2})=\frac{1}{2}\eta_3.
\end{equation}
Indeed, since $\zeta^\prime(z)=-\wp(z)$, the left hand side of the above identity is a constant independent of $z$, which follows from the following identity:
    \begin{equation}\label{p2z}
        4\wp(2z)=\sum_{k=0}^3\wp(z-\frac{\omega_k}{2}).
    \end{equation}
The constant is $\eta_3/2$ can be obtained by analysing the behavior of the left hand side near $z=0$.

By (\ref{eqn-ca}), we have that $a_1, a_2$ satisfy 
\begin{equation}\label{ccccc15}
    \zeta(a_2-a_1)+\frac{1}{4}\sum_{k=0}^3\zeta(a_1-\frac{\omega_k}{2})-\frac{1}{4}\sum_{k=0}^3\zeta(a_2-\frac{\omega_k}{2})=0,
\end{equation}
which is equivalent to  (\ref{eqn-a}) by  (\ref{ccccc16}).

By substituting  (\ref{eqn-c-1}) into  (\ref{eqn-ck}), we have
$$T_k=\zeta(a_1-a_2)-\frac{1}{2}\sum_{k=0}^3\zeta(a_1-\frac{\omega_k}{2})+\sum_{i=1}^2\zeta(a_i-\frac{\omega_k}{2})+\eta_k-\frac{1}{2}\eta_3,$$
Applying (\ref{ccccc16}), we obtain  (\ref{eqn-c}).
This proves the first part of the lemma. 

Conversely, let $a_1,a_2\in \mathbb{C}$
%Let $c, a_1, a_2\in \mathbb{C}$ 
with $[a_1], [a_2]\notin E_\tau[2]$ and $[a_1]\neq [a_2]$  satisfy (\ref{eqn-a}), we want to find $c$ such that $y(z; c, a_1, a_2)$ might be a solution of $\mathcal{L}(\mathbb{T}, B)$ for some $(\mathbb{T}, B)\in AP$. By above analysis, we have  (\ref{eqn-a}) is equivalent to 
  (\ref{ccccc15}), then we set 
\iffalse 
On the other hand, if $(\mathbb{T}, B)\in AP$ and (\ref{eqn-a}), (\ref{eqn-c}), (\ref{eqn-B}) hold, then (\ref{eqn-ck}), (\ref{eqn-Bk}),(\ref{eqn-Ba}) hold and $\sum_{j=0}^3T_k=0$.  We only need to show 
(\ref{eqn-ca}) holds. 
Notice that  $\sum_{j=0}^3T_k=0$, adding the four equations in (\ref{eqn-ck}) and combining  (\ref{ccccc15}), we obtain (\ref{eqn-ca}).

Therefore, $y(z;c,a_1,a_2)$ is a solution of  $\mathcal{L}(\mathbb{T}, B)$  if and only if $(\mathbb{T}, B)\in AP$ and  (\ref{eqn-a}), (\ref{eqn-c}),  (\ref{eqn-B}) hold. 

Conversely, if $a_1, a_2$ satisfies (\ref{eqn-a}), where we assume $[a_1]\neq [a_2] $ and $[a_i]\notin E_\tau[2]$, then (\ref{ccccc15}) holds. Set 
\fi 
\begin{equation}\label{def-c}
c=\zeta(a_2-a_1)+\frac{1}{2}\sum_{k=0}^3\zeta(a_1-\frac{\omega_k}{2})=\zeta(a_1-a_2)+\frac{1}{2}\sum_{k=0}^3\zeta(a_2-\frac{\omega_k}{2})\end{equation}
and  $T_k$ by (\ref{eqn-c}), $B$ by (\ref{eqn-B}).  Combing with (\ref{ccccc16}), we have (\ref{eqn-ck}) and (\ref{eqn-ca}) hold, which implies $G(z)$ is smooth at $a_1, a_2, \frac{\omega_k}{2}, \, k=0,1,2,3$.  By  (\ref{eqn-c}), (\ref{ccccc16}) and  (\ref{def-c}), we have 
\begin{align*}
\sum_{k=0}^3T_k&=4\zeta(a_1-a_2)-4\zeta(2a_1)+\sum_{k=0}^3\sum_{i=1}^2\zeta(a_i-\frac{\omega_k}{2})+\sum_{k=0}^3\eta_k\\
&=4\zeta(a_1-a_2)-4\zeta(2a_1)+2c-2\zeta(a_2-a_1)+2c-2\zeta(a_1-a_2)+2\eta_3\\
%&=4c+4\zeta(a_1-a_2)+4(\frac{1}{2}\eta_3-\zeta(2a_1))\\
&=4c+4\zeta(a_1-a_2)-2\sum_{k=0}^3\zeta(a_1-\frac{\omega_k}{2})=0,
%\sum_{k=0}^3T_k&=-4c+\sum_{k=0}^3\sum_{i=1}^2\zeta(a_i-\frac{\omega_k}{2})+\sum_{k=0}^3\eta_k-2\eta_3\\
%&=-4c+\sum_{k=0}^3\sum_{i=1}^2\zeta(a_i-\frac{\omega_k}{2})=\zeta(a_1-a_2)+\zeta(a_2-a_1)=0,
\end{align*}
which implies $q(z;\mathbb{T}, B)$ is an elliptic function and then $G(z)$ is elliptic. 
Hence, $G(z)\equiv $ a constant.  By the definition of $B$, we have $G(a_1)=0$, then $G(z)\equiv 0$, which implies (\ref{eqn-Bk}) holds for all $k$, thus $(\mathbb{T}, B)\in AP$ and $y(z;c,a_1, a_2)$ is a solution of $\mathcal{L}(\mathbb{T}, B)$. This proves the converse part. 
\iffalse 
\textcolor{red}{ Further, by taking a sum of (\ref{eqn-c}), we have
$$\sum_{k=0}^3T_k=\sum_{k=0}^3\sum_{i=1}^2\zeta(a_i-\frac{\omega_k}{2})-\sum_{i=1}^2\sum_{k=0}^3\zeta(a_i-\frac{\omega_k}{2})=0,$$
 which implies $G(z)$ is a smooth elliptic function. Thus $G(z)\equiv $ a constant. Since  (\ref{eqn-Ba})  holds at $i=1$ and $j=2$, we have $G(a_1)=0$ which implies $G(z)\equiv 0$ and in turn, implies (\ref{eqn-Ba}) holds for $i=2$ and $j=1$, also  by (\ref{eqn-ca}) $G(z)$ is smooth at $z=\frac{\omega_k}{2}, k=0,1,2,3$. Thus $G$ is smooth on $\mathbb{C}$.  Therefore, (\ref{eqn-Bk}) holds. %Clearly  (\ref{eqn-Ba}) and $\sum_{k=0}^3T_k=0$ yields $(\mathbb{T}, B)\in AP$. 
This proves the second part. }
\textcolor{blue}{红色部分不太明白，$B$ 是通过(\ref{eqn-B}) with $i=1, j=2$来定义的，感觉没法说明(\ref{eqn-B}) with $i=2, j=1$也是成立的，这个条件感觉必须加上。}
\fi 
\end{proof}

From Theorem \ref{lem-1}, we see that any two points $a_1, a_2\in E_\tau\setminus E_\tau[2]$ with $a_1\neq a_2$ satisfying (\ref{eqn-a}) 
can be chosen as the zeros of some common eigenfunction. Note that $a_1, a_2$ play the same role in (\ref{eqn-a}) and adding a period to $a_1$ or $a_2$ does not change  (\ref{eqn-a}), we collect these zeros and define
$$Y:=\left\{\{a_1, a_2\}\in \mathrm{Sym}^2(E_\tau \setminus E_\tau[2])\,\mid\, a_1\neq a_2 \,\,\text{satisfy}\,\, (\ref{eqn-a}) \right\},$$
\iffalse 
$$Y:=\left\{\{a_1, a_2\}\in \mathrm{Sym}^2(E_\tau \setminus E_\tau[2])\setminus \Delta_2\,\mid\, a_1, a_2 \,\,\text{satisfy}\,\, (\ref{eqn-a}) \right\},$$ %2\zeta(a_2-a_1)+\zeta(2a_1)-\zeta(2a_2)=0
where $\Delta_2:=\left\{\{z,z\}\,\mid\, z\in E_\tau\right\}$. % which is the set of all possible zero sets of a common eigenfunction. 
\fi 
 Furthermore, we get the following conclusion so that we can decompose $Y$ into $4$ components. 
\begin{lemma}\label{lem-2}
Let $a_1, a_2\in E_\tau\setminus E_\tau[2]$ with $a_1\neq a_2$. Then $a_1, a_2$ satisfy (\ref{eqn-a}) if and only if $a_2=-a_1$ or $a_2=a_1-\frac{\omega_i}{2}$ for some $i\in \{1,2,3\}$ on $E_\tau$. 
\end{lemma}

\begin{proof}
Let
$$g(z)=2\zeta(z-a_1)+\zeta(2a_1)-\zeta(2z),$$
which is clearly  an elliptic function with  five simple poles $a_1,\,\frac{\omega_j}{2}, 0\leq j\leq 3$, then $g(z)$ should has five zeros counting with multiplicity. 
By a direct computation, we obtain that  $a_1-\frac{\omega_i}{2}, i=1,2,3$ and $-a_1$ are zeros of $g(z)$, where $a_1-\frac{\omega_i}{2}, i=1,2,3$ are distinct, but we may have $-a_1= a_1-\frac{\omega_i}{2}$ for some $i\in \{1,2,3\}$. 
Note that the order of zero at $-a_1$ is at least 2, because $-a_1$  is also a zero of $g'(z)=-2\wp(z-a_1)+2\wp(2z)$,  then $-a_1, a_1-\frac{\omega_i}{2}, i=1,2,3$ are four zeros of $g(z)$ counting with multiplicity.
Denote by $a_1^\ast$  the remaining zero of $g(z)$, we have
    \begin{equation*}
        a_1^\ast-a_1+3a_1-\frac{1}{2}\sum_{i=1}^3\omega_i=a_1+\frac{1}{2}\sum_{j=0}^3\omega_j,
    \end{equation*}
    which implies $a_1^\ast=-a_1$. This proves the Lemma. % which also means the  zero order of $g(z)$ at $-a_1$ is 2.
\end{proof} 

\iffalse 
\begin{remark}\label{rek-Y}
In the proof of Lemma \ref{lem-2}, we  actually prove that $-a_1$ is a double zero of $g$. 
\end{remark}
\fi

By Lemma \ref{lem-2}, we have $Y=\cup_{j=0}^3Y_j$, where  
 \begin{align*}
Y_0:=&\left\{\{a_1, a_2\}\in \mathrm{Sym}^2(E_\tau \setminus E_\tau[2])\mid  a_1+a_2=0\right\}\\
Y_\ell:=&\left\{\{a_1, a_2\}\in \mathrm{Sym}^2(E_\tau \setminus E_\tau[2])\mid  a_1-a_2=\frac{\omega_\ell}{2}\right\}, \quad \ell=1,2,3.
%Y:=&\left\{\{z_1, z_2\}\in \mathrm{Sym}^2(E_\tau \setminus E_\tau[2])\mid  (\ref{ccccc17}) \,\, \text{holds}.\right\}.
\end{align*}
Clearly, $Y_j\cap Y_k=\emptyset$ for $j\neq k$ and $j,k\in \{1,2,3\}$. 
Denoted by $\ell\in \{1,2,3\}$ and $\{\ell, j, k\}=\{1,2,3\}$, we have 
\begin{equation}\label{Y0Yell}
Y_0\cap Y_\ell=\left\{\{\frac{\omega_\ell}{4}, -\frac{\omega_\ell}{4}\},\,\{\frac{\omega_\ell}{4}+\frac{\omega_j}{2}, -\frac{\omega_\ell}{4}+\frac{\omega_j}{2}\} \right\}.
\end{equation}
%,\,\{\frac{\omega_\ell}{4}+\frac{\omega_k}{2}, -\frac{\omega_\ell}{4}+\frac{\omega_k}{2}\}
%To our surprise, there is a correspondence between the decomposition $Y=\cup_{j=0}^3 Y_j$ and the decomposition $AP=\cup_{j=1}^3 V_j$
To our surprise, this classification of zero sets of common eigenfunctions exactly corresponds to the decomposition of the apparent space.  In the following lemma, we investigate this surprising correspondence. 

\begin{lemma}\label{lem-YV}
Let $(\mathbb{T}, B)\in AP$ and 
$y(z; c, a_1, a_2)$ be a common eigenfunction of $\mathcal{L}(\mathbb{T}, B)$ with  $c,a_1, a_2\in \mathbb{C}$ and $[a_1], [a_2]\notin E_\tau[2]$, then the followings hold:
\begin{enumerate}
\item  $\{[a_1], [a_2]\}\in Y_0$ if and only if $(\mathbb{T}, B)=0$. Moreover,  if $a_2=-a_1$, then 
\begin{equation}\label{c-0}
c=-\frac{1}{2}\eta_3.
\end{equation}
\item  $\{[a_1], [a_2]\}\in Y_\ell\setminus Y_0$ if and only if $(\mathbb{T}, B)\in V_\ell\setminus \{0\}$, where $\ell\in \{1,2,3\}$.  Moreover, let $\{[a_1], [a_2]\}\in Y_\ell$, we have 
\begin{align}
c&=c_{a_1}:=\zeta(2a_1)-\frac{1}{2}\eta_\ell-\frac{1}{2}\eta_3, \quad \text{if}\,\,\, a_2=a_1-\frac{\omega_\ell}{2}, \label{c-ell}\\
T_0&=\zeta(a_1)+\zeta(a_1-\frac{\omega_\ell}{2})-\zeta(2a_1)+\frac{1}{2}\eta_\ell. \label{T-ell}
\end{align}
%
%Let $\ell\in \{1,2,3\}$ and $\{[a_1], [a_2]\}\in Y_\ell$.  If 
%Moreover, 
%\begin{enumerate}
%\item[(i)] $\{[a_1], [a_2]\}\in Y_\ell\setminus Y_0$ if and only if $(\mathbb{T}, B)\in V_\ell\setminus \{0\}$, where $\ell\in \{1,2,3\}$.  
%\item[(ii)] $\{[a_1], [a_2]\}\in Y_\ell\cap  Y_0$ if and only if $(\mathbb{T}, B)=0$, where $\ell\in \{1,2,3\}$.  
%\end{enumerate}
\end{enumerate}
\end{lemma}

\begin{remark}\label{rek-T0}
Notice that the right hand side of (\ref{T-ell}) is invariant if $a_1$ is replaced by $a_1-\frac{\omega_\ell}{2}$ or $a_1+\omega$ with $\omega\in\Lambda_\tau$, thus $T_0$ can be viewed as a function of $\{[a_1], [a_2]\}$. 
\end{remark}
\begin{remark}\label{rek-yell}
Let $a_1\in \mathbb{C}$ with $[a_1]\notin E_\tau[2]$. 
By Theorem \ref{lem-1}, we have $y_0(z; a_1):=y(z;-\frac{1}{2}\eta_1, a_1, -a_1)$ is a common eigenfunction of $\mathcal{L}(0)$ and 
$y_\ell(z; a_1):=y(z;c_{a_1}, a_1, a_1-\frac{\omega_\ell}{2})$ is a common eigenfunction of $\mathcal{L}_\ell(T_0)$ for $\ell\in \{1,2,3\}$, where $T_0$ is given by (\ref{T-ell}).
\end{remark}

\iffalse 
\begin{theorem}
Let $T\in AP$ and 
$y(z)$ be a common eigenfunction of $\mathcal{L}(T)$.  %such that $y^2(z)$ is an elliptic function of the second kind on $E_\tau$, 
If $y(z)$ has two zeros $a_1, a_2\notin E_\tau[2]$, then $T=0$ if and only if $\{a_1, a_2\}\in Y_0$. Moreover, $y(z)=y(z;c,a_1,a_2)$ %has the expression (\ref{commony}) 
with $c=-\frac{1}{2}\eta_3$.
%Let $T\in AP$ and $y(z)$ be a common eigenfunction of $\mathcal{L}(T)$.  Up to a nonzero constant multiple,  there exists $c,a_1,a_2\in \mathbb{C}$  such that $$y(z)=  y(z;c,a_1,a_2).$$ 
%and thus  $y(z)$  vanishes at $[a_1], [a_2]$.  Furthermore,  if $T\neq 0$, then $[a_1], [a_2]\not\in E_\tau[2]$ and $[a_1]\neq  [a_2]$. 
\end{theorem}
\fi

\begin{proof}
\iffalse 
tells us that given a point $\{[a_1], [a_2]\}\in Y$, there exists $c$ given by (\ref{eqn-c}) such that $y(z;c, a_1, a_2)$ is a solution of $\mathcal{L}(\mathbb{T}, B)$, where $\mathbb{T}$  is given by (\ref{eqn-c}) and $B$ is given by (\ref{eqn-B}) for some $i\in \{1,2\}$. 

Furthermore, if $c$ of (\ref{eqn-c-1}) is substituted in (\ref{eqn-c}), then $T_k$ is not changed if $a_i$ is replaced by $a_i+\omega$ with $\omega\in \Lambda_\tau$. 
\fi 
Let $(\mathbb{T}, B)\in AP$ and 
$y(z;c, a_1, a_2)$ be a common eigenfunction of $\mathcal{L}(\mathbb{T}, B)$.  By Theorem \ref{lem-1}, $\{[a_1], [a_2]\}\in Y$ and $\mathbb{T}$ is determined by (\ref{eqn-c}). By Remark \ref{rek-a-T},  (\ref{eqn-c})  is independent of the choice of the representatives of $[a_1], [a_2]$, so we can fix representatives for $[a_1], [a_2]$ in computing $\mathbb{T}$.

%$\zeta(a_1-a_2)-\frac{1}{2}\sum_{k=0}^3\zeta(a_1-\frac{\omega_k}{2})+\sum_{i=1}^2\zeta(a_i-\frac{\omega_k}{2})+\eta_k-\frac{1}{2}\eta_3$

\begin{enumerate}
\item If $\{[a_1],[a_2]\}\in Y_0$, we can set $a_2=-a_1$ to compute $\mathbb{T}$. 
%Substituting $c$ into the left hand side of 
By (\ref{eqn-c}), we have 
\begin{align*}
T_k=\zeta(2a_1)-\zeta(2a_1)+\zeta(a_1-\frac{\omega_k}{2})-\zeta(a_1+\frac{\omega_k}{2})+\eta_k=0
%\\&=\zeta(2z_1)-\frac{1}{2}\sum_{j=0}^3\zeta(z_1-\frac{\omega_j}{2})-\frac{1}{2}\eta_3=0.
\end{align*}
for $k\in \{0,1,2,3\}$.  Hence, $(\mathbb{T}, B)=0$. % where $\zeta(a_1+\frac{\omega_k}{2})=\zeta(a_1-\frac{\omega_k}{2})+\eta_k$ is used. 
%It proves the sufficient part. 

Conversely, if $(\mathbb{T}, B)=0$, by Theorem  \ref{thm-monodromy-1},  $y^2(z)$ is an elliptic function with exactly four simple poles at $[\omega_k/2], \, k=0,1,2,3$ and exactly  two double zeros at  $[a_1], [a_2]$ on $E_\tau$. Therefore, we have $2[a_1]+2[a_2]=\sum_{k=0}^3[\omega_k/2]=0$ on $E_\tau$, i.e., $\{[a_1],[a_2]\}\in Y_0$. %  This proves the necessary part. 
%\end{proof}

Finally, if $a_2=-a_1$, by (\ref{eqn-c-1}), we have 
$$c=-\zeta(2a_1)+\frac{1}{2}\sum_{j=0}^3\zeta(a_1-\frac{\omega_j}{2})=-\frac{1}{2}\eta_3.$$

%Next, we consider the zeros of the common eigenfunction lie in $Y_\ell$ with $\ell\in \{1,2,3\}$. 
\iffalse 
\begin{lemma}\label{lem-appV}
Let $T\in AP$ and 
$y(z)$ be a common eigenfunction of $\mathcal{L}(T)$.  %such that $y^2(z)$ is an elliptic function of the second kind on $E_\tau$, 
If $y(z)$ has two zeros $[a_1], [a_2]\notin E_\tau[2]$, then $\{[a_1], [a_2]\}\in Y_\ell$ for some $\ell\in \{1,2,3\}$ implies $T\in V_\ell$.  
%Let $\{a_1, a_2\}\in Y_\ell$. Then the apparent parameter $(\mathbb{T}, B)$ in Theorem \ref{lem-1} is contained in $V_\ell$. 
\end{lemma}
\fi 
%\begin{proof}

\item  Let $\ell\in\{1,2,3\}$. If $\{[a_1],[a_2]\}\in Y_\ell$, we can set $a_2=a_1-\omega_\ell/2$ to compute $\mathbb{T}$. By (\ref{eqn-c}), we have 
%Applying (\ref{eqn-c}), (\ref{eqn-c-1}) and $a_1-a_2=\omega_\ell/2$, we have 
%$$T_k=\zeta(a_1-a_2)-\zeta(2a_1)+\sum_{i=1}^2\zeta(a_i-\frac{\omega_k}{2})+\eta_k$$
$$T_k=\zeta(\frac{\omega_\ell}{2})-\zeta(2a_1)+\zeta(a_1-\frac{\omega_k}{2})+\zeta(a_1-\frac{\omega_\ell}{2}-\frac{\omega_k}{2})+\eta_k,$$
then
\begin{align*}
T_\ell&=\zeta(\frac{\omega_\ell}{2})-\zeta(2a_1)+\zeta(a_1-\frac{\omega_\ell}{2})+\zeta(a_1-\omega_\ell)+\eta_\ell\\
&=\frac{1}{2}\eta_\ell-\zeta(2a_1)+\zeta(a_1-\frac{\omega_\ell}{2})+\zeta(a_1)=T_0.
%T_\ell-T_0=&\zeta(a_1-\frac{\omega_\ell}{2})+\zeta(a_1-{\omega_\ell})-\zeta(a_1)-\zeta(a_1-\frac{\omega_\ell}{2})+\eta_\ell\\
%=&\zeta(a_1-{\omega_\ell})-\zeta(a_1)+\eta_\ell=0.
\end{align*}
Since $(\mathbb{T}, B)\in AP$, then $T_\ell=T_0$ implies $(\mathbb{T}, B)\in V_\ell$.  Furthermore, if $\{[a_1],[a_2]\}\notin Y_0$ and $(\mathbb{T}, B)=0$, by Part (1), we have $\{[a_1],[a_2]\}\in Y_0$, which is a contradiction.
\iffalse 
and
\begin{align*}
T_k+T_0&=2\zeta(\frac{\omega_\ell}{2})-2\zeta(2a_1)+\zeta(a_1-{\omega_k})+\zeta(a_1-\frac{\omega_\ell}{2}-\frac{\omega_k}{2})+\eta_k\\
&+\zeta(a_1)+\zeta(a_1-\frac{\omega_\ell}{2})
\end{align*}
for $k\neq \ell$.
Applying (\ref{ccccc16}), we have 
$$T_k+T_0=0.$$
For example, if $\ell=1$ and $k=2$, then 
\begin{align*}
T_2+T_0&=\eta_1-2\zeta(2a_1)+\zeta(a_1-\frac{\omega_2}{2})+\zeta(a_1-\frac{\omega_3}{2})+\eta_2\\
&+\zeta(a_1)+\zeta(a_1-\frac{\omega_1}{2})\\
&=\eta_1-\eta_3+\eta_2=0,
\end{align*}
where  (\ref{ccccc16}) is applied. Similarly, if  $\ell=1$ and $k=3$, we have 
\begin{align*}
T_3+T_0&=\eta_1-2\zeta(2a_1)+\zeta(a_1-\frac{\omega_3}{2})+\zeta(a_1-\frac{\omega_1}{2}-\frac{\omega_3}{2})+\eta_3\\
&+\zeta(a_1)+\zeta(a_1-\frac{\omega_1}{2})\\
&=\eta_1-2\zeta(2a_1)+\sum_{j=0}^3\zeta(a_1-\frac{\omega_j}{2})-\eta_1+\eta_3=0,
\end{align*}
by (\ref{ccccc16}) again. This proves $\ell=1$ completely. For $\ell=2,3$, a similar strategy can be applied to prove the theorem. 
\fi 

On the other hand, if $(\mathbb{T}, B)\in V_\ell\setminus \{0\}$, i.e., $T_\ell=T_0$,
by (\ref{eqn-c}), we have 
% then 
%\begin{align*}
%%T_k=\zeta(a_1-a_2)-\zeta(2a_1)+\sum_{i=1}^2\zeta(a_i-\frac{\omega_k}{2})+\eta_k
%T_\ell&=\zeta(a_1-a_2)-\zeta(2a_1)+\sum_{i=1}^2\zeta(a_i-\frac{\omega_\ell}{2})+\eta_\ell\\
%&=\zeta(a_1-a_2)-\zeta(2a_1)+\sum_{i=1}^2\zeta(a_i)=T_0,
%\end{align*}
%thus 
\begin{equation}\label{a2}
\sum_{i=1}^2\zeta(a_i-\frac{\omega_\ell}{2})-\sum_{i=1}^2\zeta(a_i)+\eta_\ell\equiv 0.\end{equation}
Let $\varphi(z)=\zeta(z-\frac{\omega_\ell}{2})-\zeta(z)+\zeta(a_1-\frac{\omega_\ell}{2})-\zeta(a_1)+\eta_\ell$, which is an elliptic function with two simple poles at $[0], [\frac{\omega_\ell}{2}]$.  Notice that $[a_1-\frac{\omega_\ell}{2}]$ is a zero of $\varphi$, then the other zero is $[0]+[\frac{\omega_\ell}{2}]-[a_1-\frac{\omega_\ell}{2}]=[-a_1]$. By (\ref{a2}), we obtain that $[a_2]=[a_1-\frac{\omega_\ell}{2}]$ or $[-a_1]$, i.e., $\{[a_1], [a_2]\}\in Y_\ell$ or $Y_0$. Furthermore, if $\{[a_1], [a_2]\}\in Y_0$, by Part (1), we have $(\mathbb{T}, B)=0$, which is a contradiction. 

Finally,  if $a_2=a_1-\frac{\omega_\ell}{2}$, by (\ref{eqn-c-1}) and (\ref{ccccc16})  , we have 
$$c=-\zeta(\frac{\omega_\ell}{2})+\frac{1}{2}\sum_{j=0}^3\zeta(a_1-\frac{\omega_j}{2})=\zeta(2a_1)-\frac{1}{2}\eta_\ell-\frac{1}{2}\eta_3.$$
\end{enumerate}
\end{proof}

%Theorem \ref{lem-YV} builds a correspondence between the decomposition of $Y$ and the decomposition of $AP$. However, in order to determine the monodromy data of $\mathcal{L}(T)$ with $T\in AP$ in term of the zeros $\{[a_1],[a_2]\}\in Y$, we need more specific relation between $T$ and $\{[a_1],[a_2]\}$.

%and (\ref{T-ell})
Due to the value of $c$ 
%Notice that $c=\zeta(a_2-a_1)+\frac{1}{2}\sum_{j=0}^3\zeta(a_1-\frac{\omega_j}{2})$ 
depends on the representives of $[a_1]$ and $[a_2]$,  we make the following convention regarding the selection of representatives. 
%
%Let $\{[a_1], [a_2]\}\in Y$, we can define $c$ by (\ref{eqn-c-1}). 
%However, the value of $c$ 
%depends on the representives of $[a_1]$ and $[a_2]$.  Thus, 
For any  $\{[a_1],[a_2]\}\in Y$, 
we fix the representative $a_2$ of $[a_2]$ as long as the representative $a_1$ of $[a_1]$ is chosen in the following way: 
\begin{itemize}
\item   $a_2=-a_1$ if $\{[a_1], [a_2]\}\in Y_0$;
\item $a_2=a_1-\frac{\omega_\ell}{2}$  if  $\{[a_1], [a_2]\}\in Y_\ell$ with $\ell\in \{1,2,3\}$.
\end{itemize}
Under this convention, $c$ is given by (\ref{c-0}) and (\ref{c-ell}). 
For simplification of notations,  we denote by 
$\underline{a}:=\{[a], [-a]\}\in Y_0$ and $\underline{a}:=\{[a], [a-\frac{\omega_\ell}{2}]\}\in Y_\ell$  for $\ell\in \{1,2,3\}$, where $a\in \mathbb{C}$.
%we simply write $\underline{a_1}:=\{[a_1], [a_2]\}\in Y_0$ if $a_2=-a_1$ on $\mathbb{C}$ and  $\underline{a_1}:=\{[a_1], [a_2]\}\in Y_\ell$ if $a_2=a_1-\frac{\omega_\ell}{2}$ on $\mathbb{C}$ for $\ell\in \{1,2,3\}$.

Lemma \ref{lem-YV} builds a correspondence between the decomposition of $Y$ and the decomposition of $AP$. However, in order to determine the monodromy data of $\mathcal{L}(T)$ with $T\in AP$ in terms of the zeros $\underline{a}\in Y$, we need more specific relation between $T$ and $\underline{a}$.  By Lemma \ref{lem-YV}(1), we only need to make a concrete correspondence between $Y_\ell$ and $V_\ell$ for $\ell\in \{1,2,3\}$. 

 Let $\ell\in \{1,2,3\}$.  By Remark \ref{rek-T0}, 
$T=T_0\in V_\ell$ can be expressed in terms of  $\underline{a}\in Y_\ell$ by (\ref{T-ell}), which induces us to define
\iffalse 
 Denote the right hand side of  (\ref{T-ell}) by
$$f(a)=\zeta(a)+\zeta(a-\frac{\omega_\ell}{2})-\zeta(2a)+\frac{1}{2}\eta_\ell, \qquad a\in \mathbb{C},$$
 which is an elliptic function with $4$ simple poles at $E_\tau[2]$. By a direct computation, we obtain that $f(z)$ has $4$ simples zeros at $\{\pm \frac{\omega_\ell}{4}, \pm \frac{\omega_\ell}{4}+\frac{\omega_j}{2}\}$ with $j\in \{1,2,3\}\setminus\{\ell\}$, which exactly corresponds to $Y_\ell\cap Y_0$. 
Notice that $f(a)=f(a-\frac{\omega_\ell}{2})$, $f$ can be viewed as a function on $Y_\ell$, 
\fi
 the following map
\begin{equation*}
\begin{aligned}
T^{(\ell)}: Y_\ell &\to V_\ell\\
\underline{a}&\mapsto T^{(\ell)}(\underline{a}):=\zeta(a)+\zeta(a-\frac{\omega_\ell}{2})-\zeta(2a)+\frac{1}{2}\eta_\ell
\end{aligned}.
\end{equation*}
 This map will enclose the relationship between the apparent parameter $T\in V_\ell$ and the zeros $\underline{a}\in Y_\ell$ of common eigenfunctions of $\mathcal{L}(T)$.

\begin{theorem}\label{mapT}
Let $\ell\in \{1,2,3\}$ and $\{j,k,\ell\}=\{1,2,3\}$. The map $T^{(\ell)}$
 is a double covering map with four ramification points 
$\underline{\pm \frac{\omega_j}{4}}, \underline{\pm \frac{\omega_k}{4}}$ on $Y_\ell$. Moreover, the corresponding four branch points are exactly four roots of $Q_\ell(T)$. More precisely,
$$\left(T^{(\ell)}\left(\underline{\pm \frac{\omega_j}{4}}\right)\right)^2=e_j-e_\ell\quad \text{and}\quad \left(T^{(\ell)}\left(\underline{\pm \frac{\omega_k}{4}}\right)\right)^2=e_k-e_\ell.$$
%and the corresponding four branched values are $\{T\in V_\ell\mid Q_\ell(T)=0\}$.
%and their parameters are exactly $\pm\sqrt{e_j-e_\ell}$, which are the four roots of $Q_\ell(T)$.  
\end{theorem}

%\textcolor{blue}{查到如下关于``branch points"的术语：Let $\pi: Y\to X$ be a holomorphic map between Riemann surfaces / algebraic curves, $p\in Y$ and $q=f(p)$. Denote the ramification index at $p$ by $\nu_p$. If $\nu_p=1$,  we say that $f$  is unramified at $p$. Otherwise, $p$  is a ramification point of  $f$, in which case we say that $q$ is a branch point.  A branch value is exactly the same object as a branch point.  The ``branch point" is standard geometric term when treating $X$ as a curve/topological space, and the ``branch value" is predominantly used in function theory, complex analysis, or spectral problems.}

\begin{proof}
Let
\begin{equation}\label{fell}
f_\ell(a)=\zeta(a)+\zeta(a-\frac{\omega_\ell}{2})-\zeta(2a)+\frac{1}{2}\eta_\ell, \qquad a\in \mathbb{C},\end{equation}
 which is an elliptic function with $4$ simple poles at $E_\tau[2]$. By a direct computation, we obtain that $f_\ell(z)$ has $4$ simples zeros at $\{\pm \frac{\omega_\ell}{4}, \pm \frac{\omega_\ell}{4}+\frac{\omega_j}{2}\}$ with $j\in \{1,2,3\}\setminus\{\ell\}$, which exactly corresponds to $Y_\ell\cap Y_0$, thus $(T^{(\ell)})^{-1}(0)=Y_0\cap Y_\ell$. 
Notice that $f_\ell(a)=f_\ell(a-\frac{\omega_\ell}{2})$, $f_\ell$ can be viewed as a function on $Y_\ell$, thus $T^{(\ell)}$ is well-defined and is a covering map of degree $2$. 

Applying (\ref{ccccc16}), we do the following computation:

\begin{align*}
&T^{(\ell)}(\underline{\frac{\omega_j}{2}-a})-T^{(\ell)}(\underline{a})\\
%=&\zeta(\frac{\omega_j}{2}-a)+\zeta(\frac{\omega_j}{2}-a-\frac{\omega_\ell}{2})-\zeta({\omega_j}-2a)-\zeta(a)-\zeta(a-\frac{\omega_\ell}{2})+\zeta(2a)\\
%=&-\zeta(a-\frac{\omega_j}{2})-\zeta(a-\frac{\omega_j}{2}+\frac{\omega_\ell}{2})-\zeta(a)-\zeta(a-\frac{\omega_\ell}{2})+2\zeta(2a)-\eta_j\\
=&-\sum_{i=0}^3\zeta(a-\frac{\omega_i}{2})+2\zeta(2a)-\eta_j+\zeta(a-\frac{\omega_k}{2})-\zeta(a-\frac{\omega_j}{2}+\frac{\omega_\ell}{2})\\
=&\frac{1}{2}(\eta_\ell+\eta_j+\eta_k)-\eta_j+\zeta(a-\frac{\omega_k}{2})-\zeta(a-\frac{\omega_j}{2}+\frac{\omega_\ell}{2})=0.
\end{align*}
Therefore, 
\begin{equation}\label{preimage}
(T^{(\ell)})^{-1}\left(T^{(\ell)}(\underline{a})\right)=\left\{\underline{a}, \underline{\frac{\omega_j}{2}-a}\right\}.
\end{equation}
Notice that $\underline{a}=\underline{\frac{\omega_j}{2}-a}$ if and only if 
$\underline{a}\in \left\{\underline{\pm \frac{\omega_j}{4}}, \underline{\pm \frac{\omega_k}{4}}\right\}$. Hence, $T^{(\ell)}$ has four branched points: 
$\underline{\pm \frac{\omega_j}{4}}, \underline{\pm \frac{\omega_k}{4}}$. 

Finally, we compute the branch  points. Notice that $T^{(\ell)} (\underline{\pm a})=\pm T^{(\ell)} (\underline{a})$, we only need to compute 
% In particular, we have 
%$$T^{(\ell)}(\omega_k/4)=\pm \sqrt{e_k-e_\ell}.$$
%Indeed, by (\ref{T0ell}), we have 
\begin{align*}
T^{(\ell)}(\underline{\frac{\omega_k}{4}})&=\zeta(\frac{\omega_k}{4})+\zeta(\frac{\omega_k}{4}-\frac{\omega_\ell}{2})-\zeta(\frac{\omega_k}{2})+\frac{1}{2}\eta_\ell\\
%&=\zeta(\frac{\omega_k}{4})+\zeta(\frac{\omega_k}{4}-\frac{\omega_\ell}{2})-\zeta(\frac{\omega_k}{2}-\frac{\omega_\ell}{2})+\zeta(\frac{\omega_k}{2}-\frac{\omega_\ell}{2})-\frac{1}{2}\eta_k+\frac{1}{2}\eta_\ell\\
&=\zeta(\frac{\omega_k}{4})+\zeta(\frac{\omega_k}{4}-\frac{\omega_\ell}{2})-\zeta(\frac{\omega_k}{2}-\frac{\omega_\ell}{2})\\
&=-\frac{1}{2}\frac{\wp'(\frac{\omega_k}{4})-\wp'(\frac{\omega_k}{4}-\frac{\omega_\ell}{2})}{\wp(\frac{\omega_k}{4})-\wp(\frac{\omega_k}{4}-\frac{\omega_\ell}{2})}.
\end{align*}
Using the addition-theorem, we obtain that 
\begin{equation}\label{Tomega}
\begin{aligned}
\left(T^{(\ell)}(\underline{\frac{\omega_k}{4}})\right)^2&=\frac{1}{4}\left(\frac{\wp'(\frac{\omega_k}{4})-\wp'(\frac{\omega_j}{4}-\frac{\omega_\ell}{2})}{\wp(\frac{\omega_k}{4})-\wp(\frac{\omega_k}{4}-\frac{\omega_\ell}{2})}\right)^2\\
&=\wp(\frac{\omega_k}{2}-\frac{\omega_\ell}{2})+\wp(\frac{\omega_k}{4})+\wp(\frac{\omega_k}{4}-\frac{\omega_\ell}{2}).
\end{aligned}
\end{equation}
%Note the following identity 
%$$4\wp(2z)=\sum_{i=0}^3\wp(z-\frac{\omega_i}{2}),$$
By the identity  (\ref{p2z}), 
we have  
\begin{align}
e_k=\wp(2\cdot\frac{\omega_k}{4})%&=\frac{1}{4}\sum_{i=0}^3\wp(\frac{\omega_k}{4}-\frac{\omega_i}{2})\\
&=\frac{1}{4}\left(\wp(\frac{\omega_k}{4})+\wp(\frac{\omega_k}{4}-\frac{\omega_\ell}{2})+\wp(-\frac{\omega_k}{4})+\wp(\frac{\omega_k}{4}-\frac{\omega_j}{2})\right)\notag\\
&=\frac{1}{2}\left(\wp(\frac{\omega_k}{4})+\wp(\frac{\omega_k}{4}-\frac{\omega_\ell}{2})\right),\label{ek}
\end{align}
because  $\frac{\omega_k}{4}-\frac{\omega_j}{2}=-(\frac{\omega_k}{4}-\frac{\omega_\ell}{2})$ on $E_\tau$.  Therefore, 
\begin{align*}
\left(T^{(\ell)}\left(\underline{\frac{\omega_k}{4}}\right)\right)^2
%&=\wp(\frac{\omega_k}{2}-\frac{\omega_\ell}{2})+\wp(\frac{\omega_k}{4})+\wp(\frac{\omega_k}{4}-\frac{\omega_\ell}{2})\\
&=e_j+2e_k=e_k-e_\ell.
\end{align*}
\iffalse 
It is clear to see that $T^{(\ell)}$ is an elliptic function with only simple poles at $a=0$ and $a=\frac{\omega_\ell}{2}$. Thus, for any $T\in \mathbb{C}$, $(T^{(\ell)})^{-1}(T)$ consists of only two points counting multiplicity. 

Suppose $T^{(\ell)}(\{a_1, a_2\})=T$ for some $T\in \mathbb{C}$. Then $(T^{(\ell)})^{-1}(T)=\left\{\{a_1, a_2\}, \{\frac{\omega_j}{2}-a_1, \frac{\omega_j}{2}-a_2\}\right\}.$ The branch points occurs at either $a_1=\frac{\omega_j}{2}-a_1$ or $a_1=\frac{\omega_j}{2}-a_2$ on $E_\tau$ for $j\neq \ell$, which yields $a_1\in \{\pm \frac{\omega_j}{4}, \pm \frac{\omega_k}{4}\}$, $\{j, k, \ell\}=\{1,2,3\}$. Then 
$y(z; a_1^{(i)})=e^{c_iz}\frac{\sigma(z-a_1^{(i)})\sigma(z-a_2^{(i)})}{(\prod_{k=0}^3\sigma(z-\frac{\omega_k}{2}))^{\frac{1}{2}}}$ is the only solution of $\mathcal{L}_\ell(T_i^{(\ell)})$, which square is an elliptic function of second class, where $T_i^{(\ell)}=T^{(\ell)}(a_1^{(i)})$, $i=1,2,3$, that is $\mathcal{L}_\ell(T_i^{(\ell)})$ is not completely reducible. Thus $Q_\ell(T_i^{(\ell)})=0$. 
\fi
This proves the theorem. 
\end{proof}

Now, we are ready to calculate the monodromy data of $\mathcal{L}_\ell(T)$ and prove Theorem \ref{thm-intro-data}. 

%\noindent\textit{\bf Proof of Theorem \ref{thm-intro-data}. } 
\begin{proof}(Proof of Theorem \ref{thm-intro-data}.)
Let $\ell\in \{1,2,3\}$ and $T\in V_\ell$. 
%Let $\ell\in \{1,2,3\}, \{\ell,j,k\}=\{1,2,3\}$ and $a\in \mathbb{C}$ with $[a]\not\in E_\tau[2]$. 

If $Q_\ell(T)\neq 0$, then 
%$\underline{a}\not\in \left\{\underline{\pm \frac{\omega_j}{4}},\, \underline{\pm \frac{\omega_k}{4}}\right\}$, by Theorem \ref{mapT}, we have $Q_\ell(T^{(\ell)}(\underline{a}))\neq 0$, i.e.,
 $\mathcal{L}_\ell(T)$ is completely reducible.  Hence, there are two linearly independent common eigenfunctions.
Denote by $(T^{(\ell)})^{-1}(T)=\left\{\underline{a}, \underline{\frac{\omega_j}{2}-a}\right\}\subseteq Y_\ell$ with $a\in \mathbb{C}$ and $[a]\notin E_\tau[2]$.  By Remark \ref{rek-yell}, $y_\ell(z;a)$ and $y_\ell(z;\frac{\omega_j}{2}-a)$ 
%\begin{align*}
%y_\ell(z;a)&=e^{c_az}\frac{\sigma(z-a)\sigma(z-a+\frac{\omega_\ell}{2})}{(\prod_{k=0}^3\sigma(z-\frac{\omega_k}{2}))^{\frac{1}{2}}},\\
%y_\ell(z;\frac{\omega_j}{2}-a)&=e^{c_{\frac{\omega_j}{2}-a}z}\frac{\sigma(z+a-\frac{\omega_j}{2})\sigma(z+a-\frac{\omega_k}{2})}{(\prod_{k=0}^3\sigma(z-\frac{\omega_k}{2}))^{\frac{1}{2}}}
%\end{align*}  
are two common eigenfunctions of the same ODE $\mathcal{L}_\ell(T)$. By considering  the zeros of $y_\ell(z;a)$ and $y_\ell(z;\frac{\omega_j}{2}-a)$, we obtain that they are linearly independent. Hence, the monodromy data $(s,r)$ can be determined by $y_\ell(z;a)$ and $y_\ell(z;\frac{\omega_j}{2}-a)$. 
By (\ref{comm-eigen}), we have 
\begin{align*}
y_\ell(z+1;a)&=e^{c_a-\eta_1(2a-\frac{1}{2}\omega_\ell-\frac{1}{2}\omega_3)}y_\ell(z;a),\\
y_\ell(z+\tau;a)&=e^{\tau c_a-\eta_2(2a-\frac{1}{2}\omega_\ell-\frac{1}{2}\omega_3)}y_\ell(z;a).
\end{align*}
where $c_a=\zeta(2a)-\frac{1}{2}\eta_\ell-\frac{1}{2}\eta_3$ is defined in (\ref{c-ell}).
The Legendre identity $\tau\eta_1-\eta_2=2\pi i$ implies there is a unique $(s,r)\in \mathbb{C}^2$ such that 
\begin{equation}
\left\{
\begin{aligned}\label{rs-new}
r+s\tau&=2a-\frac{1}{2}\omega_\ell-\frac{1}{2}\omega_3\\
r\eta_1+s\eta_2&=\zeta(2a)-\frac{1}{2}\eta_\ell-\frac{1}{2}\eta_3
\end{aligned}\right.,
\end{equation}
then 
%By (\ref{CompRed}), we define $(s,r)$  by
$$y_\ell(z+1;a)=e^{-2\pi i s}y_\ell(z;a)\quad \text{and}\quad 
y_\ell(z+\tau;a)=e^{2\pi i r}y_\ell(z;a).$$
%Furthermore, by a direct computation, we have 
%\begin{align*}
%y_\ell(z+1;\frac{\omega_j}{2}-a)&=e^{2\pi i s}y_\ell(z;\frac{\omega_j}{2}-a)\\
%y_\ell(z+\tau;\frac{\omega_j}{2}-a)&=e^{-2\pi i r}y_\ell(z;\frac{\omega_j}{2}-a).
%\end{align*}
Finally, if $a$ is replaced by $a+m_1+m_2\tau$ with $m_1,m_2\in \mathbb{Z}$, then  solving (\ref{rs-new}) gives us  $(s+2m_1,r+2m_2)$.
This proves Part (1) of Theorem \ref{thm-intro-data}.

Next, if $Q_\ell(T)=0$, by Theorem \ref{mapT}, $T$ is a branched value of $T^{(\ell)}$.  Hence, $(T^{(\ell)})^{-1}(T)=\{\underline{a}\}$ with 
\iffalse 
Let $\ell\in \{1,2,3\}$ and $\{j,k,\ell\}=\{1,2,3\}$. The map $T^{(\ell)}$
 is a double covering map and branched along four points 
$\underline{\pm \frac{\omega_j}{4}}, \underline{\pm \frac{\omega_k}{4}}$ on $Y_\ell$. Moreover, the corresponding four branched values are exactly four roots of $Q_\ell(T)$. More precisely,
$$\left(T^{(\ell)}\left(\underline{\pm \frac{\omega_j}{4}}\right)\right)^2=e_j-e_\ell\quad \text{and}\quad \left(T^{(\ell)}\left(\underline{\pm \frac{\omega_k}{4}}\right)\right)^2=e_k-e_\ell.$$
\fi 
 $\underline{a}\in \left\{\underline{\pm \frac{\omega_j}{4}},\, \underline{\pm \frac{\omega_k}{4}}\right\}$ and  $\{\ell, j, k\}=\{1,2,3\}$. 
%%%
\iffalse 
 Notice that $y_\ell(z;a)$, $y_\ell(z; \frac{\omega_j}{2}-a)$ and $y_\ell(z; \frac{\omega_k}{2}-a)$  are linearly dependent, we need to construct another solution of $\mathcal{L}_\ell(T)$ such that it is linearly independent with $y_\ell(z;a)$. 
Solve the equation (\ref{rs-new}) gives us 
\begin{equation*}
\begin{aligned}
2\pi is&=\eta_1(2a-\frac{1}{2}\omega_\ell-\frac{1}{2}\omega_3)-(\zeta(2a)-\frac{1}{2}\eta_\ell-\frac{1}{2}\eta_3)\\
2\pi ir&=-\eta_2(2a-\frac{1}{2}\omega_\ell-\frac{1}{2}\omega_3)+\tau(\zeta(2a)-\frac{1}{2}\eta_\ell-\frac{1}{2}\eta_3).
\end{aligned}
\end{equation*}
By a direct computation, we obtain that $s,r\in \frac{1}{2}\mathbb{Z}$, i.e., the eigenvalues are $\pm 1$. 
%By (\ref{eqn-y1}), (\ref{eqn-y2}), we have an elliptic function
%\begin{align*}
%&y(z;a)y(z;\frac{\omega_j}{2}-a)\\
%=&e^{(\eta_j-\eta_i-\eta_3)z}\frac{\sigma(z-a)\sigma(z-a+\frac{\omega_i}{2})\sigma(z+a-\frac{\omega_j}{2})\sigma(z+a-\frac{\omega_j}{2}+\frac{\omega_i}{2})}{\prod_{k=0}^3\sigma(z-\frac{\omega_k}{2})},
%\end{align*}
\fi
 %Notice that $y_\ell(z;a)$, $y_\ell(z; \frac{\omega_j}{2}-a)$ and $y_\ell(z; \frac{\omega_k}{2}-a)$  are linearly dependent, 
Since $\mathcal{L}_\ell(T)$  is not completely reducible, then we need to construct another solution of $\mathcal{L}_\ell(T)$ such that it is linearly independent with $y_\ell(z;a)$. 
Notice that $y_\ell^2(z;a)$ is an elliptic solution of the second symmetry product equation (\ref{cc6}). By Theorem \ref{prop1.2}, there exists a nonzero constant $b\in \mathbb{C}$ such that $y_\ell^2(z;a)=bF^{(\ell)}(z;T)$, where
\begin{align*}%\label{eq_07141414}
 F^{(\ell)}(z;T)&=-\sum\limits_{k=0}^3\varepsilon_k^{(\ell)}\zeta(z-\frac{\omega_k}{2})+(2T+\eta_{3-\ell})\\
 &=-\sum\limits_{k=0}^3\varepsilon_k^{(\ell)}\left(\zeta(z-\frac{\omega_k}{2})-\zeta(a-\frac{\omega_k}{2})\right).
\end{align*}
%is an elliptic solution of (\ref{cc6}) for all $T\in V_i$. Moreover, we get the following important corollary. 
Here, the second equality is due to $T=T^{(\ell)}(\underline{a})$ and (\ref{ccccc16}).
%%%
Note that  $1/F^{(\ell)}(z;T)$ is also  an elliptic function with exactly two double poles at $a$ and $a-\frac{\omega_\ell}{2}$ on $E_\tau$, then we can set 
$$ \frac{1}{F^{(\ell)}(z;T)}=t_1\wp(z-a)+t_2\wp(z-a+\frac{\omega_\ell}{2})+t_0+\mu\zeta(z-a)-\mu\zeta(z-a+\frac{\omega_\ell}{2}),$$
with $t_0,t_1,t_2,\mu\in \mathbb{C}$. 
By a direct computation, we have
\begin{align*}
t_1&=\lim_{z\to a}\frac{(z-a)^2}{F^{(\ell)}(z;T)}=\frac{2}{\sum\limits_{k=0}^3\varepsilon_k^{(\ell)}\wp'(a-\frac{\omega_k}{2})},\\
t_2&=\lim_{z\to a-\frac{\omega_\ell}{2}}\frac{(z-a+\frac{\omega_\ell}{2})^2}{F^{(\ell)}(z;T)}=\frac{2}{\sum\limits_{k=0}^3\varepsilon_k^{(\ell)}\wp'(a-\frac{\omega_\ell}{2}-\frac{\omega_k}{2})}=t_1,\\
\mu&=\lim_{z\to a}\frac{(z-a)^2-t_1F^{(\ell)}(z;T)}{(z-a)F^{(\ell)}(z;T)}=0,\\
t_0&=-t_1\wp(a)-t_1\wp(a-\frac{\omega_\ell}{2})=-\frac{2\left(\wp(a)+\wp(a-\frac{\omega_\ell}{2})\right)}{\sum\limits_{k=0}^3\varepsilon_k^{(\ell)}\wp'(a-\frac{\omega_k}{2})},
\end{align*}
where the last equation is due to $0$ is a zero of $1/F^{(\ell)}(z;T)$.

%\begin{align*}
%t_1&=\lim_{z\to a}\frac{(z-a)^2}{y(z;a)y(z;\frac{\omega_j}{2}-a+\omega_i)}\\
%&=e^{(\eta_3+\eta_i-\eta_j)a}\frac{\prod_{k=0}^3\sigma(a-\frac{\omega_k}{2})}{\sigma(z-a)\sigma(z-a+\frac{\omega_i}{2})\sigma(z+a-\frac{\omega_j}{2})\sigma(z+a-\frac{\omega_j}{2}+\frac{\omega_i}{2})}
%\end{align*}

Define 
\begin{align*}
\chi(z;a):=&-\frac{1}{t_1}\int_0^z\frac{1}{F^{(\ell)}(\xi;T)}d\xi\\
=&\zeta(z-a)+\zeta(z-a+\frac{\omega_\ell}{2})+z\wp(a)+z\wp(a-\frac{\omega_\ell}{2})+\zeta(a)+\zeta(a-\frac{\omega_\ell}{2}), 
\end{align*}
then $$\chi'(z;a)=-\frac{1}{t_1F^{(\ell)}(z;T)}=-\frac{b}{t_1y_\ell^2(z;a)}.$$
Since $y_\ell''(z;a)=q_\ell(z;T)y_\ell(z;a)$, then $y_2(z;a):=\chi(z;a)y_\ell(z;a)$ is also a solution of $\mathcal{L}_\ell(T)$, which is clearly linearly independent with $y_\ell(z;a)$.
Denote by 
\begin{equation}\label{chi12}
\begin{aligned}
\chi_1=&\chi(z+1;a)-\chi(z;a)
=2\eta_1+\wp(a)+\wp(a-\frac{\omega_\ell}{2}),\\
 \chi_2=&\chi(z+\tau;a)-\chi(z;a)=2\eta_2+\tau\wp(a)+\tau\wp(a-\frac{\omega_\ell}{2}),
\end{aligned}
\end{equation}
we have 
\begin{align*}
y_2(z+1)
&=(\chi(z)+\chi_1)e^{-2\pi is}y(z)=e^{-2\pi is}y_2(z)+e^{-2\pi is}\chi_1y(z),\\
 y_2(z+\tau)
&=(\chi(z)+\chi_2)e^{2\pi ir}y(z)=e^{2\pi ir}y_2(z)+e^{2\pi ir}\chi_2y(z).
\end{align*}
%Replacing $y_1$ by $\chi_1y_1$ (still denoted by $y_1$), the monodromy matrices under $y_1$ and $y_2$ becomes 
Therefore, the monodromy matrices under the linearly independent solutions $y(z;a)$ and $y_2(z; a)$ are the following:
\begin{align*}
\begin{array}{ccc}
e^{-2\pi is}\left(\begin{array}{cc} 1&0\\
\chi_1 &1
\end{array}\right),
&&
e^{2\pi ir}\left(\begin{array}{cc} 1&0\\
\chi_2 &1
\end{array}\right).
\end{array}
\end{align*}
where $s,r\in \frac{1}{2}\mathbb{Z}$.  %and $\mathcal{C}=\frac{\chi_2}{\chi_1}$.
By (\ref{chi12}) and the Legendre relation $\tau \eta_1-\eta_2=2\pi i$, we have $\tau \chi_1-\chi_2=4\pi i$, then $\chi_1, \chi_2$ can not vanish simultaneously and 
\iffalse 
We summarize the above discussion to the following theorem.

\begin{theorem}\label{thm-monodromy-3}
Let $\ell\in \{1,2,3\}, \{\ell,j,k\}=\{1,2,3\}$. If  $\underline{a}\in \left\{\underline{\pm \frac{\omega_j}{4}},\, \underline{\pm \frac{\omega_k}{4}}\right\}$, i.e., $Q_\ell(T_\ell(\underline{a}))=0$,  then  $\mathcal{L}(T_\ell(\underline{a}))$ is not completely reducible and the monodromy matrices are 
\begin{align*}
\begin{array}{ccc}
 \varepsilon_1\left(\begin{array}{cc} 1&0\\
1 &1
\end{array}\right),
&&
 \varepsilon_2\left(\begin{array}{cc} 1&0\\
\mathcal{C}&1
\end{array}\right),
\end{array}
\end{align*}
where $ \varepsilon_1,  \varepsilon_2\in \{\pm 1\}$ and 
%\begin{align*}
%\chi_1=&
%2\eta_1+\wp(a)+\wp(a-\frac{\omega_\ell}{2}),\\
% \chi_2=&2\eta_2+\tau\wp(a)+\tau\wp(a-\frac{\omega_\ell}{2}).
%\end{align*}
\begin{equation}\label{data-C}
\mathcal{C}=\tau-\frac{\pi i}{\eta_1+e_k}, \qquad \text{for}\quad \underline{a}=\pm \underline{\frac{\omega_k}{4}}.
\end{equation}
%\footnote{\textcolor{red}{$\pm \underline{\frac{\omega_j}{4}}$ 对应的$\frac{\chi_2}{\chi_1}$是一样的，因为$\wp(-\frac{\omega_j}{4}-\frac{\omega_\ell}{2})=\wp(\frac{\omega_j}{4}+\frac{\omega_\ell}{2})=\wp(\frac{\omega_j}{4}-\frac{\omega_\ell}{2})$.}}
\end{theorem} 
\fi 
the monodromy data
$$\mathcal{C}(\underline{a})=\frac{\chi_2}{\chi_1}=\tau-\frac{4\pi i}{2\eta_1+\wp(a)+\wp(a-\frac{\omega_\ell}{2})}.$$
Without loss of generality, we assume  $\underline{a} =\underline{\pm \frac{\omega_k}{4}}$.
%It remains to prove the identity (\ref{intro-C}). By applying the Legendre relation $\tau \eta_1-\eta_2=2\pi i$, we have 
Notice that $\wp(-\frac{\omega_k}{4}-\frac{\omega_\ell}{2})=\wp(\frac{\omega_k}{4}+\frac{\omega_\ell}{2})=\wp(\frac{\omega_k}{4}-\frac{\omega_\ell}{2})$, we have $\mathcal{C}(\underline{\frac{\omega_k}{4}})=\mathcal{C}(\underline{-\frac{\omega_k}{4}})$.  Furthermore, 
by (\ref{ek}), we have  
$$e_k=\frac{1}{2}\left(\wp(\frac{\omega_k}{4})+\wp(\frac{\omega_k}{4}-\frac{\omega_\ell}{2})\right)$$
\iffalse 
\begin{align*}
e_k=\wp(2\cdot\frac{\omega_k}{4})&=\frac{1}{4}\sum_{i=0}^3\wp(\frac{\omega_k}{4}-\frac{\omega_i}{2})\\
&=\frac{1}{4}\left(\wp(\frac{\omega_k}{4})+\wp(\frac{\omega_k}{4}-\frac{\omega_\ell}{2})+\wp(-\frac{\omega_k}{4})+\wp(\frac{\omega_k}{4}-\frac{\omega_j}{2})\right)\\
&=\frac{1}{2}\left(\wp(\frac{\omega_k}{4})+\wp(\frac{\omega_k}{4}-\frac{\omega_\ell}{2})\right),
\end{align*}
because  $\frac{\omega_k}{4}-\frac{\omega_j}{2}=-(\frac{\omega_k}{4}-\frac{\omega_\ell}{2})$ on $E_\tau$. % we have $\wp(\frac{\omega_k}{4}-\frac{\omega_j}{2})=\wp(\frac{\omega_k}{4}-\frac{\omega_\ell}{2})$. Clearly, we obtain that $\wp(\frac{\omega_\ell}{4})+\wp(\frac{\omega_k}{4}-\frac{\omega_\ell}{2})=2e_k$ and 
\fi
Combing with Theorem \ref{mapT}, we have
 $$\mathcal{C}(\underline{\pm\frac{\omega_k}{4}})=\tau-\frac{2\pi i}{\eta_1+e_k},\quad\text{where}\,\, T^2= \left(T^{(\ell)}\left(\underline{\pm \frac{\omega_k}{4}}\right)\right)^2=e_k-e_\ell.$$
\end{proof}

Before ending this section, we notice that $(T^{(\ell)})^{-1}(0)=Y_\ell\cap Y_0$. By plugging (\ref{Y0Yell}) into (\ref{rs-new}), a direct computation gives us Corollary \ref{cor-intro-T=0}.

\section{Unitary monodromy}\label{sec-unitary}
In this last section, we investigate the constraints on the apparent parameter $T\in V_\ell$ for some $\ell\in \{1,2,3\}$ such that  $\mathcal{L}_\ell(T)$ is unitarizable.  %For this purpose, we introduce  the Hill's discriminant which is determined by the trace of the monodromy matrix.
First of all, $\mathcal{L}_\ell(T)$ need to be completely reducible, i.e., $Q_\ell(T)\neq 0$, then the monodromy data is given by (\ref{rs-new}) and $\mathcal{L}_\ell(T)$ is unitarizable if and only if $(s,r)\in \mathbb{R}^2$. 
%By Corollary \ref{cor-intro-T=0}, we see that the monodromy is unitarizable if $T=0$.  In this section, we will show $T=0$ is also a necessary condition 
Solve the equation (\ref{rs-new}) gives us 
\begin{equation*}
\left\{
\begin{aligned}
2\pi is&=\eta_1(2a-\frac{1}{2}\omega_\ell-\frac{1}{2}\omega_3)-(\zeta(2a)-\frac{1}{2}\eta_\ell-\frac{1}{2}\eta_3)\\
2\pi ir&=-\eta_2(2a-\frac{1}{2}\omega_\ell-\frac{1}{2}\omega_3)+\tau(\zeta(2a)-\frac{1}{2}\eta_\ell-\frac{1}{2}\eta_3),
\end{aligned}\right.
\end{equation*}
where $\underline{a}\in (T^{(\ell)})^{-1}(T)$. In order to make a more direct connection between $T$ and $(s,r)$, we notice that $T^{(\ell)}$ is a double covering map branched along the zeros of $Q_\ell(T)$, then 
%By Theorem \ref{mapT}, we have 
\begin{equation}\label{GLC}
Y_\ell\cong \Sigma_\ell=\{(T,W)\in \mathbb{C}^2\mid W^2=Q_\ell(T)\}.
\end{equation}
Here, $Y_\ell$ is called a generalized Lam\'e curve and $\Sigma_\ell$ is called a spectral curve. 
%In general, $\Sigma_\ell$ is a hyperelliptic curve. In our situation, 
In particular, $\Sigma_\ell$ is an elliptic curve because $Q_\ell$ is of degree $4$ and has no multiple roots.  
Observe that  $\Sigma_\ell$ can be  local holomorphic parametrized by $T$ around any unramified point, so does $Y_\ell$ by (\ref{GLC}). Consequently, $(s,r)$ can be locally expressed as a holomorphic function of $T$. In what follows, we will first  unravel this relation by introducing Baker-Akhiezer functions.

\subsection{Baker-Akhiezer functions.}

%$$\xymatrix{Y_\ell\ar[rr]\ar[rd]&&\Simga_\ell \\ &V_\ell&}$$
Let $T\in V_\ell$ for some $\ell\in \{1,2,3\}$, then there are two points $(T,\pm W)\in \Sigma_\ell$ such that $(\pi_1^{(\ell)})^{-1}=\{ (T,\pm W)\}$.  In particular, $(T,W)=(T,-W)$ if and only if $Q_\ell(T)=0$.  We plan to construct the common eigenfunction in terms of $(T,\pm W)$ and then derive the monodromy data. To simplify notation, we use ${F}(z;T)$ to replace $F^{(\ell)}(z;T)$.
%Let $P_0=(W_0, T_0)\in \Sigma_\ell$ for some $\ell\in \{1,2,3\}$, %we have  $W_0^2=(F'(z;T_0))^2-2F''(z;T_0)F(z;T_0)+4q_\ell(z;T_0)F^2(z;T_0)$ and
Since $W^2=Q_\ell(T)$, then 
\begin{equation}\label{ccccc3}
    \begin{aligned}
       {F}^\prime(z;T)^2-W^2&=(-4q_\ell(z;T){F}(z;T)+2{F}^{\prime\prime}(z;T)){F}(z;T)\\
        &:= H(z;T){F}(z;T).
    \end{aligned}
\end{equation}
%Here $F$ is one of $F^{(i)}$ given by (\ref{eq_07141414}), $1\leq i\leq 3$. 
%We denote $y_1^{\prime}/y_1$ by $\phi$, that is, by (\ref{eq07091551}),
We define 
\begin{equation}\label{ccccc4}
    \phi(z;T,\pm W):=\frac{F^\prime(z;T)\mp W}{2F(z;T)}=\frac{H(z;T)}{2(F^\prime(z;T)\pm W)}.
\end{equation}
Indeed, all the numerators and denominators in the last two expressions of (\ref{ccccc4}) can not vanish at some $z_0$ simultaneously, because if it does, then $W=0$ and $H(z_0;T)$, $F^\prime(z_0;T)$ and $F(z_0;T)$ vanish at $T$, which implies $T$ is a multiple zero of $Q_\ell(T)$, a contradiction to the fact that $e_i\neq e_j$ for $i\neq j$. By (\ref{ccccc4}), 
$\phi (z;T,\pm W)$ is also an elliptic function in $z$ and satisfies the Riccati equation:
\begin{equation}\label{eq07101530}
 (\phi (z;T,\pm W))^{\prime}%=\frac{F^{\prime\prime}}{2F}-\frac{F^{\prime\,2}\mp W_0F^{\prime}}{2F^2}%=-\frac{(F^{\prime}\mp W_0)^2}{4F^2}+q(z)
=-(\phi (z;T,\pm W))^2+q_\ell(z;T).
\end{equation}
%From the Riccati equation (\ref{eq07101530}), it is easy to obtain the following result: 
\begin{proposition}\label{pro09101540}
Let $(T,\pm W) \in\Sigma_\ell$ for some $\ell\in \{1,2,3\}$ and $\phi (z;T,\pm W)$ is defined by (\ref{ccccc4}).
%Denote by $S_\pm$
Then the pole set of  $\phi (z;T,\pm W)$ is 
%Then all poles of  $\phi (z;P^\pm_0)$  are simple and 
$$ \{ 0, \frac{\omega_1}{2}, \frac{\omega_2}{2},\frac{\omega_3}{2},a^\pm, a^\pm-\frac{\omega_\ell}{2}\}+\Lambda_\tau,$$
where $a^\pm\in E_\tau\setminus E_\tau[2]$ is a common zero of $F^{(\ell)}(z;T)$ and $(F^{(\ell)})^\prime(z;T)\pm W$.  Moreover, all poles are simple and the residues at $\frac{\omega_j}{2},  j=0,\cdots, 3$ are $-\frac{1}{2}$ and the residues at $a^\pm, a^\pm-\frac{\omega_\ell}{2}$ are $1$.
\end{proposition} 
\begin{proof}
From the expression of $F^{(\ell)}(z;T)$ in (\ref{eq_07141414}),   we have 
%$\phi (z;P_0^\pm)$ can be explicitly expressed by
\begin{equation}
   \phi (z;T,\pm W)=\frac{\sum\limits_{k=0}^{3}\epsilon_k^{(\ell)}\wp(z-\frac{\omega_k}{2})\mp W}{2\left(-\sum\limits_{k=0}^3\epsilon_k^{(\ell)}\zeta(z-\omega_k/2)+2T+\eta_{3-\ell}\right)},
\end{equation}
then it is clear that $\frac{\omega_j}{2}, j=0,\cdots, 3$ are simple poles  with $$\underset{z=\frac{1}{2}\omega_j}{\mathrm{res}}\phi(z;T,\pm W)=-\frac{1}{2},\quad  0\leq j\leq 3.$$
 If $z_1\in E_\tau\setminus E_\tau[2]$ is a pole of $\phi(z;T,\pm W)$, by the Riccati equation (\ref{eq07101530}),  we obtain that $z_1$ is  simple and the residue at $z_1$ is $1$. Notice that the total residue of an elliptic function is $0$, then 
$\phi (z;T,\pm W)$ has exactly two non-half period poles on $E_\tau$, which we denote by $a_1^\pm, a_2^\pm$. % then $a_1, a_2\not\in E_\tau[2]$. 
By (\ref{ccccc4}), we see that 
%Since all numerators and denominators in the definition of $\phi$ can not vanish at some point simultaneously, then the common zeros of $F(z;T_0)$ and $F^\prime(z;T_0)+W_0$ are poles of $\phi(z;P_0)$, thus  
$a_1^\pm, a_2^\pm$ are exactly the common zeros of $F^{(\ell)}(z;T)$ and $(F^{(\ell)})^\prime(z;T)\pm W$ on $E_\tau$.

 %Notice that $F^{(\ell)}(z;T)=0$ implies $W^2=Q_\ell(T)=((F^{(\ell)})^\prime(z;T))^2$. 
Let $a\in E_\tau$ be a common zero of $F^{(\ell)}(z;T)$ and $(F^{(\ell)})^\prime(z;T)+ W$, i.e., 
$F^{(\ell)}(a;T)=0$ and $W=-(F^{(\ell)})^\prime(a;T)$.
Notice that $F^{(\ell)}(a;T)=0$, i.e., 
$$ T=\frac{1}{2}\sum\limits_{i=0}^3 \varepsilon_{i}^{(\ell)}\zeta(a-\dfrac{\omega_i}{2})-\frac{1}{2}\eta_{3-\ell}, $$
 is equivalent to 
$$ T=\zeta(a)+\zeta(a-\frac{\omega_\ell}{2})-\zeta(2a)-\frac{1}{2}\eta_{\ell}=T^{(\ell)}(\underline{a}), $$
by  (\ref{ccccc16}).
% if and only if $T=T^{(\ell)}(\underline{a})$. 
Combing with (\ref{preimage}), $F^{(\ell)}(z;T)$ has exactly four simple zeros $\{a, a-\frac{\omega_\ell}{2}, \frac{\omega_j}{2}-a, \frac{\omega_j}{2}-a-\frac{\omega_\ell}{2}\}$ on $E_\tau$, where $j\in \{1,2,3\}\setminus\{\ell\}$.
%$(T^{(\ell)})^{-1}\left(T\right)=\left\{\underline{a}, \underline{\frac{\omega_j}{2}-a}\right\}$
Since
 $$(F^{(\ell)})^\prime(z;T)=\frac{1}{2}\sum\limits_{i=0}^3 \varepsilon_{i}^{(\ell)}\wp(z-\dfrac{\omega_i}{2}),$$
then $(F^{(\ell)})^\prime(a-\frac{\omega_\ell}{2};T)=(F^{(\ell)})^\prime(a;T)=-W$ and $(F^{(\ell)})^\prime(\frac{\omega_j}{2}-a;T)=(F^{(\ell)})^\prime(\frac{\omega_j}{2}-a-\frac{\omega_\ell}{2};T)=-(F^{(\ell)})^\prime(a;T)=W$.
\iffalse 
Notice that $a\in E_\tau$ is a sommon zero of $F^{(\ell)}(z;T)$ and $(F^{(\ell)})^\prime(z;T)+ W$ if and only if $a$ satisfies the $2\times 2$ system on $E_{\tau}$:
\begin{equation}\label{eqn-TW}
 \left\{   \begin{split}
        T=&\frac{1}{2}\sum\limits_{i=0}^3 \varepsilon_{i}^{(\ell)}\zeta(a-\dfrac{\omega_i}{2})-\dfrac{1}{2}\eta_{3-\ell}, \\
      W=& -\frac{1}{2}\sum\limits_{i=0}^3 \varepsilon_{i}^{(\ell)}\wp(a-\dfrac{\omega_i}{2}). 
    \end{split}\right.
\end{equation}
If $a$ is a solution of (\ref{eqn-TW}), by a direct calculation, $a-\frac{\omega_\ell}{2}$ is also a solution of (\ref{eqn-TW}).
Additionally, denote by $\{\ell,j,k\}= \{1,2,3\}$, we have  $\frac{\omega_j}{2}-a,\frac{\omega_k}{2}-a$  solve the similar system where $T$ still reminds the same but $W$ changes to $-W$. Hence, $\frac{\omega_j}{2}-a,\frac{\omega_k}{2}-a$  are exactly zeros of $y(z;T,-W)$ on $E_\tau$. 
\fi 
\end{proof}

\iffalse 
Let $P_0^{\ast}=(-W_0,T_0)$, which is also a point on $\Sigma_\ell$. Then by (\ref{ccccc3}), 
\begin{equation*}%\label{ccccc4}
    \phi(z;P_0^\ast):=%\frac{y_2^{\prime}}{y_2}=
\frac{F^\prime(z;T_0)+W_0}{2F(z;T_0)}=\frac{H(z;T_0)}{2(F^\prime(z;T_0)-W_0)},
\end{equation*}
which is an elliptic function and satisfies the Riccati equation (\ref{eq07101530}).   % Similarly, $\phi(z;P_0^\ast)$  has simple poles at all half period points $\omega_i/2, 0\leq i\leq 3$ on $E_\tau$ with residue $-{1}/{2}$, and has exactly two non-half period poles $a_1^\ast, a_2^\ast$ on $E_\tau$ with residue $1$ which are exactly common zeros of  $F(z;T_0)$ and  $F^{\prime}(z;T_0)-W_0$.
\fi 

%Denote by 
%$$S:=\{a_1, a_2, a_1^\ast, a_2^\ast, 0, \frac{\omega_1}{2}, \frac{\omega_2}{2},\frac{\omega_3}{2} \}+\Lambda_\tau.$$
Fix $z_0\in \mathbb{C}$ such that $z_0$ is not a pole of both $\phi(z;T,\pm W)$,  %such that $[z_0]\notin $
by Proposition \ref{pro09101540}, we can define 
\begin{align}
   y(z; T,\pm W)&={ \exp}(\int^z_{z_0}\phi(z;T,\pm W)dz),\label{ccccc5}
 % y(z,z_0;P_0^\ast)&={\exp}(\int^z_{z_0}\phi(z;P_0^\ast)dz)    
\end{align}
which is a single-valued meromorphic function.  
%where $z_0$ and the integration paths always avoid poles of  $\phi(z;P_0)$ and  $\phi(z;P_0^\ast)$ respectively. 
Notice that 
\begin{equation}\label{y-phi}
   y'(z; T,\pm W)=\phi(z;T,\pm W)y(z; T,\pm W),
%y(z_0; T,\pm W)&=1
\end{equation}
by the Riccati equation (\ref{eq07101530}), we have 
$y(z;T,\pm W)$  are solutions of $\mathcal{L}_\ell(T)$.  Furthermore, since $\phi(z;T,\pm W)$ is elliptic, by (\ref{y-phi}), we have 
\begin{align*}
   {y'(z+1; T,\pm W)}{y(z; T,\pm W)}-{y(z+1; T,\pm W)}{y'(z; T,\pm W)}=0,
\end{align*}
thus ${y(z+1; T,\pm W)}/{y(z; T,\pm W)}$ is dependent of $z\in U_0$, i.e., 
\begin{equation}\label{eqn-yp1}
\begin{aligned}
y(z+1; T,\pm W)&=\varepsilon_3{\rm exp}(\int_{z_0}^{z_0 +1}\frac{\mp W}{2F(z;T)}dz) y(z;T,\pm W)
\end{aligned}
\end{equation}
where $\varepsilon_3=\exp(\frac{1}{2}\int_{z_0}^{z_0+1}\frac{F'(z; T)}{F(z;T)}dz) =\exp(\pi i n_F)\in \{\pm 1\}$ and $n_F$ denotes 
the winding number of the closed curve $z\mapsto F(z; T)$, $z\in [z_0, z_0+1]$, with respect to the origin.
Similarly, we have 
\begin{equation}\label{eqn-yp2}
\begin{aligned}
y(z+\tau; T,\pm W)&=\varepsilon_4{\rm exp}(\int_{z_0}^{z_0 +\tau}\frac{\mp W}{2F(z;T)}dz) y(z;T,\pm W)
\end{aligned}
\end{equation}
where $\varepsilon_4=\exp(\frac{1}{2}\int_{z_0}^{z_0+\tau}\frac{F'(z; T)}{F(z;T)}dz) =\exp(\pi i m_F)\in \{\pm 1\}$ and $m_F$ denotes 
the winding number of the closed curve $z\mapsto F(z; T)$, $z\in [z_0, z_0+\tau]$, with respect to the origin.
Therefore, $y(z;T,\pm W)$  are common eigenfunctions of $\mathcal{L}_\ell(T)$.
%Let $P_0\in \Sigma_\ell$ for some $\ell\in \{1,2,3\}$ and $z_0\in \mathbb{C}\setminus (S_+\cup S_-)$,
 Following the convention of KdV theory (See \cite{ GW-1996,GH-2003}.), we call $y(z;T, \pm W)$ as \textit{Baker-Akhiezer functions}.

By (\ref{y-phi}) and Propsition \ref{pro09101540}, we obtain that $a^\pm, a^\pm-\frac{\omega_\ell}{2}\in E_\tau\setminus E_\tau[2]$ are zeros of $y(z;T,\pm W)$.
By Lemma \ref{lem-YV} and Remark \ref{rek-yell},  up to a nonzero constant multiple, 
$$y(z;T,\pm W)= y_\ell(z; a^\pm).$$
%By Lemma \ref{lem-YV}, we have $T\in V_\ell$ is given by (\ref{T-ell}).
%Notice that $T=T^{(\ell)}(\underline{a})$
Therefore,  we have the following commutative diagram 
$$
\xymatrix{Y_\ell \ar[dr]_{T^{(\ell)}}\ar[rr]^{\varphi}&&\Sigma_\ell\ar[dl]^{\pi_1^{(\ell)}} \\
&V_\ell }, \quad \text{given by}\quad 
\xymatrix{\underline{a} \ar@{|->}[dr]_{T^{(\ell)}}\ar@{|->}[rr]^{\varphi}&&(W,T)\ar@{|->}[dl]^{\pi_1^{(\ell)}} \\
&T },
$$
where $\varphi(\underline{a})=(T^{(\ell)}(\underline{a}), -(F^{(\ell)})'(a;T^{(\ell)}(\underline{a})))$  and $\varphi^{-1}(W,T)$ is defined by the two common zeros of $F^{(\ell)}(z;T), (F^{(\ell)})^\prime(z;T)+ W$.
%In particular, 
%then $T^{(\ell)}(\underline{\frac{\omega_j}{2}-a})=T=\pi_1^{(\ell)}(T,-W)$.
%where $T^{(\ell)}(\underline{a})=T$

If $T\in V_\ell$ such that $Q_\ell(T)\neq 0$, by (\ref{eqn-yp1}) and (\ref{eqn-yp1}), up to a sign, the monodromy data $(s,r)\mod \mathbb{Z}^2$ satisfying 
\begin{equation}\label{eqn-sr-exp}
\begin{aligned}
    e^{-2\pi  is}&=\varepsilon_3 {\rm exp}(-\int_{z_0}^{z_0+1}\frac{W}{2F(z;T)}dz)\\
e^{2\pi ir}&=\varepsilon_4 {\rm exp}(-\int_{z_0}^{z_0 +\tau}\dfrac{W}{2F(z;T)}dz).
\end{aligned}
\end{equation}
Hence, 
\begin{align}
s=s(W,T)=\frac{m_s}{2}+\frac{W}{2\pi  i}\int_{z_0}^{z_0+1}\frac{1}{2F^{(\ell)}(z;T)}dz,\label{eqn-s}\\
r=r(W,T)=\frac{m_r}{2}-\frac{W}{2\pi  i}\int_{z_0}^{z_0+\tau}\frac{1}{2F^{(\ell)}(z;T)}dz,\label{eqn-r}
\end{align}
 for some $m_s,m_r\in\mathbb{Z}$.

\subsection{Spectral geometry.}
Let $T\in V_\ell$ for some $\ell\in \{1,2,3\}$.
\iffalse 
we define the Hill's discriminant of $\mathcal{L}_\ell(T)$ by 
\begin{equation}\label{ccc1}
\Delta_k^{(\ell)} (T) =\frac{1}{2}\text{tr}M_k,\quad k=1,2,
\end{equation}
%where $M_k^{(\ell)}(T)$ is the monodrony matrix defined in Section \ref{sec-monodromy}. Since the potential $q_\ell(z;T)$ is holomorphic with respect to $T$, then $\Delta_k^{(\ell)}(T)$ is a holomophic function of $T\in V_\ell$. 
%Inspired by the spectral theory of the Schr\"{o}dinger operator $L$ with complex periodic smooth potentials which has attracted significant attention and has been studied widely in the literature; see e.g. \cite{LW5,LW6,BG-JAM,B-CPAM,GW,GW2,HHV-2017,Rofe-Beketov} and references therein.  We consider  
%As an analogy in the spectral theory of KdV potentials, we define the \emph{spectrum} of $\mathcal{L}_\ell(T)$ as
and the arc of stable spectrum  $\mathcal{L}_\ell(T)$ as 
\begin{equation}
\sigma_k^{(\ell)}:=(\Delta_k^{(\ell)})^{-1}([-1,1]), k=1,2.
\end{equation}
Clearly, our definitions follow from the eigenvalue problem of  Hill equations with complex-valued periodic potentials.
However, we note that equation $\mathcal{L}(\mathbb{T}, B)$ is not an eigenvalue problem because $B$ depends on $\mathbb{T}$. 
\fi 
By the expression of  (\ref{intro-Delta}), we have
\begin{align}
&\Delta^{(\ell)}_1(T)=\cos (2\pi s)\quad \text{and} \quad \Delta^{(\ell)}_2(T)=\cos (2\pi r),\quad \text{if}\,\, Q_\ell(T)\neq 0,\\
&\Delta^{(\ell)}_1(T)=\varepsilon_1\quad  \text{and} \quad \Delta^{(\ell)}_2(T)=\varepsilon_2,\quad \text{if}\,\, Q_\ell(T)=0.
\end{align}
Then  the  conditional stability set  on $V_\ell$ defined in (\ref{intro-spectrum}) is given by
 \begin{equation}
\sigma_k^{(\ell)}:=(\Delta_k^{(\ell)})^{-1}([-1,1]), k=1,2.
\end{equation}
 In the KdV theory, it has been known that $\sigma_k^{(\ell)}$ is a union of finitely many smooth arcs (a finite gap phenomena), and among those smooth arcs there is unique unbounded arcs. We will see this finite gap phenomena still holds for $\sigma_k^{(\ell)}$, however, the uniqueness of unbounded arcs might fail in general. 
 %Let $T\in V_\ell$ for some $\ell\in \{1,2,3\}$. 
As the classical theory,  there is a notion of endpoints in $\sigma_k^{(\ell)}$. We say $T\in \sigma_k^{(\ell)}$ is not an endpoint if there is a small $\varepsilon >0$ such that $\sigma_k^{(\ell)}\cap B_\varepsilon(T)=\cup_{j=1}^m\gamma_j$, where $\gamma_j$ is smooth curve and $T$ is an interior point of $\gamma_j$ for $j\in \{1,2,\cdots, m\}$. See Lemma \ref{lem-order} for the discussion of the position of endpoints. 

%Suppose $\Delta_i(T_0)=\pm1$. For the next result, we denote $F^{(i)}$, $\ell_i(T)$ and $\Delta_i$ by $F$, $\ell(T)$ and $\Delta$ respectively. 
\begin{lemma}\label{lem-order}
Let $T_0\in V_\ell$ for some $\ell\in \{1,2,3\}$. 
\begin{enumerate}
\item [(a)] If $Q_\ell(T_0)\neq 0$, then $d_k^{(\ell)}(T_0), \, k=1,2$ are  even integers;
\item[(b)] If $Q_\ell(T_0)=0$, then $d_k^{(\ell)}(T_0),\,  k=1,2$ are odd integers.
\end{enumerate}
\end{lemma}
\begin{proof}
 %Recall that the monodromy data $s$ are locally holomorphic function of $T$ and $(W,T)$. 
Let $G(W,T)=W^2-Q_\ell(T)$.
%so $T$ is a local parameter at $T_0$ by  the implicit function theorem, i.e., $W$  can be locally eliminated from $F(W,T)=0$. If $W_0=\ell_i(T_0)=0$, then $\frac{\partial F}{\partial T}(T_0)=-\ell_i'(T_0)\neq 0$ since $T_0$ is not a multiple zero of $\ell_i(T)$, so $W$ is a local parameter at $T_0$ by  the implicit function theorem, i.e., $T$ can be locally eliminated from $F(W,T)=0$

{\noindent (a)}  If $W_0^2=Q_\ell(T_0)\neq 0$, then $\frac{\partial G}{\partial W}(W_0,T_0)=2W_0\neq 0$.  By  the implicit function theorem, $W$ can be expressed as a holomorphic function of $T$ near $T_0$,
%then  the spectral curve $W^2=\ell_i(T)$ can be parameterized by $T$ near $T_0$, 
 i.e.,  $T$ can be chosen as the local coordinate of the spectral curve $\Sigma_\ell$ near $(W_0,T_0)$.  By (\ref{eqn-sr-exp}), $s=s(T)$ is a holomorphic function of $T$ near $T_0$ and then 
$$s(T)-s(T_0)=A(T-T_0)^j+O((T-T_0)^{j+1}), \quad (T\to T_0)$$
 for some $j\in\mathbb{N}^+$ and $A\ne0$.  If $\Delta_1^{(\ell)}(T_0)^2\neq 1$,  then $d_1^{(\ell)}(T_0)=0$ is even. If $\Delta_1^{(\ell)}(T_0)^2= 1$,  then 
\begin{align*}
&\Delta^{(\ell)}_1(T)-\Delta^{(\ell)}_1(T_0) \\
=&\cos( 2\pi s(T))-\cos (2\pi s(T_0))\\
=&-2\pi^2\cos(2\pi s(T_0))(s(T)-s(T_0))^2+O((s(T)-s(T_0))^4)\\
=&-2\pi^2\cos(2\pi s(T_0)) A^2(T-T_0)^{2j}+O((T-T_0)^{2j+1}, \quad  (T\to T_0).
\end{align*}
This implies $d_1^{(\ell)}(T_0)=2j$ is an even integer. 
Similarly, $d_2^{(\ell)}(T_0)$ is an even integer. This proves (a).

\medskip
 
{\noindent (b)}
\iffalse 
First of all, by (\ref{eqn-sr-exp}),
\begin{equation}\label{eqn-s}
s=s(W,T)=\frac{m}{2}-\frac{W}{2\pi  i}\int_{z_0}^{z_0+1}\frac{1}{2F^{(\ell)}(z;T)}dz
\end{equation}
 for some $m\in\mathbb{Z}$. 
\fi 
Since $\int_{z_0}^{z_0+1}\frac{1}{2F^{(\ell)}(z;T)}dz$ is a holomorphic function of $T$ near $T_0$, then we have 
$$\int_{z_0}^{z_0+1}\frac{1}{2F^{(\ell)}(z;T)}dz=N(T-T_0)^n+O((T-T_0)^{n+1}), \quad (T\to T_0)$$
for some $n\in\mathbb{N}$ and $N\neq 0$.

If $W^2_0=Q_\ell(T_0)=0$,  then $\frac{\partial G}{\partial T}(0,T_0)=-Q_\ell'(T_0)\neq 0$ because $Q_\ell(T)$ has no multiple zero.  By  the implicit function theorem, $T$ can be expressed as a holomorphic function of $W$ near $(0,T_0)$, i.e.,  $W$ can be chosen as the local coordinate of  $\Sigma_\ell$ near $(0,T_0)$. 
%then in a neighborhood of $(0,T_0)$, the local coordinate is $W$ and 
Hence, by (\ref{eqn-s}), $s=s(W)$ is a holomorphic function of $W$ near $(0,T_0)$ and 
$$s(W)-s(0)=\frac{W(T-T_0)^n}{2\pi i}\left(N+O((T-T_0))\right), \quad (W\to  0),$$
where $T\to T_0$ as $W \to 0$.

Since $T_0$ is a simple zero of $Q_\ell(T)$, then %$g(W):=\ell_i(T)/(T-T_0)$ is a holomorphic function of $W$ near $(0, T_0)$ and $g(0)\neq 0$, so 
$h(W)=(T-T_0)/Q_\ell(T)$ is  a holomorphic function of $W$ near $(0,T_0)$ and $h(0)\neq 0$. So we have 
%Note that $W^2=\ell_i(T)=16(T^4+3e_iT^2+2e_i^2+e_je_k)$ with $\{i,j,k\}=\{1,2,3\}$, then 
\begin{align*}
T-T_0%=\frac{T-T_0}{\ell_i(T)}\ell_i(T)
=W^2h(W)=W^2(h(0)+O(W)), \quad  (W\to 0).
%W^2=\ell_i(T)-\ell_i(T_0)%=16(T^4-T_0^4+3e_i^2(T^2-T_0^2))
%=16(T-T_0)(T+T_0)(T^2+T_0^2+3e_i^2)
\end{align*}
%and then 
%\begin{align*}
%s(W)-s(0)&=\frac{W^{2n+1}}{2\pi\sqrt{-1}}(h(0)^n+O(W))\left(N+O(W^2)\right), \\
%&=\frac{W^{2n+1}}{2\pi\sqrt{-1}}\left(h(0)^nN+O(W)\right), \quad (W\to  0).
%\end{align*}
Note that $\Delta_1^{(\ell)}(T_0)=\cos (2\pi s(0))= \pm 1$. When $W$ goes to $0$, we have $T\to T_0$  and 
\begin{align*}
&\Delta_1^{(\ell)}(T)-\Delta_1^{(\ell)}(T_0) \\
=&\cos( 2\pi s(W))-\cos (2\pi s(0))\\
=&\mp2\pi^2(s(W)-s(0))^2+O((s(W)-s(0))^4)\\
=&\mp\frac{1}{2}W^2(T-T_0)^{2n}(N^2+O(T-T_0))+O((T-T_0)^{4n+2}),\\
=&\mp\frac{1}{2}(T-T_0)^{2n+1}\frac{N^2+O(W^2)}{h(0)+O(W)}+O((T-T_0)^{4n+2}),
\end{align*}
thus, $d_1^{(\ell)}(T_0)=2n+1$ is an odd integer. 

Similarly, $d_2^{(\ell)}(T_0)$ is an odd integer. This proves (b).
\end{proof}

From Lemma \ref{lem-order}, we obtain that $T_0\in V_\ell$  is an endpoint of $\sigma_k^{(\ell)}, k=1,2$ if and only if $Q_\ell(T_0)=0$.  

Next, we will discuss if $\sigma_k^{(\ell)}, k=1,2$ have an endpoint at infinity. 

Let $R\in \mathbb{R}^+$ be a sufficient large number such that $Q_\ell(T)\neq 0 $ for all $|T|>R$.  %For any $T_0\in V_i$ with $|T_0|>R$, 
By the proof of Lemma \ref{lem-order}, $W$ can be expressed as a holomorphic function of $T$ for $|T|>R$.
%Since $W^2=\ell_k(T)=16(T^4+3e_iT^2+2e_i^2+e_je_k)$ with $\{i,j,k\}=\{1,2,3\}$, then 
By Proposition \ref{prop-ell}, we have 
$$W= 4\varepsilon_0T^2(1+O(T^{-2})),\quad (T\to \infty),$$
where $\varepsilon_0\in \{\pm 1\}$. 
%Since changing the sign of $W$ only reversing the sign of $(s,r)$, thus we can assume that $$W= 4T^2(1+O(T^{-2})),\quad (T\to \infty).$$
% and Proposition \ref{pro07121710}, we have 
%Recall that $$ F^{(\ell)}(z;T)=-\sum\limits_{j=0}^3\varepsilon_j^{(\ell)}\zeta(z-{\omega_j}/{2})+2T+\eta_{3-\ell}.$$ 
 Let $T\to \infty$, by (\ref{eq_07141414}), we have 
\begin{align*}
\frac{W}{2 F^{(\ell)}(z;T)}&=\frac{  2\varepsilon_0 T^2(1+O(T^{-2}))}{2T-\left(\sum\limits_{j=0}^3\varepsilon_j^{(\ell)}\zeta(z-{\omega_j}/{2})-\eta_{3-\ell}\right)}\\
%&=\frac{ 2\varepsilon_0 T(1+O(T^{-2}))}{1-\frac{1}{2}T^{-1}\left(\sum\limits_{j=0}^3\varepsilon_j^{(\ell)}\zeta(z-{\omega_j}/{2})-\eta_{3-\ell}\right)}\\
%&= 2\varepsilon_0 T\left[1+\frac{1}{2}T^{-1}\left(\sum\limits_{k=0}^3\varepsilon_j^{(\ell)}\zeta(z-{\omega_j}/{2})-\eta_{3-\ell}\right)+O(T^{-2})\right],\\
&=  2\varepsilon_0 T+ \varepsilon_0 \left(\sum\limits_{j=0}^3\varepsilon_j^{(\ell)}\zeta(z-{\omega_j}/{2})-\eta_{3-\ell}\right)+O(T^{-1})).
\end{align*}
Note that $\zeta(z)={\sigma'(z)}/{\sigma(z)}$, the transformation law
%$$\sigma(z+\omega_j)=-e^{\eta_j(z+\frac{\omega_j}{2})}\sigma(z), \quad  j=1,2,3,$$
% $\zeta'(z)=-\wp(z)$.
and the Legendre relation implies %$\tau \eta_1-\eta_2=2\pi i$,  we have
\begin{align*}
    b_\ell:=&\int_{z_0}^{z_0+1}\left(\sum\limits_{j=0}^3\varepsilon_j^{(\ell)}\zeta(z-\frac{\omega_j}{2})-\eta_{3-\ell}\right)dz\\
=&\sum\limits_{j=0}^3\varepsilon_j^{(\ell)}\ln(-e^{ \eta_1(z_0-\frac{\omega_j}{2}+\frac{1}{2})})-\eta_{3-\ell}\\
=&\sum\limits_{j=0}^3\varepsilon_j^{(\ell)}\left({ \eta_1(z_0-\frac{\omega_j}{2}+\frac{1}{2})}+(2p+1)\pi i\right)-\eta_{3-\ell}\\
 =&-\frac{1}{2}\eta_1\sum\limits_{j=0}^3\varepsilon_j^{(\ell)}{\omega_j}-\eta_{3-\ell}=\left\{\begin{array}{lll}
2\pi i& & \ell=1,\\
0&& \ell=2,3.
\end{array}\right.
\end{align*}
where $p\in \mathbb{Z}$ is some integer. 
%\begin{align*}
%    \text{If}\ i=1,\ \text{then}\ -&(\sum\limits_{k=0}^3\varepsilon_k^{(1)}\omega_k/2)\eta_1-\eta_2=-\eta_1(\omega_1/2-\omega_2/2-\omega_3/2)-\eta_2\\&=\eta_1\omega_2-\eta_2=2\pi i, \\
%\text{If}\ i=2,\ \text{then}\ -&(\sum\limits_{k=0}^3\varepsilon_k^{(2)}\omega_k/2)\eta_1-\eta_1=-(-\omega_1/2+\omega_2/2-\omega_3/2)\eta_1-\eta_1=0,
%\end{align*}
%$\text{If}\ i=3,\ \text{then}\ -(\sum\limits_{k=0}^3\varepsilon_k^{(3)}\omega_k/2)\eta_1=-(-\omega_1/2-\omega_2/2+\omega_3/2)\eta_1=0.$\\
Combine (\ref{eqn-s}) with the above calculations, we have 
\begin{equation*}
s=s(T)=\frac{m_s}{2}-\frac{\varepsilon_0}{2\pi i}
\left(2T+ b_\ell+O(T^{-1}))\right),\quad (T\to \infty).
\end{equation*}
 Therefore, 
\begin{align*}
\Delta_1^{(\ell)}(T)=\cos (2\pi s)&=\cos \left(m_s\pi+i\varepsilon_0\left( 2T+b_\ell+O(T^{-1})\right)\right)\\
&=(-1)^{m_s}\cos \left( 2\varepsilon_0 Ti+O(T^{-1})\right),\quad (T\to \infty).
\end{align*}
Therefore, for sufficient large $T$, $T\in \sigma_1^{(k)}$ if and only if 
$$t:= 2\varepsilon_0Ti+O(T^{-1}) \in \mathbb{R},$$
which is equivalent to 
$$T=\frac{1}{2}\varepsilon_0 t i+O(t^{-1}), \quad t \in \mathbb{R}.$$
%If $T$ is a pure imaginary number
%Let $T=x\sqrt{-1}$ with $x\in\mathbb{R}$, then $\Delta_1^{(k)}(T)$ approaches to a real number in $[-1,1]$ as $x\to \pm \infty$.  
This implies $\sigma_1^{(\ell)}$
 contains exactly two unbounded arcs $\sigma_{1\pm}^{(\ell)}$ which asymptotically tend to $\pm i\infty$. 
%$\pm i\infty$ are exactly two endpoints of $\sigma_1^{(k)}$. 
Similarly, we have 
\begin{align*}
\int_{z_0}^{z_0+\tau}\left(\sum\limits_{j=0}^3\varepsilon_j^{(\ell)}\zeta(z-\frac{\omega_j}{2})-\eta_{3-\ell}\right)dz&=-\frac{1}{2}\eta_2\sum\limits_{j=0}^3\varepsilon_j^{(\ell)}{\omega_j}-\tau\eta_{3-\ell}\\
& =\left\{\begin{array}{ll}
-2\pi i&  \ell=2,\\
0& \ell=1,3,
\end{array}\right.
\end{align*}
and then 
\begin{align*}
\Delta_2^{(\ell)}(T)=\cos (2\pi r)%&=\cos \left(m\pi+i\left(\pm 2T\pm d_k+O(T^{-1})\right)\right)\\
&=(-1)^n\cos \left(2\varepsilon_0\tau Ti+O(T^{-1})\right),\, (T\to \infty),
\end{align*}
where $n\in \mathbb{Z}$ is some integer. 
%where $d_k=c_k-\tau c_k$. 
%\begin{align*}
%d_k=b_k+c_k-\tau c_k=\left\{\begin{array}{lll}2\pi i+
%\eta_2-\tau\eta_2&& k=1\\
%\eta_1-\tau\eta_1&& k=2\\
%0&& k=3 \end{array} \right. 
%\end{align*}
Therefore, %for sufficient large $T$, $T\in \sigma_2^{(k)}$ if and only if 
%$$u:=\pm 2i\tau T\pm i(1-\tau) c_k+O(T^{-1}) \in \mathbb{R},$$
%which is equivalent to 
$$T=\frac{1}{2\tau}\varepsilon_0ui+O(t^{-1}), \quad (u\to \pm \infty),\quad u \in \mathbb{R}.$$
This implies $\sigma_2^{(\ell)}$ contains exactly two unbounded arcs $\sigma_{2\pm}^{(\ell)}$ which asymptotically tend to $\pm \frac{i}{2\tau}\infty$.
%$(\frac{1}{2}-\frac{1}{2\tau})c_k\pm \frac{i}{2\tau}\infty$ are exactly two endpoints of $\sigma_2^{(k)}$.
%\begin{align*}
%\lim\limits_{x\to \pm \infty} \Delta(x\sqrt{-1})=\lim\limits_{x\to \pm \infty} (-1)^m\cos (\mp 2x)
%\end{align*}

Summarize the above analysis,  we obtain the following lemma. 

\begin{lemma}\label{lem-endpoints}
Let $\ell\in \{1,2,3\}$ with $\{\ell,j,k\}=\{1,2,3\}$ and $T_j, T_k\in V_\ell $ such that 
$T_j^2+e_\ell-e_j=0$ and $T_k^2+e_\ell-e_k=0$ .
\begin{enumerate}
\item [(1)]
Both  $\sigma_1^{(\ell)}$ and  $\sigma_2^{(\ell)}$ have exactly 4  finite endpoints:
$ \pm T_j, \pm T_k.$ 
\item[(2)]  $\sigma_1^{(\ell)}$
 contains exactly two unbounded arcs $\sigma_{1\pm}^{(\ell)}$ which asymptotically tend to $\pm i\infty$. 
% Moreover,  there is exactly one spectral arc meeting at each endpoint.  
\item[(3)]  $\sigma_2^{(\ell)}$ contains exactly two unbounded arcs $\sigma_{2\pm}^{(\ell)}$ which asymptotically tend to $\pm\frac{i}{2\tau}\infty$.
%6 endpoints which are the following:
%$$ \pm T_1,\quad \pm T_2,\quad (\frac{1}{2}-\frac{1}{2\tau})c_k\pm \frac{i}{2\tau}\infty. $$
\end{enumerate}
\end{lemma}

\begin{proof}
Note that the spectral polynomial
$$Q_\ell(T)=16(T^2+e_\ell-e_{j})(T^2+e_\ell-e_{k}),$$
we obtain that $\pm T_j, \pm T_k$ are exactly four zeros of $Q_\ell(T)$, which is desired. 
\end{proof}

\begin{remark}
The two unbounded arcs $\sigma_{k\pm}^{(\ell)}, k=1,2$ in Lemma \ref{lem-endpoints} might be parts of the same  arc of $\sigma_k^{(\ell)}$.
\end{remark}

\begin{lemma}\label{lem-loop}
Let $\ell\in \{1,2,3\}$ and $k\in \{1,2\}$. 
\begin{enumerate}
\item [(a)]  $\sigma_k^{(\ell)}$ does not contain any loop.   
\item [(b)]  $\sigma_k^{(\ell)}$ is symmetric with respect to the origin in $V_\ell$, i.e.,  $\Delta_k^{(\ell)}(T)=\Delta_k^{(\ell)}(-T)$.
\end{enumerate}
\end{lemma}

\begin{proof}
{\noindent (a)} Suppose that %the spectrum $\sigma_\ell^{(k)}$ has a self-intersection point, then
 there is a loop $\gamma\in \sigma_k^{(\ell)}$ such that $\Delta_k^{(\ell)}(\gamma)\subseteq [-1,1]$, so $$\mathrm{Im}(\Delta_k^{(\ell)})|_\gamma\equiv 0.$$ Since  $\mathrm{Im}(\Delta_k^{(\ell)})$ is harmonic, by the maximum principle, $\mathrm{Im}(\Delta_k^{(\ell)})|_\gamma\equiv 0$ on the domain bounded by $\gamma$. Thus $\Delta_k^{(\ell)}(T)$ is a constant, which is impossible. 

\medskip

{\noindent (b)}  Recall that 
$$q_\ell(z;T)=\frac{3}{4}\sum_{j=0}^3\wp(z-\frac{\omega_j}{2})+T\sum_{j=0}^3\varepsilon_j^{(\ell)} \zeta(z-\frac{\omega_j}{2})+T^2-\eta_{3-\ell}T.$$
Note that  $\zeta(-z)=-\zeta(z)$ and $\zeta(z+\frac{1}{2}\omega_j)=\zeta(z-\frac{1}{2}\omega_j)+\eta_j,\ 0\leq j\leq 3$, %where we set $\eta_0=0$, 
we have 
\begin{align*}
q_\ell(-z;-T)&=\frac{3}{4}\sum_{j=0}^3\wp(z-\frac{\omega_j}{2})+T\sum_{j=0}^3\varepsilon_j^{(\ell)} \zeta(z+\frac{\omega_j}{2})+T^2+\eta_{3-\ell}T\\
&=\frac{3}{4}\sum_{j=0}^3\wp(z-\frac{\omega_j}{2})+T\sum_{j=0}^3\varepsilon_j^{(\ell)} \zeta(z-\frac{\omega_j}{2})+T\sum_{j=0}^3\varepsilon_j^{(\ell)}\eta_j+T^2+\eta_{3-\ell}T\\
&=q_\ell(z;T).
\end{align*}
Hence, if $y''(z)=q_\ell(z;-T)y(z)$, then $y''(-z)=q_\ell(z;T)y(z)$. 
By the definition of monodromy matrices in (\ref{CompRed}),  we get $$M_k^{(\ell)}(-{T})=M_k^{(\ell)}(T)^{-1}$$ and then $\Delta_k^{(\ell)}(-T)=\Delta_k^{(\ell)}(T)$  since $M_k^{(\ell)}$ has the determinant $1$. 
\end{proof}

We now consider the conditional stability sets for $\tau\in i\mathbb{R}_{>0}$. % and  $\frac{1}{2}+i\mathbb{R}$. 
Recall that 
\begin{align*}
\wp(z; \tau)&=\frac{1}{z^2}+\sum_{(m,n)\neq (0,0)} \left(\frac{1}{(z-m-n\tau)^2}-\frac{1}{(m+n\tau)^2}\right),\\
\zeta(z;\tau)&=\frac{1}{z}+\sum_{(m,n)\neq (0,0)} \left(\frac{1}{z-m-n\tau}+\frac{1}{m+n\tau}+\frac{z}{(m+n\tau)^2}\right).
\end{align*}
If $\tau \in i\mathbb{R}_{>0}$, then $\overline{m+n\tau}=m-n\tau$, so $\wp(\overline{z})=\overline{\wp(z)}$ and  $\zeta(\overline{z})=\overline{\zeta(z)}$. %Similarly, if $\tau\in \frac{1}{2}+i\mathbb{R}$, then $\overline{m+n\tau}=m+n-n\tau$, so $\wp(\overline{ z})=\overline{\wp(z)}$ and  $\zeta(\overline{ z})=\overline{\zeta(z)}$. 

\begin{lemma}\label{lem-sym}
Let $\ell\in \{1,2,3\}$ and $k\in \{1,2\}$.  Let $\tau\in i\mathbb{R}_{>0}$, 
%\begin{enumerate}
%\item [(1)] 
then $$\Delta_k^{(\ell)}(T)=\overline{\Delta_k^{(\ell)}(\overline{T})}.$$
 Consequently, $\sigma_k^{(\ell)}$ is symmetric with respect to the real axis and imaginary axis.
%\item [(2)] If $\tau\in \frac{1}{2}+i\mathbb{R}$, then $\Delta_1^{(1)}(T)=\overline{\Delta_1^{(1)}(\overline{T})}$ and $\Delta_1^{(3)}(T)=\overline{\Delta_1^{(2)}(\overline{T})}$. Consequently, 
%\begin{enumerate}
%\item $\sigma_1^{(1)}$ is symmetric with respect to the real axis and imaginary axis;
%\item $\sigma_1^{(3)}=\overline{\sigma_1^{(2)}}$, i.e., $\sigma_1^{(3)}$ and ${\sigma_1^{(2)}}$ are symmetric with respect to the real axis.
%\end{enumerate}
%\end{enumerate}
\end{lemma}
\begin{proof}
By a direct computation, we have 
\begin{align*}
\overline{q_\ell(\overline{z}; \overline{T})}&=\frac{3}{4}\sum_{j=0}^3\overline{\wp(\overline{z}-\frac{\omega_j}{2})}+{T}\sum_{j=0}^3\overline{\varepsilon_j^{(\ell)} \zeta(\overline{z}-\frac{\omega_j}{2})}+{T}^2-\overline{\eta_{3-\ell}}{T}\\
%&=\frac{3}{4}\sum_{j=0}^3{\wp({z}-\frac{\overline{\omega_j}}{2})}+{T}\sum_{j=0}^3{\varepsilon_j^{(\ell)} \zeta({z}-\frac{\overline{\omega_j}}{2})}+{T}^2-\overline{\eta_{3-\ell}}{T}\\
&=\frac{3}{4}\sum_{j=0}^3{\wp({z}-\frac{{\omega_j}}{2})}+{T}\sum_{j=0}^3{\varepsilon_j^{(\ell)} \zeta({z}-\frac{\overline{\omega_j}}{2})}+{T}^2-\overline{\eta_{3-\ell}}{T}.
%&=q_k(z;T).
\end{align*}
%Denote by $\{k,k_1,k_2\}=\{1,2,3\}$, we have 

If $\tau\in i\mathbb{R}_{>0}$, then $\overline{\tau}=-\tau$, so
\begin{align*}
&\varepsilon_{2}^{(\ell)}\zeta({z}-\frac{\overline{\tau}}{2})+\varepsilon_{3}^{(\ell)}\zeta({z}-\frac{\overline{1+\tau}}{2})\\
=&\varepsilon_{2}^{(\ell)}\zeta({z}+\frac{{\tau}}{2})+\varepsilon_{3}^{(\ell)}\zeta({z}-\frac{{1-\tau}}{2})\\
=&\varepsilon_{2}^{(\ell)}\zeta({z}-\frac{{\tau}}{2})+\varepsilon_{2}^{(\ell)}\eta_2+\varepsilon_{3}^{(\ell)}\zeta({z}-\frac{{1+\tau}}{2})+\varepsilon_{3}^{(\ell)}\eta_2.\\
\end{align*}
Since 
\begin{align*}
\overline{\eta_1}&=\overline{\zeta(z+1)-\zeta(z)}=\zeta(\overline{z}+1)-\zeta(\overline{z})=\eta_1;\\
\overline{\eta_2}&=\overline{\zeta(z+\tau)-\zeta(z)}=\zeta(\overline{z}-\tau)-\zeta(\overline{z})=-\eta_2,
\end{align*}
we have 
\begin{align*}
\overline{q_\ell(\overline{z}; \overline{T})}=q_\ell(z;T)+T\left(\varepsilon_{2}^{(\ell)}\eta_2+\varepsilon_{3}^{(\ell)}\eta_2+\eta_{3-\ell}-\overline{\eta_{3-\ell}}\right)=q_\ell(z;T).
\end{align*}
Hence, if $y''(z)=q_\ell(z;\overline{T})y(z)$, then $\overline{y''(\overline{z})}=q_\ell(z;T)\overline{y(\overline{z})}$. 
By the definition of monodromy matrices in (\ref{CompRed}), we get 
\begin{align*}
M_1^{(\ell)}(\overline{T})&=\overline{M_1^{(\ell)}(T)},\\
M_2^{(\ell)}(\overline{T})^{-1}&=\overline{M_2^{(\ell)}(T)}
\end{align*} and then $\Delta_k^{(\ell)}(\overline{T})=\overline{\Delta_k^{(\ell)}(T)}$ because $M_2^{(\ell)}$ has determinant $1$.

 Since $\Delta_k^{(\ell)}(-{T})={\Delta_k^{(\ell)}(T)}$ and $\Delta_k^{(\ell)}(\overline{T})=\overline{\Delta_k^{(\ell)}(T)}$, then $-T, \overline{T}\in \sigma_k^{(\ell)}$ if $T\in \sigma_k^{(\ell)}$, so $\sigma_k^{(\ell)}$ is symmetric with respect to the real axis and imaginary axis.
\end{proof}

Notice that $\tau\in i\mathbb{R}_{>0}$ gives us that  $e_j(\tau)=\wp({\omega_j}/{2}), j=1,2,3$ are real and satisfy $$e_2(\tau)<e_3(\tau)<e_1(\tau),\quad \sum\limits_{i=1}^3 e_i(\tau)=0.$$ Then we can determine the  conditional stability set  $\sigma_k^{(\ell)}$ precisely. 
\begin{proposition}\label{prop-spec}
    Let $\tau\in i\mathbb{R}_{>0}$. Then   
    \begin{align*}
 \sigma_1^{(1)}=&(-i\infty,-i\sqrt{e_1-e_2}]\cup[-i\sqrt{e_1-e_3},i\sqrt{e_1-e_3}]\cup [i\sqrt{e_1-e_2},i\infty),\\
 \sigma_2^{(1)}=&\mathbb{R}\cup[-i\sqrt{e_1-e_2},-i\sqrt{e_1-e_3}]\cup [i\sqrt{e_1-e_3},i\sqrt{e_1-e_2}],\\
 \sigma_1^{(2)}=&i\mathbb{R}\cup[-\sqrt{e_1-e_2},-\sqrt{e_3-e_2}]\cup [\sqrt{e_3-e_2},\sqrt{e_1-e_2}],\\
 \sigma_2^{(2)}=&(-\infty,-\sqrt{e_1-e_2}]\cup[-\sqrt{e_3-e_2},\sqrt{e_3-e_2}]\cup [\sqrt{e_1-e_2},+\infty),\\
\sigma_1^{(3)}=&(-i\infty,-i\sqrt{e_3-e_2}]\cup[-\sqrt{e_1-e_3},\sqrt{e_1-e_3}]\cup [i\sqrt{e_3-e_2},i\infty),\\
 \sigma_2^{(3)}=&(-\infty,-\sqrt{e_1-e_3}]\cup[-i\sqrt{e_3-e_2},i\sqrt{e_3-e_2}]\cup [\sqrt{e_1-e_3},+\infty).
    \end{align*}
 \end{proposition}
\begin{proof}
Let $k=1,2$. By Lemma \ref{lem-loop}, the  conditional stability set $\sigma_k^{(\ell)}$ does not contain any loop.   By Lemma \ref{lem-sym}, $\sigma_k^{(\ell)}$ is symmetric with respect to the real axis and imaginary axis. 
 Furthermore, by Lemma \ref{lem-endpoints},  $\sigma_1^{(\ell)}$
 contains exactly two unbounded arcs $\sigma_{1\pm}^{(\ell)}$ which asymptotically tend to $\pm i\infty$ and  $\sigma_2^{(\ell)}$ contains exactly two unbounded arcs $\sigma_{2\pm}^{(\ell)}$ which asymptotically tend to $\pm\frac{i}{2\tau}\infty$.  These facts together gives us that $\sigma_{1\pm}^{(\ell)}\subseteq i\mathbb{R}$ and $\sigma_{2\pm}^{(\ell)}\subseteq \mathbb{R}$ . % are contained in the imaginary axis. 

By Proposition \ref{prop-ell},
$$Q_\ell(T)=16(T^2+e_\ell-e_{j})(T^2+e_\ell-e_{k})$$
where $\{\ell, j,k\}=\{1,2,3\}$. 

For $\ell=1$, $Q_1(T)=0$ has four distinct purely imaginary zeros:
$$ \pm i\sqrt{e_1-e_2},\quad \pm i\sqrt{e_1-e_3},$$
which are exactly finite endpoints of $\sigma_\ell^{(1)}$. %and  $\sigma_2^{(1)}$. 
%By Theorem \ref{thm07121705}, all $iT_\ell,\ 1\leq \ell\leq 4$, are end points. Therefore, we have $\sigma^{(-)}=(-i\infty,i\ T_1]$ and $\sigma^{(+)}=[i\ T_{4},+i\infty)$.
Suppose that there is a bounded arc $\gamma\subseteq \sigma_k^{(1)}$ which is not contained in the  imaginary axis. Since  $\sigma_k^{(1)}$ is symmetric with respect to the imaginary axis, then $$-\overline{\gamma}=\{-\overline{T}\mid T\in {\gamma}\}\subseteq\sigma_k^{(1)}.$$  Note that the two endpoints of $\gamma$ lie in the  imaginary axis, then $-\overline{\gamma}\cup{\gamma}$ contains a loop, which is a contradiction. Hence, we conclude  that all bounded  arcs of $\sigma_k^{(1)}$ are contained in the imaginary axis.  Therefore, 
 \begin{align*}
 \sigma_1^{(1)}=&(-i\infty,-i\sqrt{e_1-e_2}]\cup[-i\sqrt{e_1-e_3},i\sqrt{e_1-e_3}]\cup [i\sqrt{e_1-e_2},i\infty),\\
 \sigma_2^{(1)}=&\mathbb{R}\cup[-i\sqrt{e_1-e_2},-i\sqrt{e_1-e_3}]\cup [i\sqrt{e_1-e_3},i\sqrt{e_1-e_2}].
    \end{align*}

For $k=2$, the same proof with the imaginary axis changed by real axis gives us the  conditional stability sets $
\sigma_1^{(2)}$ and $\sigma_2^{(2)}$.
\iffalse 
$\ell_2(T)=0$ has four distinct purely imaginary zeros:
$$ \pm \sqrt{e_1-e_2},\quad \pm \sqrt{e_3-e_2},$$
which are exactly finite endpoints of $\sigma_\ell^{(2)}$. 
Suppose that there is a bounded spectral arc $\gamma\subseteq \sigma_\ell^{(2)}$ which is not contained in the real axis. Since  $\sigma_\ell^{(2)}$ is symmetric with respect to the real axis, then $$\overline{\gamma}=\{\overline{T}\mid T\in {\gamma}\}\subseteq\sigma_\ell^{(1)}.$$  Note that the two endpoints of $\gamma$ lie in the real axis, then $\overline{\gamma}\cup{\gamma}$ contains a loop, which is a contradiction. Hence, we conclude  that all bounded spectral arcs of $\sigma_\ell^{(2)}$ are contained in the real axis.  Therefore, 
 \begin{align*}
\sigma_1^{(2)}=&i\mathbb{R}\cup[-\sqrt{e_1-e_2},-\sqrt{e_3-e_2}]\cup [\sqrt{e_3-e_2},\sqrt{e_1-e_2}],\\
 \sigma_2^{(2)}=&(-\infty,-\sqrt{e_1-e_2}]\cup[-\sqrt{e_3-e_2},\sqrt{e_3-e_2}]\cup [\sqrt{e_1-e_2},+\infty).
    \end{align*}
\fi 

For $k=3$, $Q_3(T)=0$ has four distinct purely imaginary zeros:
$$ \pm \sqrt{e_1-e_3},\quad \pm i\sqrt{e_3-e_2},$$
which are exactly finite endpoints of $\sigma_k^{(3)}$. %and  $\sigma_2^{(1)}$. 
%By Theorem \ref{thm07121705}, all $iT_\ell,\ 1\leq \ell\leq 4$, are end points. Therefore, we have $\sigma^{(-)}=(-i\infty,i\ T_1]$ and $\sigma^{(+)}=[i\ T_{4},+i\infty)$.
%Clearly, $\sigma_{1-}^{(3)}=(-i\infty, -i\sqrt{e_3-e_2}]$ and $\sigma_{1-}^{(3)}=[i\sqrt{e_3-e_2}, +i\infty)$
Suppose that there is a bounded  arc $\gamma\subseteq \sigma_1^{(3)}$ which is not contained in the  real axis. Since  $\sigma_1^{(3)}$ is symmetric with respect to the real axis and the imaginary axis, then $$-\gamma,\overline{\gamma},-\overline{\gamma}\subseteq\sigma_1^{(3)}.$$  Note that the two endpoints of $\gamma$ lie in the axises, then $\gamma\cup(-\gamma)\cup\overline{\gamma}\cup(-\overline{\gamma})$ contains a loop, which is a contradiction. Hence, we conclude  that all bounded arcs of $\sigma_1^{(3)}$ are contained in the real axis. Similarly, all bounded arcs of $\sigma_2^{(3)}$ are contained in the imaginary axis.  So, 
 \begin{align*}
\sigma_1^{(3)}=&(-i\infty,-i\sqrt{e_3-e_2}]\cup[-\sqrt{e_1-e_3},\sqrt{e_1-e_3}]\cup [i\sqrt{e_3-e_2},i\infty),\\
 \sigma_2^{(3)}=&(-\infty,-\sqrt{e_1-e_3}]\cup[-i\sqrt{e_3-e_2},i\sqrt{e_3-e_2}]\cup [\sqrt{e_1-e_3},+\infty).
    \end{align*}
\end{proof}

In fact, $\sigma_1^{(\ell)}$ and $\sigma_2^{(\ell)}$ are related by the modular property of $\wp(z), \zeta(z)$. 

 %Since our potential has another period $\tau$, we can define the Hill discriminant $\Delta_{j,*}(T,\tau)$ as the half of the trace of $M_2$ and arc of stable quasi-eigenvalues $\Delta^{-1}_{j,*}(T,\tau)([-1,1])$, denoted by $\sigma_{j,*}(\tau)$ (or simply $\sigma_{j,\ast}$ if no confusion is arised). Now we want to calculate $\sigma_{j,*}(\tau),\ 1\leq j\leq 3$ when $\tau=ib$, $b>0$.

\begin{proposition}\label{prop-modular}
 Let $\tau\in i\mathbb{R}_{>0}$. Then
    $$\sigma_{2}^{(1)}(\tau)=\tau \sigma_1^{(2)}(-\frac{1}{\tau}),\quad \sigma_2^{(2)}(\tau)=\tau\sigma_1^{(1)}(-\frac{1}{\tau}), \quad \sigma_2^{(3)}(\tau)=\tau\sigma_1^{(3)}(-\frac{1}{\tau}).$$
\end{proposition}
\begin{proof}
For any $\left(\begin{array}{cc}a&b\\ c&d\end{array}\right)\in \mathrm{SL}(2,\mathbb{Z})$, the modular property
\begin{align*}
    \wp\left(\frac{z}{c\tau+d};\frac{a\tau+b}{c\tau+d}\right)&=(c\tau+d)^2\wp(z;\tau),\\
\zeta\left(\frac{z}{c\tau+d};\frac{a\tau+b}{c\tau+d}\right)&=(c\tau+d)\zeta(z;\tau),\\
\left(\begin{array}{c}\eta_2(\frac{a\tau+b}{c\tau+d})\\\eta_1(\frac{a\tau+b}{c\tau+d})\end{array}\right)&=(c\tau+d)\left(\begin{array}{cc}a&b\\ c&d\end{array}\right)\left(\begin{array}{c}\eta_2(\tau)\\\eta_1(\tau)\end{array}\right)
\end{align*}
gives us $\eta_1(\hat\tau)=\tau\eta_2, \,\eta_2(\hat\tau)=-\tau\eta_1$ and 
\begin{align*}
  \wp\left(\frac{z}{-\tau};\hat \tau\right)=\tau^2\wp(z;\tau),  \quad  \zeta\left(\frac{z}{-\tau};\hat \tau \right)=-\tau\zeta(z;\tau).
\end{align*}
where $\hat\tau=-{1}/{\tau}$. 
If $y(z)$ satisfies 
 $$y''(z)=q(z;T_0, T_1, T_2, T_3, B, \tau)y(z),$$  then 
\begin{align*}
   % \hat{y}''(z)&=\tau^2
\frac{d^2}{dz^2}y(\tau z)&=\tau^2
\left(\dfrac{3}{4}\sum\limits_{j=0}^3\wp(\tau z-\frac{\omega_j}{2};\tau)+\sum\limits_{j=0}^3T_j\zeta(\tau z-\frac{\omega_j}{2};\tau)+B\right)y(\tau z)\\
&=
  \left(  \dfrac{3}{4}\sum\limits_{j=0}^3\wp(z-\frac{\hat{\omega}_j}{2};\hat\tau)+\sum\limits_{j=0}^3\hat{T}_j\zeta(t-\frac{\hat{\omega}_j}{2};\hat\tau)+\hat{B}\right)y(\tau z),
\end{align*}
where $\hat{\omega}_0=0, \hat{\omega}_1=\hat\tau, \hat{\omega}_2=1, \hat{\omega}_3=1+\hat\tau, \hat{T}_j=\tau T_j, j=0,1,2,3$,  and $\hat{B}=\tau^2B+(T_1+T_3)\tau\eta_2(\hat\tau)$. 
Hence, $y(\tau z)$ is a solution of
$$y''(z)=q(z;\hat{T}_0,\hat{T}_2,\hat{T}_1,\hat{T}_3,\hat B,\hat\tau)y(z).$$
Note that 
\begin{itemize}
\item $(T_0,T_1,T_2,T_3,B)\in V_1(\tau)$ implies
 $(\hat{T}_0,\hat{T}_2,\hat{T}_1,\hat{T}_3,\hat{B})\in V_2(\hat\tau)$;  
\item 
$(T_0,T_1,T_2,T_3,B)\in V_2(\tau)$  implies $(\hat{T}_0,\hat{T}_2,\hat{T}_1,\hat{T}_3,\hat{B})\in V_1(\hat\tau)$; 
\item 
$(T_0,T_1,T_2,T_3,B)\in V_3(\tau)$  if and only if  $(\hat{T}_0,\hat{T}_2,\hat{T}_1,\hat{T}_3,\hat{B})\in V_3(\hat\tau)$,
\end{itemize}
that is, $T\in V_1(\tau)$ implies $\tau T\in V_2(\hat \tau)$,  $T\in V_2(\tau)$ implies $\tau T\in V_1(\hat \tau)$, and  $T\in V_3(\tau)$ if and only if $\tau T\in V_3(\hat \tau)$.
By the definition of monodromy matrices in (\ref{CompRed}) and (\ref{CompNRed}),  we get 
\begin{align*}
&M_2^{(1)}(T;\tau)=M_1^{(2)}(\tau T;\hat\tau), &\Delta_2^{(1)}(T;\tau)=\Delta_1^{(2)}(\tau T;\hat\tau),\\
&M_2^{(2)}(T;\tau)=M_1^{(1)}(\tau T;\hat\tau),\quad \text{and then}&\Delta_2^{(2)}(T;\tau)=\Delta_1^{(1)}(\tau T;\hat\tau), \\
&M_2^{(3)}(T;\tau)=M_1^{(3)}(\tau T;\hat\tau), &\Delta_2^{(3)}(T;\tau)=\Delta_1^{(3)}(\tau T;\hat\tau).
%M_2^{(2)}(\tau T)=M_1^{(1)}(T)^{-1},
% M_2^{(3)}(\tau T)=M_2^{(3)}(T)^{-1},
\end{align*}  
Therefore,   $$\sigma_{2}^{(1)}(\tau)=\tau \sigma_1^{(2)}(\hat\tau),\quad \sigma_2^{(2)}(\tau)=\tau\sigma_1^{(1)}(\hat\tau), \quad \sigma_2^{(3)}(\tau)=\tau\sigma_1^{(3)}(\hat\tau).$$
\end{proof}

\iffalse 
Clearly, (\ref{cc1}) with $({\bf T},B)$ has unitarizable monodromy matrices iff $({\bf T},B)\in\sigma\cap\sigma_\ast$ and $({\bf T},B)$ is not an end point of $\sigma$ and $\sigma_\ast$. Hence we apply Corollary \ref{cor07121714}, that is, $\ell(T)=0$ iff (\ref{cc1}) is not completely reducible. In our case, $\sigma(\tau)=\bigcup\limits_{j=1}^3\sigma_j(\tau)$, $\sigma_\ast(\tau)=\bigcup\limits_{j=1}^3\sigma_{j,\ast}(\tau)$ and  $\sigma(\tau)\cap\sigma_\ast(\tau)=\bigcup\limits_{j=1}^3(\sigma_j(\tau)\cap\sigma_j^\ast(\tau))$. Note that $({\bf T},B)=(0,0,0,0,0)\in\sigma_i(\tau)\cap\sigma_{j,\ast}(\tau)$ for any $i,j$. So, we want to ask when $\sigma_j(\tau)\cap\sigma_{j,\ast}(\tau)$ consists only of end points and $(0,0,0,0,0)$. The next result answers it affirmatively.
\fi 

Now, we are ready to determine when the monodromy matrices are unitarizable for $\tau\in i\mathbb{R}_{>0}$. 

\begin{theorem}
Let $\tau\in i\mathbb{R}_{>0}$. Then $\mathcal{L}(T)$ with  $T\in AP$ has unitarizable monodromy matrices if and only if $T=0$. %, that is, $q(z)=\frac{3}{4}\sum\limits_{k=0}^3\wp(z-\omega_k/2)$.
\end{theorem}
\begin{proof}
    The sufficient part is obviously true. To prove the necessary part, let $T\in V_\ell$ for some $\ell\in \{1,2,3\}$, we first show that $\mathcal{L}_\ell(T)$ has unitarizable monodromy matrices if and only if $T\in\sigma_1^{(\ell)}\cap\sigma_2^{(\ell)}$   and $T$ is not an end point of $\sigma_1^{(\ell)}$ and $\sigma_2^{(k)}$, i.e., $Q_\ell(T)\neq 0$.  Indeed,  $\mathcal{L}_\ell(T)$ has unitarizable monodromy matrices if and only if $\mathcal{L}_\ell(T)$ is completely reducible and the eigenvalues have norm $1$. By Proposition \ref{prop-poly}, 
we have  $Q_\ell(T)\neq0$ if and only if  $\mathcal{L}_\ell(T)$ is completely reducible. Note that the eigenvalues of the monodromy matrices have norm $1$ if and only if the monodromy data $s,r\in \mathbb{R}$, which is equivalent to $T\in\sigma_1^{(\ell)}\cap\sigma_2^{(\ell)}$, because %  By (\ref{eqn-Delta})
$$
\Delta_1^{(\ell)}(T)=\cos (2\pi s),\quad 
\Delta_2^{(\ell)}(T)=\cos (2\pi r).
$$
By Proposition \ref{prop-spec}, it is clear to see that 
$$\sigma_1^{(\ell)}\cap\sigma_2^{(\ell)}=\{0\}\cup \{T\,|\, Q_\ell(T)=0\},$$
which is desired. 
\end{proof}

 \noindent{\bf Data Availability} No data was used for the research described in the article.

\noindent{\bf Conflict of interest} The authors have no conflicts to disclose.

%\subsection*{Acknowledgements} The authors thank A. Eremenko for his interest and comments. Z. Chen was supported by National Key R\&D Program of China (No. 2023YFA1010002) and NSFC (No. 12222109).  E. Fu was supported by NSFC (No. 12401188) and BIMSA Start-up Research Fund.

\end{document}